\documentclass[leqno]{article}
\usepackage{arxiv}

\usepackage{graphicx}
\usepackage{enumerate}
\usepackage{amsmath, amssymb, amsthm, amscd}
\usepackage{mathtools}
\usepackage[dvipsnames]{xcolor}
\usepackage{caption}
\usepackage{subcaption}
\usepackage{multirow}
\usepackage{extarrows}

\usepackage{multicol}
\usepackage{stmaryrd}
\usepackage{enumitem}
\usepackage[pdfencoding=auto, psdextra, hidelinks]{hyperref}
\numberwithin{equation}{section}
\usepackage{algorithm}
\usepackage{algpseudocode}

\newtheorem{theorem}{Theorem}[section]
\newtheorem{lemma}{Lemma}[section]
\newtheorem{corollary}{Corollary}[section]

\newcommand{\yb}{\bar{\y}}
\newcommand{\rb}{\bar{r}}
\newcommand{\sbar}{\bar{s}}
\newcommand{\pb}{\bar{\pee}}
\newcommand{\ub}{\bar{\U}}
\newcommand{\yhb}{\bar{\y}_h}
\newcommand{\rhb}{\bar{r}_h}
\newcommand{\shbar}{\bar{s}_h}
\newcommand{\phb}{\bar{\pee}_h}
\newcommand{\uhb}{\bar{\U}_h}

\newcommand{\Om}{\Omega}
\newcommand{\D}{\Delta}
\newcommand{\n}{\nabla}
\newcommand{\y}{\textbf{y}}
\newcommand{\pee}{\textbf{p}}

\newcommand{\pw}{\rm{pw}}

\newcommand{\ls}{[L^2(\Omega)]^2}

\newcommand{\h}{\textbf{H}_0^1(\Omega)}
\newcommand{\U}{{\bf u}}

\newcommand{\csol}{(\bar{\bf y},\bar{r},\bar{\bf p},\bar{s},\bar{\U})}
\newcommand{\dsol}{(\bar{\bf y}_h,\bar{r}_h,\bar{\bf p}_h,\bar{s}_h,\bar{\U}_h)}

\newcommand{\1}{\|\widehat{\bar{\U}}_h-\bar{\U}_h\|}
\newcommand{\2}{\enorm{\widehat{\bar{\y}}_h-\bar{\y}_h}_{\pw}}
\newcommand{\w}{\|\widehat{\bar{\y}}_h-\bar{\y}_h\|}
\newcommand{\3}{\enorm{\widehat{\bar{\pee}}_h-\bar{\pee}_h}_{\pw}}
\newcommand{\x}{\|\widehat{\bar{\pee}}_h-\bar{\pee}_h\|}
\newcommand{\4}{\|\widehat{\bar{r}}_h-\bar{r}_h\|}
\newcommand{\5}{\|\widehat{\bar{s}}_h-\bar{s}_h\|}
\newcommand{\6}{\enorm{\widehat{\widetilde{\y}}_h-\bar{\y}_h}_{\pw}}
\newcommand{\q}{\|\widehat{\widetilde{\y}}_h-\bar{\y}_h\|}
\newcommand{\7}{\enorm{\widehat{\widetilde{\pee}}_h-\bar{\pee}_h}_{\pw}}
\newcommand{\z}{\|\widehat{\widetilde{\pee}}_h-\bar{\pee}_h\|}
\newcommand{\8}{\|\widehat{\widetilde{r}}_h-\bar{r}_h\|}
\newcommand{\9}{\|\widehat{\widetilde{s}}_h-\bar{s}_h\|}

\newcommand{\cerr}{\|\bar{\U}-\bar{\U}_h\|^2}

\newcommand{\wcerr}{\|\bar{\U}-\bar{\U}_h\|}
\newcommand{\wsterr}{\enorm{\bar{\y}-\bar{\y}_h}_{\rm \pw}}
\newcommand{\wtum}{\|\bar{r}-\bar{r}_h\|}
\newcommand{\waderr}{\enorm{\bar{\pee}-\bar{\pee}_h}_{\rm \pw}}
\newcommand{\wmai}{\|\bar{s}-\bar{s}_h\|}

\newcommand{\wlsterr}{\|\bar{\y}-\bar{\y}_h\|}
\newcommand{\wladerr}{\|\bar{\pee}-\bar{\pee}_h\|}
\newcommand{\wlauxsterr}{\|\widetilde{\y}-\bar{\y}_h\|}
\newcommand{\wlauxaderr}{\|\widetilde{\pee}-\bar{\pee}_h\|}
\newcommand{\wTUM}{\|\widetilde{r}-\bar{r}_h\|}
\newcommand{\wMAI}{\|\widetilde{s}-\bar{s}_h\|}
\newcommand{\wauxsterr}{\enorm{\widetilde{\y}-\bar{\y}_h}_{\rm \pw}}
\newcommand{\wauxaderr}{\enorm{\widetilde{\pee}-\bar{\pee}_h}_{\rm \pw}}

\newcommand{\normone}{\|\widetilde{\bf{u}}-\uhb\|}
\newcommand{\normtwo}{\|\widetilde{\bf{u}}_h-\uhb\|}

\newcommand{\enorm}[1]{{\vert\kern-0.25ex\vert\kern-0.25ex\vert #1 	\vert\kern-0.25ex \vert\kern-0.25ex \vert}}

\newcommand{\T}{\mathcal{T}}
\newcommand{\E}{\mathcal{E}}
\title{Quasi-Optimality of AFEM for Distributed Optimal Control Problems of Stokes Equations: An Axiomatic Framework}

\hypersetup{
    colorlinks=true, 
    linkcolor=blue, 
    urlcolor=blue, 
    citecolor=magenta,
    linktoc=all 
}

\author{
\makebox[\textwidth][c]{%
\parbox{0.3\textwidth}{\centering
\href{https://orcid.org/0009-0004-3203-4018}{\includegraphics[scale=0.06]{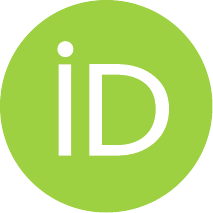}\hspace{1mm}Tooba M. Shaikh}\\
School of Mathematics\\
IISER TVM\\
Kerala, India \\
\texttt{toobashaikh20@iisertvm.ac.in}
}
\hfill
\parbox{0.3\textwidth}{\centering
\href{https://orcid.org/0000-0003-1663-334X}{\includegraphics[scale=0.06]{orcid.pdf}\hspace{1mm}Asha K. Dond  
\thanks{Corresponding author (Asha K. Dond).}}\\
School of Mathematics\\
IISER TVM\\
Kerala, India Email: \texttt{ashadond@iisertvm.ac.in}
}
}
}

\begin{document}
\maketitle
\begin{abstract} 
	\noindent This paper focuses on the quasi-optimality of an adaptive nonconforming finite element method for a distributed optimal control problem governed by the Stokes equation. The nonconforming lowest order Crouzeix-Raviart element and piecewise constant spaces are used to discretise the velocity and pressure variables, respectively. The control variable is discretised using both variational and discretised approach. The error equivalence results at both continuous and discrete levels, leading to a priori and a posteriori error estimates are presented under minimal regularity assumption on optimal solutions. The quasi-optimal convergence rates of the adaptive algorithm are established based on a general axiomatic framework that includes stability, reduction, discrete reliability, and quasi-orthogonality.	The theoretical findings are validated through numerical experiments on convex as well as nonconvex domains.	
\end{abstract}

\section{Introduction} 
\label{introduction}
The aim of this paper is to study the quasi-optimality of an adaptive nonconforming finite element method for a distributed optimal control problem governed by the Stokes equation. Such problems are fundamental in the regulation of viscous incompressible flows and are computationally demanding because the optimality system couples the state, adjoint, and control variables. Since these variables may exhibit localized features in different parts of the domain, adaptive mesh refinement is crucial to resolve them accurately with minimal computational cost. Proving quasi-optimality of the adaptive finite element method ensures the efficient use of degrees of freedom.


\subsection{Optimal control problem}
Let $\Om$ be a bounded domain in $\mathbb{R}^2$
with the boundary $\partial \Omega$. Consider the optimal control problem defined by
\begin{equation}
	\label{min1}
	\min_{\U\in {\bf U_{\rm ad}}} J(\y,\U)=\frac{1}{2}\int_{\Om} (\y-\y_d)^2 dx+\frac{\alpha}{2}\int_{\Om}\U^2 dx
\end{equation}
subject to
\begin{equation}
	\label{mod_pr}
	\begin{cases}
		-\D\y + \n r=\textbf{f}+\U &\mbox{in}\hspace{0.1cm}\Om,\\
		\quad\quad\hspace{0.25cm} \mbox{div} \y =0 &\mbox{in}\hspace{0.1cm}\Om,\\
		\quad\quad\quad\hspace{0.25cm} \y=0 &\mbox{on}\hspace{0.1cm}\partial \Om,\\
	\end{cases}
\end{equation}
\noindent where ${\bf f}\in \ls:={\bf L}^2(\Om)$ is a load function, $\y_d\in\textbf{L}^2(\Om)$ is a given desired state, and $\alpha>0$ is a fixed regularization parameter.  The admissible set for control variable is denoted by ${\bf U}_{\rm ad}$ and for any $\U_a,\U_b\in\mathbb{R}^2$ it is defined as 
\begin{center}
	${\bf U}_{\rm ad} :=\{\U\in\textbf{L}^2(\Om) : \U_a\leq \U(x)\leq \U_b\hspace{0.1cm}\mbox{a.e. in}\hspace{0.1cm}\Om \}$.
\end{center}
The inequality is to be interpreted component wise.

\subsection{Literature review}
The error analysis of optimal control problems (OCPs) governed by the Stokes equation has received significant attention in the literature. In \cite{DAGTSR19}, an abstract framework was developed for the error analysis of discontinuous finite element methods (FEMs) for distributed and Neumann boundary control problems with control constraints. The authors derived optimal a priori estimates for velocity and pressure in energy and ${\bf L}^2(\Om)$ norms, respectively. In \cite{RV06}, OCPs with pointwise control constraints for the Stokes problem in 2D and 3D were analyzed using piecewise constant control discretization, along with a post-processing strategy to enhance the accuracy of a priori estimates. These results were extended in \cite{MR2803865} to conforming and nonconforming FEMs under weaker regularity assumptions, focusing on superconvergence in weighted Sobolev spaces. Recently, \cite{leykekhman2025priori} studied a state-constrained OCP governed by the transient Stokes equation using an inf-sup stable spatial discretization and a discontinuous Galerkin method in time. In \cite{MR4304887}, a priori estimates were derived for control-constrained OCPs involving the Stokes system with Dirac measures. Finally, \cite{MR2237883} presented global superconvergence results for distributed OCPs using patch recovery techniques for the control variable, and post-processing methods for the state and adjoint variables in bilinear-constrained schemes on uniform rectangular meshes.\vspace{0.2cm}\\
A~posteriori error estimates for optimal control problems governed by the Stokes equation have been extensively studied. In \cite{DAGTSR19}, the authors derived efficient and reliable error estimators for problems with distributed and Neumann boundary control constraints, applicable under minimal regularity assumptions. The work in \cite{MR1950625} presents a posteriori estimates for both state and adjoint variables in ${\bf L}^2(\Om)$ and energy norms. In \cite{MR2387126}, recovery-type superconvergence analysis is performed using patch recovery and least-squares surface fitting techniques. More recently, virtual element methods (VEMs) have emerged as a promising approach for adaptive techniques in distributed OCPs. For instance, \cite{MR4661255} investigates adaptive VEM for control-constrained Stokes problems, employing polynomial projection and variational discretization of the control variable. Further developments on VEM for Stokes OCPs are discussed in \cite{MR4753595} and the references therein.\vspace{0.2cm}\\
The quasi-optimal behavior of AFEM implies that the adaptive refinement strategy constructs meshes that achieve, the optimal rate of convergence for the error, so that the discrete solutions approximate the exact solution as efficiently as theoretically possible.
The convergence and quasi-optimality of AFEMs for the Stokes equations have been widely explored. In \cite{CCPDR13}, the authors proposed a nonconforming AFEM algorithm and demonstrated its quasi-optimal complexity using a novel discrete Helmholtz decomposition, without explicitly involving the pressure variable. Similarly, \cite{HJXJ13} studied the convergence and quasi-optimality of standard adaptive nonconforming linear elements for the Stokes problem. Further contributions for convergence and quasi-optimality of AFEMs for OCPs can be found in \cite{BRMS11, CCHR2006, SYYNZZ19, GWYN2017, LHCY18} and the references therein. \vspace{0.2cm}\\
{More recently, the work in \cite{YWN2022} investigated the convergence and quasi-optimality of AFEMs for Stokes optimal control problems with distributed control and pointwise constraints. Their approach employed the lowest order Crouzeix-Raviart space for velocity, piecewise constant functions for pressure, and the control variable is treated within the variational formulation. While the error equivalence result derived in \cite{YWN2022} bears similarity to the one presented in the current article, this work advances the analysis by establishing error equivalence results at both the continuous and discrete levels (see Theorem \ref{v error equivalence} and \ref{d error equivalence}) for both variational and discretised control approaches without any smallness assumption on initial mesh size. 
}

\subsection{Axioms of adaptivity}
The quasi-optimality of AFEMs guarantees that the algorithm generates a sequence of adaptive meshes on which the approximate solution converges to the exact solution at optimal rates. The framework of {\it Axioms of adaptivity} in \cite{CMMD2014} explains the convergence and quasi-optimality of AFEMs using the axiomatic framework. This framework explains the four axioms of adaptivity which includes {\bf (A1)}-{\it Stability}, {\bf (A2)}-{\it Reduction}, {\bf (A3)}-{\it Discrete Reliability}, and {\bf (A4)}-{\it Quasi-orthogonality}. One can refer to \cite{BPCCSG18, CCDARH19, CCRH17, LY2021} for their further applications. The axioms approach establishes the optimal convergence of the adaptive estimator, which is equivalent to the underlying discretization error. It replaces a massive, PDE-dependent optimality proof by the verification of four structural axioms, from which the quasi-optimality of AFEM follows by a standard argument. Consequently, one can avoid repeating long, technical, and by now well-known optimality proofs. This methodology is widely used in the literature \cite{MR3757107,MR3935887,MR4766712} and the references cited therein.

\noindent {\it To the best of our knowledge, the convergence and quasi-optimality of AFEMs for Stokes-governed OCPs have not yet been studied within an  axiomatic framework. This article addresses this gap by establishing convergence and quasi-optimality using the axiomatic approach developed in \cite{CMMD2014}, covering both variational and discretised control schemes.}

\subsection{Main contribution}The present work investigates the convergence and quasi-optimality of AFEMs for OCPs governed by the steady-state Stokes equations, within the axiomatic framework. A key feature of the analysis is that it is carried out for both the variational control formulation and the discretised control formulation, where the latter employs piecewise constant functions for approximating the control variable. In state and adjoint equations, the velocity and pressure variables are discretised using the lowest order Crouzeix-Raviart finite element space and the space of piecewise constant functions, respectively.\vspace{0.2cm}\\
A central contribution of this article is the derivation of novel error equivalence results for both control approaches, achieved without imposing any smallness condition on the mesh size. These results are established by introducing auxiliary problems in each case. In the discretised control setting, since the discrete control variable does not belong to the admissible control set ${\bf U}_{\rm ad}$,  an auxiliary control variable $\widetilde{\bf u}$ is introduced. As a result, the error equivalence for the discretised control formulation includes an additional term capturing the discrepancy between the auxiliary and the discrete control variables. To estimate this discrepancy, an additional control error estimator is defined.\vspace{0.2cm}\\
These error equivalence results significantly simplify the convergence and optimality analysis of the OCP by reducing it to the corresponding analysis for the Stokes problem, for which extensive theoretical results are available in the literature. A notable distinction is that the a priori and a posteriori error analysis for the OCP, along with error equivalence results established in this article does not rely on any mesh size restrictions, either at the initial level or throughout the refinement process, for both variational and discretised control approach. In addition, the explicit constants involved in the analysis are carefully tracked and presented to demonstrate their independence from the mesh size, thereby enhancing the comprehensibility of the results.\vspace{0.2cm}\\
In the process of verifying the axioms of adaptivity, the formulation of an appropriate distance function plays a crucial role for both the variational and discretised control approaches. For the variational control approach, the control variable can typically be estimated in terms of the velocity component of the adjoint problem, facilitated by the Lipschitz continuity of $\Pi_{[a,b]}(g):=\min\{b,\max\{a,g\}\}$ operator. However, such an estimation is not directly feasible in the discretised control setting. To address this challenge, the definition of the distance function for the discretised approach is augmented with an additional term that captures the control error between the coarse and refined triangulations. This modification ensures the consistency of the adaptive framework across both formulations. The main contributions of this work are outlined as follows :
\begin{itemize}
    \item Error equivalence results are established at both the continuous and discrete levels, measured in both the energy and $L^2(\Om)$ norms,  for the variational and discretised control formulations without any smallness assumption on mesh size.
    \item  Derivations of a priori and a posteriori error estimates are presented for each control strategy and all the constants are explicitly tracked.
    \item Detailed convergence and quasi-optimality analyses are conducted for both the variational and discretised control settings within the axiomatic framework with minimal regularity assumption.
    \item A series of numerical experiments on convex and non-convex domains are provided to validate the theoretical findings on both variational and discretised control formulations.
\end{itemize}
\subsection{Outline of the paper}
The paper is organized into eight sections. Section \ref{introduction} introduces the model problem and reviews the related literature. In Section \ref{main result}, we present the main results of the paper, including the error equivalence results for both variational and discretised control formulations, and outline the adaptivity axioms used in the analysis. Section \ref{proof of error equi} provides detailed proofs of these error equivalence results for both control approaches and discusses the explicit constants appearing in the analysis. In Section \ref{error control}, we derive the a priori and a posteriori error estimates for the model problem based on the error equivalence framework. The adaptivity axioms for the variational and discretised control formulations are verified in Sections \ref{verification of variational axioms} and \ref{verification of discretised axioms}, respectively. Numerical experiments on convex and nonconvex domains are presented in Section \ref{numerical experiments} to illustrate and support the theoretical results. Finally, Section \ref{conclusions} concludes the paper with a summary of the main findings. Additionally an appendix is included in Section \ref{appendix}. 

\subsection{Notation}
\subsubsection{General notations.} Standard notations for Sobolev spaces are employed throughout. The norm on \( L^2(\Omega) \) is denoted by \( \|\cdot\| \), while \( \|\cdot\|_m \) and \( |\cdot|_m \) denote the norm and semi norm on \( H^m(\Omega) \), respectively. For a subregion \( R \subset \mathbb{R}^2 \), the corresponding localized notations \( \|\cdot\|_R \), \( \|\cdot\|_{m,R} \), and \( |\cdot|_{m,R} \) are used. The inner product on \( L^2(\Omega) \) is denoted by \( (\cdot,\cdot) \). The symbol \( a \lesssim b \) indicates the inequality \( a \leq Cb \) for a generic constant \( C > 0 \) independent of the quantities \( a \) and \( b \). The notation \( a \approx b \) stands for equivalence, i.e., \( a \lesssim b \lesssim a \). The weighted Young's inequality is given by
\[
ab \leq \frac{a^2}{2\epsilon} + \frac{\epsilon b^2}{2}, \quad \text{for all } a, b, \epsilon > 0.
\]
For vectors \( \mathbf{a} = (a_1, a_2) \) and \( \mathbf{b} = (b_1, b_2) \) in \( \mathbb{R}^2 \), the notation \( \mathbf{a} \leq \mathbf{b} \) is understood component wise, i.e., \( a_i \leq b_i \) for \( i = 1, 2 \).
\subsubsection{Discrete level notations.}
Let \( \mathcal{T} \) be an admissible and shape-regular triangulation of the polygonal domain \( \Omega \subset \mathbb{R}^2 \). Let \( \widehat{\mathcal{T}} \) denotes a refinement of \( \mathcal{T} \). For each element \( K \in \mathcal{T} \), the local mesh size is defined as \( h_K := |K|^{1/2} \approx \operatorname{diam}(K) \), where \( |K| \) denotes the area of \( K \). The global mesh size is given by \( h := \max_{K \in \mathcal{T}} h_K \).
For each element \( K \in \mathcal{T} \), denote by \( \mathcal{E}_K \) the set of all edges of \( K \), and let \( \omega(K) \) be the patch of elements sharing at least one edge with \( K \), i.e., 
\[
\omega(K) := \{ K' \in \mathcal{T} : K' \text{ shares an edge with } K \}.
\]
The set of all boundary edges is denoted by \( \mathcal{E}_{\Gamma} \), and the collection of all edges in the triangulation \( \mathcal{T} \) is represented by \( \mathcal{E} \). Let $n_E$ and $t_E$ denote the unit normal and tangential vectors, respectively, across an interior edge $E$. For any vector-valued function ${\bf q}$, the normal jump across $E$ is denoted by $[{\bf q}]_E\cdot n_E$ and is defined as $[{\bf q}]_E\cdot n_E = \big({\bf q}|_{K_+}-{\bf q}|_{K_-}\big)\cdot n_E.$
Similarly, the tangential jump across $E$ is denoted by $[{\bf q}]_E\cdot t_E$ and is given by $[{\bf q}]_E\cdot t_E = \big({\bf q}|_{K_+}-{\bf q}|_{K_-}\big)\cdot t_E.$
Here, $K_+$ and $K_-$ are the two triangles sharing the common edge $E$.

Let \( \mathbb{T}(\mathcal{T}) \) be the set of all admissible refinements of \( \mathcal{T} \) obtained using the newest vertex bisection (NVB) method. Let $\mathbb{T}(\delta)$ denotes the collection of all shape regular triangulation with $h\leq \delta$ for $\delta\in (0,1]$. The \( k \)-th refinement of an initial triangulation \( \mathcal{T}_0 \) is denoted by \( \mathcal{T}_k \). For any \( N \in \mathbb{N} \), define the collection of all refinements with at most \( N \) additional elements as
\[
\mathbb{T}(\mathcal{T}_0, N) := \left\{ \mathcal{T} \in \mathbb{T}(\mathcal{T}_0) : \#\mathcal{T} \leq N + \#\mathcal{T}_0 \right\}.
\]
The space of piecewise polynomials of degree at most \( k \in \mathbb{N} \) over \( \mathcal{T} \) is denoted by \( \mathcal{P}_k(\mathcal{T}) \), and the corresponding vector-valued space by \( \boldsymbol{\mathcal{P}}_k(\mathcal{T}) := \mathcal{P}_k(\mathcal{T}) \times \mathcal{P}_k(\mathcal{T}) \). The \( L^2(\Omega) \)-orthogonal projection onto \( \mathcal{P}_0(\mathcal{T}) \) is denoted by \( \Pi_0\). Let $\widehat{\Pi}_0$ be the \( L^2(\Omega) \)-orthogonal projection onto \( \mathcal{P}_0(\widehat{\mathcal{T}}) \).
A generic positive constant \( C \) is used throughout to denote quantities independent of the mesh size parameter \( h \).


\section{Main result}
\label{main result}
\noindent The first subsection presents the continuous and discrete optimality systems corresponding to the variational and discretised control approaches. The second subsection presents the error equivalence results along with the corresponding auxiliary problems. Finally, the third subsection states the adaptivity axioms required for the AFEM to achieve optimal convergence rates.

\subsection{Continuous and discrete formulations} Let ${\bf V}:=H_0^1(\Om)\times H_0^1(\Om):=\h$ and $Q:=L_0^2(\Om)=\{ q\in L^2(\Om) : \int_{\Om} q dx=0\}$. 
The weak formulation corresponding to \eqref{min1}-\eqref{mod_pr} is to find $({\bf y},r,{\bf u})\in {\bf V}\times Q\times {\bf U}_{\rm ad}$ such that
\begin{equation}
\label{nfunc}
	\min_{\U\in {\bf U_{\rm ad}}} J(\y,\U)=\frac{1}{2}\|\y-\y_d\|^2+\frac{\alpha}{2}\|\U\|^2, \quad {\rm subject~ to}
    \end{equation} 
\begin{equation}
\begin{aligned}
\label{nst}
    a(\y,\textbf{v})-b(\textbf{v},r)&=(\bf{f}+\U,\textbf{v}) \quad{\rm for~all~} \textbf{v}\in V,\\
    b(\y,q)&=0 \quad {\rm for~all~} q\in Q.
\end{aligned}
\end{equation}
Let ${\bf w}=(w_1,w_2)$ and ${\bf v}=(v_1,v_2)$. For any $({\bf w},{\bf v})\in {\bf V}\times {\bf V}$ and $r\in Q$, bilinear forms $a(\cdot,\cdot)$ and $b(\cdot,\cdot)$ are defined as  
\begin{equation*}
	a({\bf w},\textbf{v}):=\int_{\Om}\n{\bf w} :\n\textbf{v}\,dx=\sum_{i=1}^{2}\int_{\Om}\n w_i \cdot \n v_i\,dx\mbox{ and } b(\textbf{v},r) := \int_{\Om} r\,\mbox{div}\textbf{v}\,dx.
\end{equation*}
\noindent The bilinear form $a$ is bounded and ${\bf V}$-elliptic, and the bilinear form $b(\cdot,\cdot)$ satisfies Ladyzhenskaya-Babu\v{s}ka-Brezzi (LBB) condition \cite[Theorem 6.3, Theorem 6.4]{Braess}. Therefore, the well-posedness of  the state equation follows from \cite[Theorem 4.1]{MR0548867}. As the cost functional \eqref{nfunc} is convex, and the state equation is well posed, the standard theory of optimal control problems \cite[Theorem 2.14]{TF2010} implies the well posedness of \eqref{nfunc}-\eqref{nst}. The optimality system seeks $(\yb,\rb,\pb,\sbar,\ub)\in{\bf V}\times Q\times {\bf V}\times Q\times {\bf U}_{\rm ad}$ such that for all $({\bf v},q)\in {\bf V}\times Q$ the following equation holds.
\vspace{-1cm}
\begin{multicols}{2}
\begin{equation}
\label{cstate}
	\begin{cases}
		a(\bar{\y},\textbf{v})-b(\textbf{v},\bar{r})=(\mathbf{f}+\bar{\U},\textbf{v}),\\
		\quad\quad\quad\hspace{0.5cm} b(\bar{\y},q)=0, & 
	\end{cases}
\end{equation}\break 
\begin{equation}
\label{cadjoint}
	\begin{cases}	a(\bar{\pee},\textbf{v})+b(\textbf{v},\bar{s})=(\bar{\y}-\y_d,\textbf{v}), & \\
		\quad\quad\quad\hspace{0.3cm} b(\bar{\pee},q)=0.& 
	\end{cases}
\end{equation}
\end{multicols}
\noindent
The optimality condition is at a continuous level reads 
\begin{equation}
\label{cop}
	(\alpha\bar{\U}+\bar{\pee},\textbf{z}-\bar{\U})\geq 0\hspace{0.2cm}\mbox{ for all } \textbf{z}\in {\bf U}_{\rm ad}.
\end{equation}

\noindent The optimal control $\bar{\U}$ has  a representation in  terms of the adjoint variable $\bar{\pee}$ and the projection operator $\Pi_{[\U_a,\U_b]}$ defined by $\Pi_{[a,b]}(g):=\min\{b,\max\{a,g\}\}$ as
\begin{equation}
	\label{defu}
	\bar{\U}(x)=\Pi_{[\U_a,\U_b]}\left(-\frac{1}{\alpha}\bar{\pee}(x)\right).
\end{equation}
For the discretization of the velocity variables $({\bf y},{\bf p})$ and the pressure variables $(r,s)$, we employ the nonconforming lowest-order Crouzeix-Raviart finite element space and the space of piecewise constant polynomials intersected with $Q$, respectively.
 The lowest order Crouzeix-Raviart space is defined as
\begin{center}
	${\bf V}_h = \{ {\bf v}_h\in{\bf L}^2(\Om) : {\bf v}_h|_{K}\in \boldsymbol{\mathcal{P}}_1(K)\mbox{ for all }K\in\T,{\bf v}$ is continuous at the midpoints $z_E$ of $E\in\mathcal{E}$ and ${\bf v}(z_E)=0$ for all $E\in\mathcal{E}_{\Gamma}\}$.
\end{center}
The discrete space of pressure is defined as $Q_h:=\mathcal{P}_{0}(\T)\cap L_0^2(\Omega)$. The discrete gradient and divergence operators are given by
\begin{equation*}
	(\n_h\textbf{v}_h)|_{K} := \nabla(\textbf{v}_h|_{K})\hspace{0.1cm}\mbox{and}\hspace{0.1cm} (\mbox{div}_h \textbf{v}_h)|_{K} := \mbox{div}(\textbf{v}_h|_{K})\hspace{0.1cm}\mbox{ for all }\hspace{0.1cm} K\in\T ,\textbf{v}_h\in {\bf V}_h.
\end{equation*}
Let
\begin{equation*}
	\enorm{\cdot}_{\rm{pw}}^2:=\sum_{K\in\T}\int_{K} |\n \cdot |^2\,dx.
\end{equation*}

\noindent For ease of notation, the energy norm is defined as $\enorm{({\bf v},q)}_{\rm pw}:=\enorm{\bf v}_{\rm pw}+\|q\|$ and $\|({\bf v},q)\|:=\|{\bf v}\|+\|q\|$ for any ${\bf v}\in {\bf V}\mbox{ and }q\in Q$. Analogously $\enorm{({\bf v},q)}_{\rm pw}^2:=\enorm{\bf v}_{\rm pw}^2+\|q\|^2$ and $\|({\bf v},q)\|^2:=\|{\bf v}\|^2+\|q\|^2$ for any ${\bf v}\in {\bf V}\mbox{ and }q\in Q$.  For any $({\bf y}_h,{\bf v}_h)\in {\bf V}_h\times {\bf V}_h$ and $r_h\in Q_h$ the discrete bilinear forms $a_h(\cdot,\cdot)$ and $b_h(\cdot,\cdot)$ are defined as
\begin{equation*}
a_h(\y_h,\textbf{v}_h):=\sum_{K\in\T}\int_{K}\n\y_h :\n\textbf{v}_h\,dx\mbox{ and }b_h(\y_h,q_h):=\sum_{K\in\T}\int_{K} q_h\mbox{div} \y_h\,dx.
\end{equation*}

The discrete optimality seeks $(\yhb,\rhb,\phb,\shbar,\uhb)\in {\bf V}_h\times Q_h\times {\bf V}_h\times Q_h\times {\bf U}_h$ such that for all $({\bf v}_h,q_h)\in {\bf V}_h\times Q_h$ we can have 
\begin{equation}
\label{dstate}
	\begin{cases}
		a_h(\bar{\y}_h,\textbf{v}_h)-b_h(\textbf{v}_h,\bar{r}_h)=(\textbf{f}+\bar{\U}_h,\textbf{v}_h),\\
		\quad\quad\quad\quad\quad\hspace{0.2cm} b_h(\bar{\y}_h,q_h)=0,
	\end{cases}
\end{equation}
\begin{equation}
\label{dadjoint}
	\begin{cases}		a_h(\bar{\pee}_h,\textbf{v}_h)+b_h(\textbf{v}_h,\bar{s}_h)=(\bar{\y}_h-\y_d,\textbf{v}_h),\\
		\quad\quad\quad\quad\hspace{0.5cm} b_h(\bar{\pee}_h,q_h)=0.
	\end{cases}
\end{equation}

The discrete optimality condition reads
 \begin{equation}
 \label{dop}
 (\alpha\bar{\U}_h+\bar{\pee}_h,\textbf{z}_h-\bar{\U}_h)\geq 0\mbox{ for all } \textbf{z}_h\in {\bf U}_{h}.
 \end{equation}
Here ${\bf U}_h={\bf U}_{\rm ad}$ for variational control approach and ${\bf U}_h={\bf U}_{h,{\rm ad}}:={\bf U}_{\rm ad}\cap{\boldsymbol{\mathcal{P}}}_{0}(\T)$ for discretised control approach.
\noindent
The discrete control in both variational and discretised settings can be represented in terms of the discrete adjoint variable as
\begin{equation}
	\label{defuh}
	\bar{\U}_h(x)=\Pi_{[\U_a,\U_b]}\left(-\frac{1}{\alpha}\bar{\pee}_h(x)\right)\mbox{ and }\bar{\U}_h(x)=\Pi_{[\U_a,\U_b]}\Pi_{0}\left(-\frac{1}{\alpha}\bar{\pee}_h(x)\right),\mbox{ respectively}. 
\end{equation}

\subsection{Error equivalence} 

This section presents the error equivalence results for both the variational and discretised control approaches. Several auxiliary problems are introduced for each approach, and their formulations are discussed in detail. The continuous-level error equivalence results allow the analysis of the optimal control problem to be reduced to the standard analysis of the steady-state Stokes problem. This reduction plays a key role in deriving the a priori and a posteriori error estimates for both control strategies.

\subsubsection{Auxiliary problem for variational approach}
Define auxiliary variables $(\widetilde{\y},\widetilde{r},\widetilde{\pee},\widetilde{s})\in {\bf V}\times Q\times {\bf V}\times Q$  at the continuous level such that, for all $({\bf v}, q)\in {\bf V}\times Q$, they solve the following auxiliary problems \vspace{-1cm}
\begin{multicols}{2}
\begin{align}
\label{cvstate}
	\begin{cases}
		a(\widetilde{\y},\textbf{v})-b(\textbf{v},\widetilde{r})=(\textbf{f}+\bar{\U}_h,\textbf{v}),\\
        \quad\quad\quad\quad\hspace{0.01cm} b(\widetilde{\y},q)=0,
	\end{cases}
\end{align}\break 
\begin{equation}
\label{cvadjoint}
	\begin{cases}
	a(\widetilde{\pee},\textbf{v})+b(\textbf{v},\widetilde{s})=(\bar{\y}_h-\y_d,\textbf{v}),\\
		\quad\quad\quad\quad\hspace{0.01cm} b(\widetilde{\pee},q)=0.
	\end{cases}
\end{equation}
\end{multicols}

\noindent 
Let $\widehat{\bf V}_h$ and $\widehat{Q}_h$ denote the corresponding discrete spaces on triangulation $\widehat{\T}\in \mathbb T(\mathcal T)$. The auxiliary problem at the discrete level is to find $(\widehat{\widetilde{\y}}_h,\widehat{\widetilde{r}}_h, \widehat{\widetilde{\pee}}_h,\widehat{\widetilde{s}}_h)\in {\bf\widehat{V}}_h\times\widehat{Q}_h\times {\bf \widehat{V}}_h\times\widehat{Q}_h$ such that for all $(\widehat{{{\bf v}}}_h, \widehat{{q}}_h)\in \widehat{{\bf V}}_h\times \widehat{Q}_h$ the auxiliary variables solves
\begin{equation}
\label{dvstate}
	\begin{cases}
		a_h(\widehat{\widetilde{\y}}_h,\widehat{{\bf v}}_h)-b_h(\widehat{{\bf v}}_h,\widehat{\widetilde{r}}_h)=(\textbf{f}+\bar{\U}_h,\widehat{{\bf v}}_h)\\
		\quad\quad\quad\quad\quad\quad\hspace{0.001cm} b_h(\widehat{\widetilde{\y}}_h,\widehat{{q}}_h)=0, 
	\end{cases}
\end{equation}
and
\begin{equation}
\label{dvadjoint}
	\begin{cases}
		a_h(\widehat{\widetilde{\pee}}_h,\widehat{{\bf v}}_h)+b_h(\widehat{{\bf v}}_h,\widehat{\widetilde{s}}_h)=(\bar{\y}_h-\y_d,\widehat{{\bf v}}_h)\\
		\quad\quad\quad\quad\quad\quad\hspace{0.001cm} b_h(\widehat{\widetilde{\pee}}_h,\widehat{{q}}_h)=0. 
	\end{cases}
\end{equation}

\subsubsection{Error equivalence result for variational approach}
{\small
\begin{table}[h!]
\centering
\begin{tabular}{|c|c|c|c|c|c|}
\hline
\textbf{Type} & $E_{\rm C}$ & $E_{\rm S}$ & $E_{\rm A}$ & $\widetilde{E}_{\rm S}$ & $\widetilde{E}_{\rm A}$ \\
\hline
Continuous & $\bar{\mathbf{u}} - \bar{\mathbf{u}}_h$ & $(\bar{\mathbf{y}} - \bar{\mathbf{y}}_h,\ \bar{r} - \bar{r}_h)$ & $(\bar{\mathbf{p}} - \bar{\mathbf{p}}_h,\ \bar{s} - \bar{s}_h)$ & $(\widetilde{\mathbf{y}} - \bar{\mathbf{y}}_h,\ \widetilde{r} - \bar{r}_h)$ & $(\widetilde{\mathbf{p}} - \bar{\mathbf{p}}_h,\ \widetilde{s} - \bar{s}_h)$ \\
\hline
 - & $E_{{\rm C},h}$ & $E_{{\rm S},h}$ & $E_{{\rm A},h}$ & $\widetilde{E}_{{\rm S},h}$ & $\widetilde{E}_{{\rm A},h}$ \\
\hline
Discrete & $\widehat{\bar{\mathbf{u}}}_h - \bar{\mathbf{u}}_h$ & $(\widehat{\bar{\mathbf{y}}}_h - \bar{\mathbf{y}}_h,\ \widehat{\bar{r}}_h - \bar{r}_h)$ & $(\widehat{\bar{\mathbf{p}}}_h - \bar{\mathbf{p}}_h,\ \widehat{\bar{s}}_h - \bar{s}_h)$ & $(\widehat{\widetilde{\mathbf{y}}}_h - \bar{\mathbf{y}}_h,\ \widehat{\widetilde{r}}_h - \bar{r}_h)$ & $(\widehat{\widetilde{\mathbf{p}}}_h - \bar{\mathbf{p}}_h,\ \widehat{\widetilde{s}}_h - \bar{s}_h)$ \\
\hline
\end{tabular}
\caption{Notation for error quantities at continuous and discrete levels}
\label{tab:error_notation_horizontal}
\end{table}
}
\begin{theorem}
\label{v error equivalence}
Let $(\bar{\bf y},\rb,\bar{\bf p},\sbar,\bar{\bf u})\in {\bf V}\times Q\times {\bf V}\times Q\times {\bf U}_{\rm ad}$ solve \eqref{cstate}-\eqref{cop} and $(\bar{\bf y}_h,\rhb,\bar{\bf p}_h,\shbar,\bar{\bf u}_h)\in{\bf V}_h\times Q_h\times {\bf V}_h\times Q_h\times\bf U_{\rm ad}$ (resp.$(\widehat{\bar{\bf y}}_h,\widehat{\bar{r}}_h,\widehat{\bar{\bf p}}_h,\widehat{\bar{s}}_h,\widehat{\bar{\U}}_h)\in {\bf \widehat{V}}_h\times \widehat{Q}_h\times {\bf \widehat{V}}_h\times \widehat{Q}_h\times {\bf U}_{ad}$) solve \eqref{dstate}-\eqref{dop} with respect to the triangulation $\T$ (resp $\widehat{\T}$). Let $(\widetilde{\bf y},\widetilde{r},\widetilde{\bf p},\widetilde{s})\in {\bf V}\times Q\times {\bf V}\times Q $ solve the auxiliary problem \eqref{cvstate} and \eqref{cvadjoint} and $(\widehat{\widetilde{\bf y}}_h,\widehat{\widetilde{r}}_h,\widehat{\widetilde{\bf p}}_h,\widehat{\widetilde{s}}_h)\in {\bf \widehat{V}}_h\times \widehat{Q}_h\times {\bf \widehat{V}}_h\times \widehat{Q}_h$ solve \eqref{dvstate} and \eqref{dvadjoint}. Then the following error estimates hold.

\begin{enumerate}[label=(\roman*)]
	\item 
	$\|E_{\rm C}\|+\enorm{E_{\rm S}}_{\rm pw}+\enorm{E_{\rm A}}_{\rm pw}\approx \enorm{\widetilde{E}_{\rm S}}_{\pw}+\enorm{\widetilde{E}_{\rm A}}_{\pw}$,
	\item 
	$\|E_{\rm C}\|+\|E_{\rm S}\|+\|E_{\rm A}\|\approx \|\widetilde{E}_{\rm S}\|+\|\widetilde{E}_{\rm A}\|$,
	\item 
	$\|E_{{\rm C},h}\|+\enorm{E_{{\rm S},h}}_{\rm pw}+\enorm{E_{{\rm A},h}}_{\rm pw}\approx \enorm{\widetilde{E}_{{\rm S},h}}_{\rm pw}+\enorm{\widetilde{E}_{{\rm A},h}}_{\rm pw}$,
	\item 
	$\|E_{{\rm C},h}\|+\|{E_{{\rm S},h}}\|+\|E_{{\rm A},h}\|\approx \|\widetilde{E}_{{\rm S},h}\|+\|\widetilde{E}_{{\rm A},h}\|$.
\end{enumerate}
\end{theorem}
\noindent All error notations are as defined in Table \ref{tab:error_notation_horizontal}. The constants involved in the above equivalence are independent of the mesh size $h$ but depend upon regularizing parameter $\alpha$. The explicit constant dependencies will be discussed in Section \ref{proof of error equi}.
\subsubsection{Auxiliary problem for discretised approach}
Define an auxiliary control variables at continuous and discrete levels as
\begin{equation}
\label{2.21}
	\widetilde{\bf{u}}:=\Pi_{[\U_a,\U_b]}\left(-\frac{\bar{\pee}_h}{\alpha}\right)\in {\bf U}_{\rm{ad}}\mbox{ and }\widetilde{\bf{u}}_h:=\Pi_{[\U_a,\U_b]}\widehat{\Pi}_0\left(-\frac{1}{\alpha}\bar{\pee}_h\right).
\end{equation} 
In addition to the auxiliary problems \eqref{cvstate}–\eqref{dvadjoint}, an additional set of auxiliary problems for the state equation in the discretised control approach is introduced. The auxiliary problem at the continuous level consists in finding 
$(\widehat{\bf y},\widehat{r}) \in {\bf V}\times Q $
such that, the auxiliary variables satisfy
\begin{equation}
\label{cdstate}
	\begin{cases}
		a(\widehat{\y},\textbf{v})-b(\textbf{v},\widehat{r})=(\textbf{f}+\widetilde{\bf u},\textbf{v})\mbox{ for all }{\bf v}\in{\bf V},\\
		\quad\quad\quad\quad b(\widehat{\y},q)=0\mbox{ for all }q\in Q.
	\end{cases}
\end{equation}
At the discrete level, the auxiliary variables are introduced as the solution 
$(\widehat{\bf y}_h,\widehat{r}_h)\in {\bf \widehat{V}}_h\times \widehat{Q}_h$
such that

\begin{equation}
\label{ddstate}
	\begin{cases}
		a_h(\widehat{{\y}}_h,\widehat{\bf v}_h)-b_h(\widehat{{\bf v}}_h,\widehat{{r}}
		_h)=(\textbf{f}+\widetilde{\U}_h,\widehat{\bf v}_h)  \mbox{ for all } \widehat{\bf v}_h\in\widehat{\bf V}_h,\\
		\quad\quad\quad\quad\quad\quad b_h(\widehat{\y}_h,\widehat{q}_h)=0 \mbox{ for all } \widehat{q}_h\in\widehat{Q}_h.
	\end{cases}
\end{equation} 
\subsubsection{Error equivalence result for discretised approach}
Define $\widetilde{E}_{C}:=\widetilde{\bf u}-\uhb$, and $\widetilde{E}_{{\rm C},h}:=\widetilde{\bf u}_h-\uhb$. The rest of the notations for error are same as Table \ref{tab:error_notation_horizontal} mentioned in. 
\begin{theorem}
\label{d error equivalence}

Let $\csol\in{\bf V}\times Q\times {\bf V}\times Q\times {\bf U}_{\rm ad}$ solve \eqref{cstate}-\eqref{cop} and $\dsol$ $\in {\bf V}_h\times Q_h\times {\bf V}_h\times Q_h\times {\bf U}_{{h,\rm ad}}$ (resp.$(\widehat{\bar{\bf y}}_h,\widehat{\bar{r}}_h,\widehat{\bar{\bf p}}_h,\widehat{\bar{s}}_h,\widehat{\bar{\U}}_h)\in {\bf \widehat{V}}_h\times \widehat{Q}_h\times {\bf \widehat{V}}_h\times \widehat{Q}_h\times \widehat{\bf U}_{{h,{\rm ad}}}$) solve \eqref{dstate}-\eqref{dop} with respect to the triangulation $\T$ (resp $\widehat{\T}$). Let $(\widetilde{\bf y},\widetilde{r},\widetilde{\bf p},\widetilde{s})\in {\bf V}\times Q\times {\bf V}\times Q$ solve the auxiliary problem \eqref{cvstate} and \eqref{cvadjoint}. Let $(\widehat{\widetilde{\bf y}}_h,\widehat{\widetilde{r}}_h,\widehat{\widetilde{\bf p}}_h,\widehat{\widetilde{s}}_h)\in \widehat{\bf V}_h\times \widehat{Q}_h\times \widehat{\bf V}_h\times \widehat{Q}_h$ solve \eqref{dvstate} and \eqref{dvadjoint}. Let $\widetilde{\bf u}$ and $\widetilde{\bf u}_h$ be as defined in \eqref{2.21}. Then the following error estimates hold.

\begin{enumerate}[label=(\roman*)]
	\item 
	$\|E_{\rm C}\|+\enorm{E_{\rm S}}_{\rm pw}+\enorm{E_{\rm A}}_{\rm pw}\approx \|\widetilde{E}_{\rm C}\|+\enorm{\widetilde{E}_{\rm S}}_{\pw}+\enorm{\widetilde{E}_{\rm A}}_{\pw}$,
	\item 
	$\|E_{\rm C}\|+\|E_{\rm S}\|+\|E_{\rm A}\|\approx \|\widetilde{E}_{\rm C}\|+\|\widetilde{E}_{\rm S}\|+\|\widetilde{E}_{\rm A}\|$,
	\item 
	$\|E_{{\rm C},h}\|+\enorm{E_{{\rm S},h}}_{\rm pw}+\enorm{E_{{\rm A},h}}_{\rm pw}\approx \|\widetilde{E}_{{\rm C},h}\|+\enorm{\widetilde{E}_{{\rm S},h}}_{\rm pw}+\enorm{\widetilde{E}_{{\rm A},h}}_{\rm pw}$,
	\item 
	$\|E_{{\rm C},h}\|+\|{E_{{\rm S},h}}\|+\|E_{{\rm A},h}\|\approx \|\widetilde{E}_{{\rm C},h}\|+\|\widetilde{E}_{{\rm S},h}\|+\|\widetilde{E}_{{\rm A},h}\|$.
\end{enumerate}
\end{theorem}

\subsection{Quasi-optimality of AFEM}
\label{adaptive convergence}

\subsubsection{Adaptive estimator}
For any $K\in\T$, define the local estimators for state, adjoint, and control variables as follows
\begin{align*}
	\mu_{S,K}^2+\rho_{S,\E(K)}^2&:= h_{K}^2\|\mathbf{f}+\bar{\U}_h\|_{K}^2+\sum_{E\in\E_{K}}h_{K}\|[\n_h\bar{\y}_h.t_E]\|_{E}^2,\\
 \mu_{A,K}^2+\rho_{A,\E(K)}^2&:=
	h_{K}^2\|\bar{\y}_h-\y_d\|_{K}^2+\sum_{E\in\E_{K}}h_{K}\|[\n_h\bar{\pee}_h.t_E]\|_{E}^2,\\
    \eta_{C,K}^2&:=h_{K}^2\|\n_h(\uhb-\alpha^{-1}\phb)\|_{K}^2.
\end{align*}

\noindent Thus define global state, adjoint, and control estimator as
\begin{align*}
&\eta_{\rm st}^2(\T):=\sum_{K\in\T}(\mu_{S,K}^2+\rho_{S,\E(K)}^2),\hspace{0.2cm}\eta_{\rm adj}^2(\T):=\sum_{K\in\T}(\mu_{A,K}^2+\rho_{A,\E(K)}^2),\mbox{ and }\eta_{\rm C}^2(\T):=\sum_{K\in\T}\eta_{C,K}^2. 
\end{align*}
The complete estimator $\eta$ is define as
\begin{align}
\label{estidef}
\eta^2(\T):= \eta_{\rm st}^2(\T)+\eta_{\rm adj}^2(\T)\mbox{ for variational approach and }
\eta^2(\T):= \eta_{\rm st}^2( \T)+\eta_{\rm adj}^2(\T)+\eta_{\rm C}^2(\T)
\end{align}
\noindent for discretised approach. The complete volume estimator for the variational and discretised control approach will be considered as $\mu^2(\T):=\sum_{K\in\T}(\mu_{S,K}^2+\mu_{A,K}^2)$ and $\mu^2(\T) :=\sum_{K\in\T}(\mu_{S,K}^2+\mu_{A,K}^2+\eta_{C,K}^2)$, respectively. Section \ref{a pos} proves the reliability and efficiency of the estimators in the sense that
\begin{equation*}
    \|E_C\|^2+\enorm{E_{\rm S}}_{\rm pw}^2+\enorm{E_{\rm A}}_{\rm pw}^2\lesssim \eta^2(\T)\lesssim \|E_C\|^2+\enorm{E_{\rm S}}_{\rm pw}^2+\enorm{E_{\rm A}}_{\rm pw}^2+{\rm osc}_{\rm st}^2(\T)+{\rm osc}_{\rm adj}^2(\T),
\end{equation*}
\noindent where the data oscillation terms corresponding to state and adjoint problems are given by
\begin{align*}
\rm{osc}_{\rm st}^2(\T)&:=\sum_{K\in\T} h_{K}^2\|(1-\Pi_{0})(\mathbf{f}+{\bar{\bf u}}_{\it h})\|_{K}^2\mbox{ and } \rm{osc}_{\rm adj}^2(\T):=\sum_{{\it K}\in\T} {\it h}_{\it K}^2\|(1-\Pi_{0})({\bar{\bf y}}_{\it h}-{\bf y}_d)\|_{\it K}^2.
\end{align*}

Define complete global oscillation as
$\rm{osc}^2(\T):=\rm{osc}_{\rm st}^2(\T)+\rm{osc}_{\rm adj}^2(\T).$
\subsubsection{Axioms of adaptivity} Let $\widehat{\T}\in\mathbb{T}(\T)$. For notational convenience, the quantities $\mu$, $\eta_{\rm st}$, $\eta_{\rm adj}$, and $\eta$ denote the corresponding global error indicators over the triangulation $\mathcal T$, while $\widehat{\mu}$, $\widehat{\eta}_{\rm st}$, $\widehat{\eta}_{\rm adj}$, and $\widehat{\eta}$ denote those over $\widehat{\mathcal T}$, unless stated otherwise.
 For any $k\in\mathbb{N}$, recall $\T_k$ denotes $kth$ refinement of initial triangulation $\T_0$. Let $\mu_k, \eta_{{\rm st},k}, \eta_{{\rm adj},k}, \eta_{k}$ denote corresponding global quantities over triangulation $\T_k$. Let $({\bar{\bf y}}_h,{\bar{\bf p}}_h,\bar{\bf u}_h)$ and $({\widehat{\bar{\bf y}}}_h,{\widehat{\bar{\bf p}}}_h,\widehat{\bar{\bf u}}_h)$ be the solutions of problem  \eqref{dstate}-\eqref{dop} corresponding to triangulations $\T$ and $\widehat{\T}$, respectively. 
Define the distance function as 
\begin{equation*}
	{\bf d} ^2(\T,\widehat{\T}) := \enorm{\widehat{\bar{\y}}_h-\bar{\y}_h}_{\pw}^2 + \enorm{\widehat{\bar{\pee}}_h-\bar{\pee}_h}_{\pw}^2
\end{equation*}
for variational control approach and 
\begin{equation*}
	{\bf d} ^2(\T,\widehat{\T}) := \|\widehat{\bar{\bf u}}_h-\bar{\bf u}_h\|^2+\enorm{\widehat{\bar{\y}}_h-\bar{\y}_h}_{\pw}^2 + \enorm{\widehat{\bar{\pee}}_h-\bar{\pee}_h}_{\pw}^2
\end{equation*}
for discretised control approach. 
We state below the axioms of adaptivity following \cite{CMMD2014, CCRH17}. 

\begin{enumerate}
	\item[\textbf{(A1)}] \textbf{(Stability.)} For all $\T\in\mathbb{T}$ and  $\widehat{\T}\in\mathbb{T}(\T)$. Axiom \textbf{(A1)} explains stability over non-refined elements of $\T$ 
	\begin{equation*}
		|\widehat{\eta}(\T\cap\widehat{\T})-\eta(\T\cap\widehat{\T})|\leq \Lambda_1 {\bf d}(\T,\widehat{\T})\hspace{0.2cm}\mbox{and}\hspace{0.2cm}|\widehat{\mu}(\T\cap\widehat{\T})-\mu(\T\cap\widehat{\T})|\leq h \Lambda_0 {\bf d}(\T,\widehat{\T}).
	\end{equation*}
     \item[\textbf{(A2)}] \textbf{(Reduction.)} The second axiom \textbf{(A2)} introduces the reduction constant $\rho\in(0,1)$, which governs the contraction of the error indicators on the refined elements of $\mathcal T$. Specifically, it holds that
\begin{equation*}
\widehat{\eta}(\widehat{\mathcal T}\setminus\mathcal T)
\leq \rho\,\eta(\mathcal T\setminus\widehat{\mathcal T})
+\Lambda_2\,{\bf d}(\mathcal T,\widehat{\mathcal T}),
\qquad
\widehat{\mu}(\widehat{\mathcal T}\setminus\mathcal T)
\leq 2^{-1/2}\mu(\mathcal T\setminus\widehat{\mathcal T})
+h\Lambda_0\,{\bf d}(\mathcal T,\widehat{\mathcal T}).
\end{equation*}

	\item[\textbf{(A3)}] \textbf{(Discrete reliability.)} Define a set
	\begin{center}
		$\mathcal{R}(\T,\widehat{\T}):=\{K\in \T : \mbox{ there exists } T\in \T\backslash \widehat {\T}\hspace{0.2cm}\mbox{with dist}(K,T)=0\}$. 
	\end{center}
    The discrete reliability shows
	\begin{equation*}
		{\bf d}^2(\T,\widehat{\T})\leq\Lambda_3 \eta^2(\mathcal{R}(\T,\widehat{\T})).
	\end{equation*}
 Note that, $\T\setminus\widehat{\T}\subseteq\mathcal{R}(\T,\widehat{\T})$. The cardinality of $\mathcal{R}(\T,\widehat{\T})$ is same as cardinality of refined elements upto some fixed multiplicative constant \cite{MR2324418}, where $\mathcal{R}(\T,\widehat{\T})$ denotes $\T\setminus\widehat{\T}$ plus one additional layer of elements. 
	\item[\textbf{(A4)}] \textbf{(Quasi-orthogonality.)} $\mbox{ For all } l\in\mathbb{N}_0,$  \textbf{(A4)} states $\sum_{k=l}^{\infty}{\bf d}^2(\T_k,\T_{k+1})\leq\Lambda_4 \eta_{l}^2$.
\end{enumerate}
Here $\Lambda_i, i=0,1,2,3,4$ are universal constants and they are independent of mesh parameter $h$. The adaptive convergence rates then follows from the Theorem \ref{main theorem} below.
\begin{theorem}
\label{main theorem}
	Let {\bf (A1)-(A4)} hold and there exists a $\theta$ such that, $0<\theta<\theta_0:=\frac{1}{1+\Lambda_{1}^2\Lambda_3}$ and $0<\delta\ll 1$ sufficiently small. Then, for the AFEM-generated sequences $(\T_l)_{l\in\mathbb{N}_0} \in \mathbb{T}(\delta)$ and $(\eta_l)_{l\in\mathbb{N}_0}$, the following holds for all $s>0$
\begin{equation*}
\sup_{l\in\mathbb{N}_0}(1+|\T_l|-|\T_0|)^s\,\eta_l \approx 
\sup_{N\in\mathbb{N}_0}(1+N)^s \min_{\T\in\mathbb{T}(N)} \eta(\T).
\end{equation*}
\end{theorem}
\noindent The adaptivity axioms {\bf (A1)}–stability, {\bf (A2)}–reduction, and {\bf (A4)}–quasi-orthogonality together ensure the R-linear convergence of the estimator $\eta(\mathcal T)$; see \cite[Proposition 4.15]{CMMD2014}. The axiom {\bf (A3)}–discrete reliability is responsible for the quasi-optimality of the estimator. Consequently, the axioms {\bf (A1)}–{\bf (A4)} collectively guarantee both convergence and quasi-optimality of the proposed estimator $\eta(\T)$. Further details can be found in \cite{CMMD2014,CCSP20}. The verification of these axioms for the variational and discretised control approaches is carried out in Section \ref{verification of variational axioms} and Section \ref{verification of discretised axioms}, respectively.

\section{Proofs of error equivalence results}
\label{proof of error equi}
\noindent
The main goal of this section is to derive the explicit constant dependencies appearing in the error equivalence results for the variational and discretised approaches stated in Theorem \ref{v error equivalence} and Theorem \ref{d error equivalence}, respectively. The proofs rely on the Poincaré-type inequalities and discrete Sobolev embeddings introduced in the first subsection. The second and third subsections are devoted to the proofs of the corresponding error equivalence results. These results play a crucial role in the analysis and establishment of the a priori and a posteriori error estimates.

\subsection{Preliminaries}

\begin{lemma}(Poincar\'e type inequality)
	\label{Poincare}
	Let ${\bf v}\in {\bf V}$ and ${\bf v}_h\in {\bf V}_h$. Then there exists constants $C_{\rm{P}},C_{\rm{dP}}>0$ such that
	\begin{enumerate}[label=(\roman*)]
		\item $\|{\bf v}\|\leq C_{\rm{P}}\enorm {{\bf v}}_{\pw}$,
		\item $\|{\bf v}_h\|\leq C_{\rm{dP}}\enorm{{\bf v}_h}_{\pw}$.
	\end{enumerate}
\end{lemma}
\noindent For the proof of \textit{(i)} refer \cite[Theorem 3]{Evans}. The inequality in \textit{(ii)} is discussed in \cite[Theorem 10.6.15]{Brenner}.

\begin{lemma}(Discrete Sobolev embeddings)\cite[Lemma 3.4]{MR4766712}
	\label{discrete sobolev}
	Let ${\bf V}_h$ and $\widehat{\bf V}_h$ be the nonconforming finite element spaces over the triangulations $\T$ and $\widehat{\T}$, respectively. For any ${\bf v}\in {\bf V}+{\bf V}_h$ and ${\bf v}_h\in {\bf V}_h+\widehat{{\bf V}}_h$ the following estimates hold true
	\begin{enumerate}[label=(\roman*)]
		\item $\|{\bf v}\|\leq C_{\rm{PJ}} \enorm{{\bf v}}_{\pw}$,
		\item $\|{\bf v}_h\|\leq C_{\rm{PI}} \enorm{{\bf v}_h}_{\pw}$.
	\end{enumerate}
\end{lemma} 
\subsection{Proof of Theorem \ref{v error equivalence} (Error equivalence result for variational approach)}
\noindent Recall that $(\widetilde{\bf y},\widetilde{r},\widetilde{\bf p},\widetilde{s})\in {\bf V}\times Q\times {\bf V}\times Q$ are the solution of auxiliary problems defined in \eqref{cvstate}-\eqref{cvadjoint}. We observe that \eqref{dstate} and \eqref{dadjoint} are standard finite element approximation of \eqref{cvstate} and \eqref{cvadjoint}, respectively. Subtraction of \eqref{cvstate} and \eqref{cvadjoint} from \eqref{cstate} and \eqref{cadjoint}, respectively gives
\begin{equation}
\label{2.11}
	\begin{cases}
		a(\bar{\y}-\widetilde{\y},\textbf{v})-b(\textbf{v},\bar{r}-\widetilde{r})=(\bar{\U}-\bar{\U}_h,\textbf{v}) & \mbox{ for all } \textbf{v}\in {\bf V},\\
		\quad\quad\quad\quad\quad\hspace{0.5cm} b(\bar{\y}-\widetilde{\y},q)=0 & \mbox{ for all } q\in Q,
	\end{cases}
\end{equation}
and 
\begin{equation}
\label{2.12}
	\begin{cases}
		a(\bar{\pee}-\widetilde{\pee},\textbf{v})+b(\textbf{v},\bar{s}-\widetilde{s})=(\bar{\y}-\bar{\y}_h,\textbf{v}) & \mbox{ for all }\textbf{v}\in {\bf V},\\
		\quad\quad\quad\quad\quad\hspace{0.5cm} b(\bar{\pee}-\widetilde{\pee},q)=0 & \mbox{ for all } q\in Q.
	\end{cases}
\end{equation}

\noindent The stability results \cite[Theorem 5.2]{MR0548867} for \eqref{2.11} and \eqref{2.12} gives
\begin{align}
	\label{cst}
	\enorm{(\bar{\y}-\widetilde{\y},\bar{r}-\widetilde{r})}_{\pw}\leq C_{\rm{cst}} \|\bar{\U}-\bar{\U}_h\|,\\
	\label{cad}
	\enorm{(\bar{\pee}-\widetilde{\pee},\bar{s}-\widetilde{s})}_{\pw}\leq C_{\rm{cad}}\|\bar{\y}-\bar{\y}_h\|.
\end{align}

\noindent Recall that $(\widehat{{\widetilde{\bf y}}}_h,\widehat{\widetilde{r}}_h,\widehat{{\widetilde{\bf p}}}_h,\widehat{\widetilde{s}}_h)\in {\widehat{\bf V}}_h\times \widehat{Q}_h\times {\widehat{\bf V}}_h\times \widehat{Q}_h$ are the solution of auxiliary problems at discrete levels defined in \eqref{dvstate}-\eqref{dvadjoint}. Subtraction of equations \eqref{dvstate} and \eqref{dvadjoint} from equations \eqref{dstate} and \eqref{dadjoint}, respectively, results in
\begin{equation}
	\label{2.17}
	\begin{cases}
		a_h(\widehat{\bar{\y}}_h-\widehat{\widetilde{\y}}_h,\widehat{{\bf v}}_h)-b_h(\widehat{{\bf v}}_h,\widehat{\bar{r}}_h-\widehat{\widetilde{r}}_h)=(\widehat{\bar{\U}}_h-\bar{\U}_h,\widehat{{\bf v}}_h) & \mbox{ for all } \widehat{{\bf v}}_h\in {\bf \widehat{V}}_h,\\
		\quad\quad\quad\quad\quad\quad\quad\quad b_h(\widehat{\bar{\y}}_h-\widehat{\widetilde{\y}}_h,\widehat{{q}}_h)=0 & \mbox{ for all } \widehat{{q}}_h\in \widehat{Q}_h,
	\end{cases}
\end{equation}
and 
\begin{equation}
	\label{2.18}
	\begin{cases}
		a_h(\widehat{\bar{\pee}}_h-\widehat{\widetilde{\pee}}_h,\widehat{{\bf v}}_h)+b_h(\widehat{{\bf v}}_h,\widehat{\bar{s}}_h-\widehat{\widetilde{s}}_h)=(\widehat{\bar{\y}}_h-\bar{\y}_h,\widehat{{\bf v}}_h) &\mbox{ for all } \widehat{{\bf v}}_h\in {\bf \widehat{V}}_h,\\
		\quad\quad\quad\quad\quad\quad\quad\quad\hspace{0.3cm} b_h(\widehat{\bar{\pee}}_h-\widehat{\widetilde{\pee}}_h,\widehat{{q}}_h)=0 & \mbox{ for all }\widehat{{q}}_h\in \widehat{Q}_h.
	\end{cases}
\end{equation}

\noindent The stability result for \eqref{2.17} and \eqref{2.18} shows
\begin{align}
	\label{dst}
	\enorm{(\widehat{\bar{\y}}_h-\widehat{\widetilde{\y}}_h,\widehat{\bar{r}}_h-\widehat{\widetilde{r}}_h)}_{\pw}&\leq
	C_{\rm{dst}}\|\widehat{\bar{\U}}_h-\bar{\U}_h\|,\\
	\label{dad}
	\enorm{(\widehat{\bar{\pee}}_h-\widehat{\widetilde{\pee}}_h,\widehat{\bar{s}}_h-\widehat{\widetilde{s}}_h)}_{\pw}&\leq C_{\rm{dad}} \|\widehat{\bar{\y}}_h-\bar{\y}_h\|.
\end{align}
The estimates mentioned in \eqref{cst}-\eqref{cad} and \eqref{dst}-\eqref{dad} play a crucial role in proving the equivalence of the errors as mentioned below. 
\subsubsection{Proof of Theorem \ref{v error equivalence}\textit{(i)} and \textit{(ii)}} 
\label{(2.1)(i-ii)}
\begin{proof}{\bf Step-1 : (primary estimates)}
Choose ${\bf v}=\bar{\y}-\widetilde{\y} , q=\bar{r}-\widetilde{r}$ in \eqref{2.11} , ${\bf v}=\bar{\pee}-\widetilde{\pee} ,  q=\bar{s}-\widetilde{s}$ in \eqref{2.12}, ellipticity of the bilinear form $a(\cdot,\cdot)$, and Lemma \ref{Poincare}\textit{(i)} to get
\begin{equation}
	\label{3.1}
	\enorm{\bar{\y}-\widetilde{\bf y}}_{\pw}\leq C_{\rm P}\wcerr\hspace{0.2cm}\mbox{and}\hspace{0.2cm}\|\bar{\y}-\widetilde{\bf y}\|\leq C_{1}\wcerr\mbox{ with }C_1:=C_{\rm P}^2.
\end{equation}
Analogous calculations give
\begin{equation}
	\label{3.2}
	\enorm{\bar{\pee}-\widetilde{\pee}}_{\rm pw}\leq C_{\rm P}\wlsterr\hspace{0.2cm}\mbox{and}\hspace{0.2cm}\|\bar{\pee}-\widetilde{\pee}\|\leq C_{1} \wlsterr.
\end{equation}
{\bf Step-2 : (Control based bounds for the state and adjoint errors)}
A use of triangle inequality, along with \eqref{3.1} show
\begin{equation}
	\label{3.3}
	\wsterr\leq C_{\rm P}\wcerr +\wauxsterr\hspace{0.2cm}\mbox{and}\hspace{0.2cm} \wlsterr \leq C_{1}\wcerr+\wlauxsterr.
\end{equation}
 Twice use of the triangle inequality, \eqref{3.1}-\eqref{3.2}, and Lemma \ref{discrete sobolev}\textit{(i)} show
 \begin{equation}
 \label{ppw}
 	\waderr\leq C_{2}(\wcerr+\wauxsterr+\wauxaderr)\mbox{ and }
 \end{equation}
\begin{equation}
	\label{3.5}
	\wladerr\leq C_1\wlsterr+\wlauxaderr\leq C_{3}(\wcerr+\wlauxsterr+\wlauxaderr)
\end{equation}
with $C_2:=\max\{C_1C_{\rm{P}},C_{\rm{PJ}}C_{\rm{P}},1\}$ and $C_3:=\max\{1,C_1^2\}$. The use of triangle inequality and \eqref{cst} yields
\begin{equation}
	\label{3.6}
\wtum\leq C_{4}(\wcerr+\wTUM)\mbox{ with }C_4:=\max\{1,C_{\rm cst}\}.
\end{equation}
Similarly, the triangle inequality, \eqref{cad}, and \eqref{3.3} result in
\begin{equation}
	\label{3.7}
\wmai\leq C_{5}(\wcerr+\enorm{\widetilde{\y}-\bar{\bf y}_h}_{\pw}+\wMAI)\mbox{ and }	\wmai\leq C_{6}(\wcerr+\|{\widetilde{\y}-\bar{\bf y}_h}\|+\wMAI)
\end{equation}
with $C_5:=\max\{1, C_{\rm cad}C_{\rm PJ}, C_{\rm cad}C_{\rm PJ}C_{\rm cst}\}$, $C_6:=\max \{1, C_{\rm cad}, C_{\rm cad}C_{\rm PJ}C_{\rm cst}\}.$ 

{\bf Step-3 : (Upper bound on the control)}
The choice ${\bf z}=\bar{{\bf u}}_h\in {\bf U}_{\rm ad}$ in \eqref{cop} and ${\bf z}_h=\bar{\bf u}\in{\bf U}_{\rm ad}$ in \eqref{dop} yield
\begin{equation*}
	(\alpha \bar{\U}+\bar{\pee},\bar{\U}_h-\bar{\U})\geq 0\hspace{0.3cm}\mbox{and}\hspace{0.3cm} (\alpha \bar{\U}_h+\bar{\pee}_h,\bar{\U}-\bar{\U}_h)\geq 0.
\end{equation*}
Add both displayed inequalities and introduce $\widetilde{\pee}$ to the above equation to get
\begin{align}
\label{new}
	\alpha \cerr &\leq (\bar{\pee}_h-\bar{\pee},\bar{\U}-\bar{\U}_h)=(\bar{\pee}_h-\widetilde{\pee},\bar{\U}-\bar{\U}_h)+(\widetilde{\pee}-\bar{\pee},\bar{\U}-\bar{\U}_h).
\end{align}
Choice of ${\bf v}=\widetilde{\pee}-\bar{\pee}$ in \eqref{2.11} and $q=r-\widetilde{r}$ in \eqref{2.12} give
\begin{equation}
\label{ptil}
(\widetilde{\pee}-\bar{\pee},\bar{\U}-\bar{\U}_h)=a(\bar{\y}-\widetilde{\y},\widetilde{\pee}-\bar{\pee})-b(\widetilde{\pee}-\bar{\pee},\bar{r}-\widetilde{r})=a(\bar{\y}-\widetilde{\y},\widetilde{\pee}-\bar{\pee}).
\end{equation}
Again, the choice of ${\bf v}=\widetilde{\y}-\bar{\y}$ in \eqref{2.12} and $q=\bar{s}-\widetilde{s}$ in \eqref{2.11} results in
\begin{equation*}
    a(\yb-\widetilde{\y},\widetilde{\pee}-\bar{\pee})=(\yb-\yhb,\widetilde{\y}-\yb)-b(\widetilde{\y}-\yb,\sbar-\widetilde{s})=(\yb-\yhb,\widetilde{\y}-\yb).
\end{equation*}
Substitution of the above estimate in \eqref{ptil} leads to $(\widetilde{\pee}-\bar{\pee},\bar{\U}-\bar{\U}_h)=(\yb-\yhb,\widetilde{\y}-\yb)$. Substitution of this into \eqref{new} yields
\begin{align*}
    \alpha \cerr\leq (\bar{\pee}_h-\widetilde{\pee},\bar{\U}-\bar{\U}_h)+(\yb-\yhb,\widetilde{\y}-\yb).
\end{align*}
Some elementary algebra in the above estimate leads to
\begin{align*}
    \alpha \cerr\leq (\bar{\pee}_h-\widetilde{\pee},\bar{\U}-\bar{\U}_h)+(\yb-\widetilde{\y},\widetilde{\y}-\yhb)+\|\widetilde{\y}-\yhb\|^2-\|\yb-\yhb\|^2.
\end{align*}
Use the Cauchy-Schwartz inequality and \eqref{3.1} to get
\begin{align*}
    \alpha \cerr\leq \|\widetilde{\pee}-\bar{\pee}_h\|\|\bar{\U}-\bar{\U}_h\|+C_{1}\|\bar{\bf u}-\bar{\bf u}_h\|\|\widetilde{\y}-\yhb\|+\|\widetilde{\y}-\yhb\|^2.
\end{align*}
Applying H\"older’s inequality with 
$a=\|\widetilde{\bf p}-\bar{\bf p}_h\|$, $b=\|\bar{\bf u}-\bar{\bf u}_h\|$, and $\epsilon=\alpha/2$, 
and again with 
$a=C_{1}\|\bar{\bf u}-\bar{\bf u}_h\|$, $b=\|\widetilde{\bf y}-\bar{\bf y}_h\|$, and $\epsilon=2C_{1}^{2}/\alpha$, 
together with Lemma \ref{discrete sobolev}{\it (i)}, lead to

\begin{equation}
\label{3.8}
\wcerr\leq C_7 (\wlauxaderr+\wlauxsterr)\mbox{ and } \wcerr\leq C_7C_{\rm PJ}(\wauxaderr+\wauxsterr)
\end{equation}
with $C_7^2:=\max\{{2}/{\alpha}^2,2C_1^2/\alpha^2+2/\alpha\}$.

{\bf Step-4 : (Complete upper and lower bounds)}
The use of \eqref{3.3}-\eqref{ppw}, \eqref{3.6}-\eqref{3.7}, and \eqref{3.8} show
\begin{align*}
	\|E_{\rm C}\|+\enorm{E_{\rm S}}_{\rm pw}+\enorm{E_{\rm A}}_{\rm pw}\leq C_{\rm c1}(\enorm{\widetilde{E}_{\rm S}}_{\rm pw}+\enorm{\widetilde{E}_{\rm A}}_{\rm pw})
\end{align*}
with $C_{\rm c1}:=C_7C_{\rm PJ}(2+2C_{\rm P}+2C_4+C_2+C_5+2C_{\rm PJ}).$ Using \eqref{3.3}, \eqref{3.5}-\eqref{3.7}, and \eqref{3.8}, we obtain
\begin{align*}
	\|E_{\rm C}\|+\|{E_{\rm S}}\|+\|E_{\rm A}\|\leq C_{\rm c2}(\|\widetilde{E}_{\rm S}\|+\|\widetilde{E}_{\rm A}\|)\mbox{ with }C_{\rm c2}:=C_7(1+C_1+C_3+C_4+C_6).
\end{align*}
Introducing $(\bar{\bf y}, \bar{r}, \bar{\bf p}, \bar{s})$, triangle inequality, stability results from \eqref{cst}-\eqref{cad}, and Lemma \ref{discrete sobolev}{\it (i)} leads to
\begin{align*}
	\enorm{\widetilde{E}_{\rm S}}_{\rm pw}+\enorm{\widetilde{E}_{\rm A}}_{\rm pw}\leq C_{\rm c3}(\|E_{\rm C}\|+\enorm{E_{\rm S}}_{\rm pw}+\enorm{E_{\rm A}}_{\rm pw})\mbox{ with }C_{\rm c3}:=\max\{1,C_{\rm cst}, C_{\rm cad}C_{\rm PJ}\}.
\end{align*}
Similarly, the triangle inequality and the stability results  \eqref{cst}-\eqref{cad} yield
\begin{align*}
	\|\widetilde{E}_{\rm S}\|+\|\widetilde{E}_{\rm A}\|\leq C_{\rm c4}(\|E_{\rm C}\|+\|{E_{\rm S}}\|+\|E_{\rm A}\|)\mbox{ with }C_{\rm c4}:=\max\{1, C_3C_{\rm cst},C_3C_{\rm cad}\}.
\end{align*}
 This completes the proof of Theorem \ref{v error equivalence}\textit{(i)} and \textit{(ii)}.
 \end{proof}
\subsubsection{Proof of Theorem \ref{v error equivalence}\textit{(iii)} and \textit{(iv)}}
The proofs of \textit{(iii)} and \textit{(iv)} are analogous to those of \textit{(i)} and \textit{(ii)} at the discrete level. Because the Crouzeix–Raviart nonconforming spaces are not nested under mesh refinement, discrete-level equivalence is crucial for discrete reliability; hence, the proof is included for completeness.
\begin{proof}
{\bf Step-1 : (Primary estimates)}
Using \eqref{2.17} and \eqref{2.18} with the choices
$\widehat{{\bf v}}_h=\widehat{\bar{\y}}_h-\widehat{\widetilde{\y}}_h$, $\widehat{{q}}_h=\widehat{\bar{r}}_h-\widehat{\widetilde{r}}_h$ and
$\widehat{{\bf v}}_h=\widehat{\bar{\pee}}_h-\widehat{\widetilde{\pee}}_h$, $\widehat{{q}}_h=\widehat{\bar{s}}_h-\widehat{\widetilde{s}}_h$, respectively together with the ellipticity of the bilinear form $a(\cdot,\cdot)$, and Lemma \ref{Poincare}{\it (ii)}, we obtain
\begin{equation}
	\label{3.9}
	\enorm{\widehat{\bar{\y}}_h-\widehat{\widetilde{\y}}_h}_{\rm{pw}}\leq C_{\rm{dP}}\1\mbox{, }\|\widehat{\bar{\y}}_h-\widehat{\widetilde{\y}}_h\|\leq C_8\1,
\end{equation}
and
\begin{equation}
   \label{3.10}
   \enorm{\widehat{\bar{\pee}}_h-\widehat{\widetilde{\pee}}_h}_{\rm{pw}}\leq C_{\rm{dP}}\w\mbox{, }\|\widehat{\bar{\pee}}_h-\widehat{\widetilde{\pee}}_h\|\leq C_8\w\mbox{ with }C_8:=C_{\rm dP}^2.
\end{equation}
{\bf Step-2 : (Control based bounds for the state and adjoint errors)}
A use of triangle inequality along with equation \eqref{3.9} give
\begin{equation}
	\label{3.11}
	\2\leq C_{\rm{dP}}\1 +\6\hspace{0.2cm}\mbox{and}\hspace{0.2cm} \w \leq C_{8}\1+\q.
\end{equation}
Applying the triangle inequality twice along with \eqref{3.9}-\eqref{3.10}, and Lemma \ref{discrete sobolev}{\it (ii)} yield
\begin{equation}
	\label{3.12}
	\3\leq C_9(\1+\6+\7)\mbox{ with }C_9:=\max\{1, C_{\rm dP}^2C_{\rm PI}\}.
\end{equation}
Again use of triangle inequality twice, with \eqref{3.9}-\eqref{3.10} result in
\begin{equation}
	\label{3.13}
	\x\leq C_{10}(\1+\q+\z)\mbox{ with }C_{10}:=\max\{1,C_8^2\}.
\end{equation}
 Use of triangle inequality and \eqref{dst} yield
\begin{equation}
	\label{3.14}
	\4\leq C_{11} (\1+\8)\mbox{ with }C_{11}:=\max\{1,C_{\rm{dst}}\}.
\end{equation} Analogously, an application of the triangle inequality, Lemma \ref{Poincare}{\it (ii)}, and the stability results from \eqref{dst}-\eqref{dad} leads to
\begin{align}
	\label{3.15}
    &\5\leq C_{12} (\1+\q+\9)\mbox{ with }C_{12}:=\max\{1, C_{\rm dP}C_{\rm dad}C_{\rm dst}\}\\
    \label{s}
	&\mbox{ and }\5\leq C_{13} (\1+\6+\9)\mbox{ with }C_{13}:=C_{\rm PI}C_{12}.
\end{align} 
{\bf Step-3 : (Upper bound on the control)}
An application of the optimality condition \eqref{dop} at the $\widehat{\mathcal T}$ level with the choice ${\bf z}_h=\bar{\bf u}_h\in {\bf U}_{\rm ad}$, together with the same condition at the $\mathcal {T}$ level with ${\bf z}_h=\widehat{\bar{\bf u}}_h\in {\bf U}_{\rm ad}$, leads to

\begin{equation*}
	(\alpha \widehat{\bar{\U}}_h+\widehat{\bar{\pee}}_h,\bar{\U}_h-\widehat{\bar{\U}}_h)\geq 0\hspace{0.3cm}\mbox{and}\hspace{0.3cm} (\alpha \bar{\U}_h+\bar{\pee}_h,\widehat{\bar{\U}}_h-\bar{\U}_h)\geq 0.
\end{equation*}

\noindent Add the displayed inequalities and introduce the auxiliary variable $\widehat{\widetilde{\pee}}_h$ to obtain
\begin{equation}
\label{neww}
	\alpha\1^2\leq (\phb-\widehat{\bar{\pee}}_h,\widehat{\bar{\U}}_h-\uhb)=(\phb-\widehat{\widetilde{\pee}}_h,\widehat{\bar{\U}}_h-\uhb)+(\widehat{\widetilde{\pee}}_h-\widehat{\bar{\pee}}_h,\widehat{\bar{\U}}_h-\uhb).
\end{equation} 
Choice of $\widehat{{\bf v}}_h=\widehat{\widetilde{\bf p}}_h-{\widehat{\bar{\bf p}}}_h$ in \eqref{2.17} and $\widehat{{q}}_h=\widehat{\bar{r}}_h-\widehat{\widetilde{r}}_h$ in \eqref{2.18} give
\begin{equation}
\label{328}
    (\widehat{\bar{\bf u}}_h-\bar{\bf u}_h, \widehat{\widetilde{\bf p}}_h-\widehat{\bar{\bf p}}_h)=a_h(\widehat{\bar{\bf y}}_h-\widehat{\widetilde{\bf y}}_h,\widehat{\widetilde{\bf p}}_h-\widehat{\bar{\bf p}}_h)-b_h(\widehat{\widetilde{\bf p}}_h-{\widehat{\bar{\bf p}}}_h, \widehat{\bar r}_h-\widehat{\widetilde{r}}_h)=a_h(\widehat{\bar{\bf y}}_h-\widehat{\widetilde{\bf y}}_h,\widehat{\widetilde{\bf p}}_h-\widehat{\bar{\bf p}}_h).
\end{equation}
Again choice of $\widehat{{\bf v}}_h=\widehat{\widetilde{\bf y}}_h-\widehat{\bar{\bf y}}_h$ in \eqref{2.18} and $\widehat{q}_h=\widehat{\bar{s}}_h-\widehat{\widetilde{s}}_h$ in \eqref{2.17} yield
\begin{equation*}
    a_h(\widehat{\bar{\bf y}}_h-\widehat{\widetilde{\bf y}}_h,\widehat{\widetilde{\bf p}}_h-\widehat{\bar{\bf p}}_h)=(\widehat{\bar{\bf y}}_h-\bar{\bf y}_h, \widehat{\widetilde{\bf y}}_h-\widehat{\bar{\bf y}}_h)-b_h(\widehat{\widetilde{\bf y}}_h-\widehat{\bar{\bf y}}_h, \widehat{\bar{s}}_h-\widehat{\widetilde{s}}_h)=(\widehat{\bar{\bf y}}_h-\bar{\bf y}_h, \widehat{\widetilde{\bf y}}_h-\widehat{\bar{\bf y}}_h).
\end{equation*}
Substituting above estimate in \eqref{328} results in $(\widehat{\bar{\bf u}}_h-\bar{\bf u}_h, \widehat{\widetilde{\bf p}}_h-\widehat{\bar{\bf p}}_h)=(\widehat{\bar{\bf y}}_h-\bar{\bf y}_h, \widehat{\widetilde{\bf y}}_h-\widehat{\bar{\bf y}}_h)$. Substitution of this in \eqref{neww} yield
\begin{equation*}
    \alpha\1^2\leq (\phb-\widehat{\widetilde{\pee}}_h,\widehat{\bar{\U}}_h-\uhb)+(\widehat{\bar{\bf y}}_h-\bar{\bf y}_h, \widehat{\widetilde{\bf y}}_h-\widehat{\bar{\bf y}}_h).
\end{equation*}
Some elementary calculation in above estimate leads to
\begin{equation*}
    \alpha\1^2\leq (\phb-\widehat{\widetilde{\pee}}_h,\widehat{\bar{\U}}_h-\uhb)+(\widehat{\bar{\bf y}}_h-\widehat{\widetilde{\bf y}}_h, \widehat{\widetilde{\bf y}}_h-\bar{\bf y}_h)+\|\widetilde{\bf y}_h-\bar{\bf y}_h\|^2-\|\widehat{\bar{\bf y}}_h-\bar{\bf y}_h\|^2.
\end{equation*}
Cauchy-Schwartz inequality and \eqref{3.9} give
\begin{equation*}
    \alpha\1^2\leq \|\widehat{\widetilde{\pee}}_h-\phb\|\|\widehat{\bar{\U}}_h-\uhb\|+C_{8}\|\widehat{\bar{\U}}_h-\uhb\|\|\widehat{\widetilde{\bf y}}_h-\bar{\bf y}_h\|+\|\widetilde{\bf y}_h-\bar{\bf y}_h\|^2.
\end{equation*}
Applying Young’s inequality with $a=\|\widehat{\widetilde{\pee}}_h-\phb\|$, $b=\|\widehat{\bar{\U}}_h-\uhb\|$, and $\epsilon=\alpha/2$, and again with $a=C_{8}\|\widehat{\bar{\U}}_h-\uhb\|$, $b=\|\widehat{\widetilde{\bf y}}_h-\bar{\bf y}_h\|$, and $\epsilon=2C_8^2/\alpha$, together with Lemma \ref{discrete sobolev}{\it (ii)}, yields

\begin{equation}
\label{c15}
    \1\leq C_{14}(\|\widehat{\widetilde{\bf p}}_h-\bar{\bf p}_h\|+\|\widehat{\widetilde{\bf y}}_h-\bar{\bf y}_h\|)\mbox{ and }\1\leq C_{15}(\enorm{\widehat{\widetilde{\bf p}}_h-\bar{\bf p}_h}_{\rm pw}+\enorm{\widehat{\widetilde{\bf y}}_h-\bar{\bf y}_h}_{\rm pw})
\end{equation}
with $C_{14}:=\max\{{2}/{\alpha}^2,2C_8^2/\alpha^2+2/\alpha\}$ and $C_{15}:=C_{\rm PI}C_{14}$. \\
{\bf Step-4 : (Complete upper and lower bounds)}
Use \eqref{3.11}-\eqref{3.12}, and \eqref{3.14}, \eqref{s}, and \eqref{c15} to obtain
\begin{align*}
	\|E_{{\rm C},h}\|+\enorm{E_{{\rm S},h}}_{\rm pw}+\enorm{E_{{\rm A},h}}_{\rm pw}\leq C_{\rm c5}(\enorm{\widetilde{E}_{{\rm S},h}}_{\rm pw}+\enorm{\widetilde{E}_{{\rm A},h}}_{\rm pw})
\end{align*}
\noindent with $C_{\rm c5}:=C_{14}(1+C_{\rm dP}+C_9C_{\rm dP}+C_{11}+C_{12})$. Use of  \eqref{3.11}, \eqref{3.13}-\eqref{3.15}, and \eqref{c15} leads to
\begin{align*}
	\|E_{{\rm C},h}\|+\|E_{{\rm S},h}\|+\|E_{{\rm A},h}\|\leq C_{\rm c6}(\|\widetilde{E}_{{\rm S},h}\|+\|\widetilde{E}_{{\rm A},h}\|)\mbox{ with }C_{\rm c6}:=C_{15}(1+C_8+C_{10}+C_{11}+C_{13}).
\end{align*}
Following the analogous calculation as we have done in Step-4 of Section \ref{(2.1)(i-ii)} at discrete level by introducing $(\widehat{\bar{\bf y}}_h,\widehat{\bar{r}}_h,\widehat{\bar{\bf p}}_s,\widehat{\bar{s}}_h)$ results in 
\begin{align*}
	\enorm{\widetilde{E}_{{\rm S},h}}_{\rm pw}+\enorm{\widetilde{E}_{{\rm A},h}}_{\rm pw}\leq C_{\rm c7}(\|E_{{\rm C},h}\|+\enorm{E_{{\rm S},h}}_{\rm pw}+\enorm{E_{{\rm A},h}}_{\rm pw})\mbox{ with }C_{\rm c7}:=\max\{1, C_{\rm dst},C_{\rm dad}C_{\rm PI}\}\mbox{ and }
\end{align*}
\begin{align*}
	\|\widetilde{E}_{{\rm S},h}\|+\|\widetilde{E}_{{\rm A},h}\|\leq C_{\rm c8}(\|E_{{\rm C},h}\|+\|E_{{\rm S},h}\|+\|E_{{\rm A},h}\|)\mbox{ with }C_{\rm c8}:=\max\{1, C_{10}C_{\rm dst}, C_{10}C_{\rm dad}\}.
\end{align*}This completes the proof of Theorem \ref{v error equivalence}. 
 \end{proof}

\subsection{Proof of Theorem \ref{d error equivalence} (Error equivalence result for discretised approach)}
\noindent Recall $(\widehat{{\bf y}},\widehat{r},\widehat{\bf y}_h,\widehat{r}_h)\in {\bf V}\times Q\times {\widehat{\bf V}}_h\times {\widehat{Q}}_h$ are the solution of continuous and discrete auxiliary problems defined in \eqref{cdstate}-\eqref{ddstate}. Subtraction of \eqref{cdstate}-\eqref{ddstate} from \eqref{cstate},\eqref{dstate}(at ${\widehat{\T}}$ level), respectively give
\begin{equation}
	\label{2.24}
	\begin{cases}
		a(\bar{\y}-\widehat{\y},\textbf{v})-b(\textbf{v},\bar{r}-\widehat{r})=(\bar{\U}-\widetilde{\bf u},\textbf{v}) & \mbox{ for all } \textbf{v}\in {\bf V},\\
		\quad\quad\quad\quad\quad\hspace{0.3cm} b(\bar{\y}-\widehat{\y},q)=0 & \mbox{ for all } q\in Q,
	\end{cases}
\end{equation}
and
\begin{equation}
	\label{2.25}
	\begin{cases}
		a_h(\widehat{{\bar{\y}}}_h-\widehat{\y}_h,\widehat{\textbf{v}}_h)-b_h(\widehat{\textbf{v}}_h,\widehat{\bar{{r}}}_h-\widehat{r}_h)=(\widehat{\bar{{\U}}}_h-\widetilde{\bf u}_h,\widehat{\textbf{v}}_h) & \mbox{ for all } \widehat{\textbf{v}}_h\in \widehat{V}_h,\\
		\quad\quad\quad\quad\quad\quad\quad\quad b_h(\widehat{{\bar{\y}}}_h-\widehat{\y}_h,\widehat{q}_h)=0 & \mbox{ for all } \widehat{q}_h\in \widehat{Q}_h. 
	\end{cases}
\end{equation}
\noindent Subtraction of \eqref{cvstate} and \eqref{dvstate} from \eqref{cdstate} and \eqref{ddstate}, respectively give

\begin{equation}
	\label{332}
	\begin{cases}
		a(\widehat{\y}-\widetilde{\y},\textbf{v})-b(\textbf{v},\widehat{r}-\widetilde{r})=(\widetilde{\U}-\widetilde{\bf u}_h,\textbf{v}) & \mbox{ for all } \textbf{v}\in {\bf V},\\
		\quad\quad\quad\quad\quad\hspace{0.3cm} b(\widehat{\y}-\widetilde{\y},q)=0 & \mbox{ for all } q\in Q,
	\end{cases}
\end{equation}
and
\begin{equation}
	\label{333}
	\begin{cases}
		a_h(\widehat{{{\y}}}_h-\widehat{\widetilde{\y}}_h,\widehat{\textbf{v}}_h)-b_h(\widehat{\textbf{v}}_h,\widehat{{{r}}}_h-\widehat{\widetilde{r}}_h)=(\widetilde{\bf u}_h-\uhb,\widehat{\textbf{v}}_h) & \mbox{ for all } \widehat{\textbf{v}}_h\in \widehat{V}_h,\\
		\quad\quad\quad\quad\quad\quad\quad\quad b_h(\widehat{{{\y}}}_h-\widehat{\widetilde{\y}}_h,\widehat{q}_h)=0 & \mbox{ for all } \widehat{q}_h\in \widehat{Q}_h. 
	\end{cases}
\end{equation}


The proof of Theorem~\ref{d error equivalence} follows the same steps as those of the previous Theorem~\ref{v error equivalence}. Due  to the discrete control approach, some modifications are required; therefore, we highlight only the main differences here. For completeness, the full proof is provided in Appendix~\ref{appendix}.
\subsubsection{Proof of Theorem \ref{d error equivalence}{\it (i)-(ii)}} 
\label{(2.2)(i)-(ii)}
\begin{proof}
{\bf Step-1 : (Primary estimates)}
Choose ${\bf v}=\bar{\y}-\widetilde{\bf y}, q=\bar{r}-\widetilde{r}$ in \eqref{2.11}, ${\bf v}=\bar{\y}-\widehat{\y}, q=\bar{r}-\widehat{r}$ in \eqref{2.24}, and ${\bf v}=\bar{\pee}-\widetilde{\pee}, q=\bar{s}-\widetilde{s}$ in \eqref{2.12} along with the ellipticity of bilinear form $a$ to get
\begin{equation}
	\label{3.17}
	\enorm{\bar{\y}-\widetilde{\y}}_{\rm pw}\leq C_{\rm P}\|\bar{{\bf u}}-\bar{{\bf u}}_h\|, \enorm{\bar{\y}-\widehat{\y}}_{\rm pw}\leq C_{\rm P}\|\bar{{\bf u}}-\widetilde{{\bf u}}\|,\mbox{ and }\enorm{\bar{\pee}-\widetilde{\pee}}_{\rm pw}\leq C_{\rm P}\|\bar{{\bf y}}-\bar{{\bf y}}_h\|. 
\end{equation}
Lemma \ref{Poincare}{\it(i)} and \eqref{3.17} give
\begin{equation}
	\label{3.18}
	\|\bar{\y}-\widetilde{\y}\|\leq C_{\rm P}^2\|\bar{{\bf u}}-\bar{{\bf u}}_h\|, \|\bar{\y}-\widehat{\y}\|\leq C_{\rm P}^2\|\bar{{\bf u}}-\widetilde{{\bf u}}\|,\mbox{ and }\|\bar{\pee}-\widetilde{\pee}\|\leq C_{\rm P}^2\|\bar{{\bf y}}-\bar{{\bf y}}_h\|. 
\end{equation}

{\bf Step-2 : (Control based bounds for the state and adjoint errors)} The proof strategy and results  of this step are the same as those in {\bf Step~2} of Theorem~\ref{v error equivalence}, part~\textit{(i)}.

{\bf Step-3 : (Upper bound on the control)}
Definition of $\widetilde{\bf u}$ from \eqref{2.21}, \cite[Theorem 7.1.2]{kesavann}, and \eqref{cop} result in
\begin{equation*}
	(\alpha\widetilde{\bf u}+\bar{\pee}_h, {\bf z}-\widetilde{\bf u})\geq 0\mbox{ and }(\alpha\bar{\bf u}+\bar{\pee},{\bf z}-\bar{\bf u})\geq 0\mbox{ for all }{\bf z}\in {\bf U}_{\rm ad}. 
\end{equation*}
Choice of  ${\bf z}=\bar{\bf u}\in {\bf U}_{\rm ad}\mbox{ and }{\bf z}=\widetilde{\bf u}\in {\bf U}_{\rm ad}$ in the above mentioned equation, respectively with some algebraic calculations give
\begin{equation*}
	\alpha\|\bar{\bf u}-\widetilde{\bf u}\|^2\leq (\bar{\bf p}-\bar{\bf p}_h, \widetilde{\bf u}-\bar{\bf u})=(\widetilde{\bf p}-\bar{\bf p},\bar{\bf u}-\widetilde{\bf u})+(\widetilde{\bf p}-\bar{\bf p}_h,\widetilde{\bf u}-\bar{\bf u}). 
\end{equation*}
Choice of ${\bf v}=\widetilde{\pee}-\bar{\bf p}$ in \eqref{2.24} and $q=\bar{r}-\widehat{r}$ in \eqref{2.12} give
\begin{equation*}
	\alpha\|\bar{\bf u}-\widetilde{\bf u}\|^2\leq a(\bar{\bf y}-\widehat{\bf y},\widetilde{\bf p}-\bar{\bf p})+(\widetilde{\bf p}-\bar{\bf p}_h,\widetilde{\bf u}-\bar{\bf u})=a(\widetilde{\bf p}-\bar{\bf p},\bar{\bf y}-\widehat{\bf y})+(\widetilde{\bf p}-\bar{\bf p}_h,\widetilde{\bf u}-\bar{\bf u}). 
\end{equation*}
Again choose ${\bf v}=\widehat{\bf y}-\bar{\bf y}$ in \eqref{2.12} and $q=\bar{s}-\widetilde{s}$ in \eqref{2.24} to get
\begin{equation*}
	\alpha\|\bar{\bf u}-\widetilde{\bf u}\|^2\leq(\bar{\bf y}-\bar{\bf y}_h,\widehat{\bf y}-\bar{\bf y})+(\widetilde{\bf p}-\bar{\bf p}_h,\widetilde{\bf u}-\bar{\bf u}).
\end{equation*}
Some algebraic calculations yield
\begin{equation*}
	\alpha\|\bar{\bf u}-\widetilde{\bf u}\|^2\leq (\bar{\bf y}-\widehat{\bf y},\widehat{\bf y}-\bar{\bf y}_h)+\|\widehat{\bf y}-\bar{\bf y}_h\|^2+(\widetilde{\bf p}-\bar{\bf p}_h,\widetilde{\bf u}-\bar{\bf u}).
\end{equation*}
By applying the Cauchy-Schwarz inequality, the triangle inequality, and \eqref{3.18}, we obtain
\begin{equation*}
 \alpha\|\bar{\bf u}-\widetilde{\bf u}\|^2\leq C_{\rm P}^2\|\bar{\bf u}-\widetilde{\bf u}\|(\|\widehat{\y}-\widetilde{\y}\|+\|\widetilde{\y}-\yhb\|)+2(\|\widehat{\y}-\widetilde{\y}\|^2+\|\widetilde{\y}-\yhb\|^2)+\|\widetilde{\bf p}-\bar{\bf p}_h\|\|\widetilde{\bf u}-\bar{\bf u}\|.
\end{equation*}
Choice of ${\bf v}=\widehat{\y}-\widetilde{\y}$ and $q=\widehat{r}-\widetilde{r}$ in \eqref{332} along with the ellipticity of bilinear form $a$ leads to
\begin{equation*}
	\alpha\|\bar{\bf u}-\widetilde{\bf u}\|^2\leq C_{\rm P}^2\|\bar{\bf u}-\widetilde{\bf u}\|\|\widetilde{\bf u}-\bar{\bf u}_h\|+C_{\rm P}^2\|\widetilde{\bf y}-\bar{\bf y}_h\|\|\bar{\bf u}-\widetilde{\bf u}\|+2\|\widetilde{\bf u}-\bar{\bf u}_h\|^2+2\|\widetilde{\bf y}-\bar{\bf y}_h\|^2+\|\widetilde{\bf p}-\bar{\bf p}_h\|\|\widetilde{\bf u}-\bar{\bf u}\|. 
\end{equation*}
Applying Young’s inequality with $a=C_{\rm P}^2\|\bar{\bf u}-\widetilde{\bf u}\|$, $b=\|\bar{\bf u}-\widetilde{\bf u}\|$, and $\epsilon={3C_{\rm P}^4}/{\alpha}$, with $a=C_{\rm P}^2\|\widetilde{\bf y}-\bar{\bf y}_h\|$, $b=\|\bar{\bf u}-\widetilde{\bf u}\|$, and $\epsilon={\alpha}/{3}$, and with $a=\|\widetilde{\bf p}-\bar{\bf p}_h\|$, $b=\|\widetilde{\bf u}-\bar{\bf u}\|$, and $\epsilon={\alpha}/{3}$ in the above equations, together with some algebraic manipulations, yields
\begin{equation}
	\label{3.22}
	\|\bar{\bf u}-\widetilde{\bf u}\|\leq C_{20}(\|\widetilde{\bf u}-\bar{\bf u}_h\|+\|\widetilde{\bf y}-\bar{\bf y}_h\|+\|\widetilde{\bf p}-\bar{\bf p}_h\|)
\end{equation}
with $C_{20}:=\max\{\sqrt{\frac{2}{\alpha}(\frac{3C_{\rm P}^4}{\alpha}+2)},\frac{\sqrt{3}}{\alpha}\}$. Triangle inequality, \eqref{3.22}, and Lemma \ref{discrete sobolev}{\it (i)} give
\begin{equation}
	\label{3.23}
	\wcerr\leq C_{21}(\normone+\wauxsterr+\wauxaderr)\mbox{ with }C_{21}:=C_{\rm PJ}(1+C_{20}).
\end{equation}

{\bf Step-4 : (Complete upper and lower bounds)} The upper bound is obtained by combining the results of Steps 1–3, as in the previous theorem. For the lower bound, an additional term appears compared to the previous result, which is estimated as follows:
Applying triangle inequality, by introducing $\bar{\bf u}$, using \eqref{defu} and \eqref{2.21}, and Lipschitz continuity of projection operator \cite[Proposition 2.5, Step-3]{KKRK2014}, $\Pi_{[{\bf u}_a,{\bf u}_b]}$,(with Lipschitz constant $\alpha^{-1}$) gives
\begin{align*}
	\|\widetilde{\bf u}-\bar{\bf u}_h\|=\|\widetilde{E}_{\rm C}\| \leq C_{\rm d3}(\|E_{\rm C}\|+\enorm{\bar{\bf p}-\bar{\bf p}_h}_{\rm pw})\mbox{ and }\|\widetilde{E}_{\rm C}\| \leq C_{\rm d4}(\|E_{\rm C}\|+\|\bar{\bf p}-\bar{\bf p}_h\|).
\end{align*}
This concludes the proof.
\end{proof}
{Proof of Theorem \ref{d error equivalence}{\it (iii)-(iv)}}
follows arguments analogous to those used in the proof of Theorem \ref{v error equivalence} {\it(iii)-(iv)}. Whenever discrete control terms arise, we refer to the proof of Theorem \ref{d error equivalence}{\it (i)-(ii)}. For completeness, the full proof is provided in the Appendix \ref{appendix}.

\noindent The error equivalence results derived in this section will be instrumental in proving the a priori and a posteriori error estimates in Section~\ref{error control}.
\section{Error estimates}
\label{error control}
\subsection{A priori error estimates}
This section aims to study explicit a priori error estimates for the model problem \eqref{nfunc}-\eqref{nst}. The main idea used to prove the a priori estimates for the optimal control problem is to combine the error equivalence result from Theorem \ref{v error equivalence} and \ref{d error equivalence} with the standard a priori estimates known for the Stokes problem. 

For given ${\bf g}\in
{\bf L}^2(\Om)$, let $(\y,r)\in{\bf V}\times Q$ and $(\y_h,r_h)\in ({\bf V}_h,Q_h)$ solves the continuous and discrete formulations of the Stokes equations given by
\begin{equation}
	\label{4.1}
	\begin{cases}
		a(\y,\textbf{v})-b(\textbf{v},r)=({\bf g},\textbf{v}) &\mbox{ for all } \textbf{v}\in V,\\
		\quad\quad\quad\hspace{0.45cm} b(\y,q)=0 &\mbox{ for all }q\in Q,
	\end{cases}
\end{equation}
and 
\begin{equation}
	\label{4.2}
	\begin{cases}
		a_h(\y_h,\textbf{v}_h)-b_h(\textbf{v}_h,r_h)=({\bf g},\textbf{v}_h) &\mbox{ for all }\textbf{v}_h\in V_h,\\
		\quad\quad\quad\quad\hspace{0.6cm} b_h(\y_h,q_h)=0 &\mbox{ for all }q_h\in Q_h.
	\end{cases}
\end{equation}
\noindent It is well known that the solution of the above problems belongs to $({\bf V}\cap {\bf H}^{1+s}(\Om)\times Q\cap H^{s}(\Omega))$. Here $s\in (0,1]$ is the elliptic regularity index.\\

\begin{lemma}(Interpolation operator)
	\label{interpolation}
	Let $\widehat{\T}$ be a shape regular refinement of $\T$. The nonconforming interpolation operator $ I_h : {\bf V}+{\widehat{\bf V}}_h\to {\bf V}_h$ and $\widehat{I}_h : {\bf V}\to {\widehat{\bf V}}_h$ satisfy the following properties
	\begin{enumerate}[label=(\roman*)]
		\item The integral mean value property that is $\n_h I_h =\Pi_{0} \n_h.$
		\item For all ${\bf v}\in {\bf H}^{1+s}(\Om)$ with $s\in [0,1]$, $I_h$ satisfy the following estimate. 
        \begin{equation*}
            \|(1-I_h){\bf v}\|_{K}+h_{K}\|(1-I_h){\bf v}\|_{1,K}\leq C_{\rm I}h_{K}^{1+s}\|{\bf v}\|_{1+s,K}. 
        \end{equation*}
		\item The orthogonality condition, $a_h({\bf v}_h,{\widehat{\bf v}}_h-I_h{\widehat{\bf v}}_h)=0$ for all ${\bf v}_h\in {\bf V}_h$ and ${\widehat{\bf v}}_h\in{\widehat{\bf V}}_h+{\bf V}$.
		\item $(1-I_h){\widehat{\bf v}}_h=0$ and $(\widehat{I}_h-I_h)({\bf v}+{\widehat{\bf v}}_h)=0$ in $\T\cap\widehat{\T}$ for all ${\bf v}\in {\bf V}$ and ${\widehat{\bf v}}_h\in {\widehat{\bf V}}_h$.
		\item For any given ${\bf v}_h\in {\bf V}_h$ and $q_h\in Q_h$,  $(\mbox{\rm div}_h(I_h{\bf v}_h-{\bf v}_h), q_h)=0$ holds true. 
	\end{enumerate}
\end{lemma}

\begin{proof}
For the explicit details of proof of properties {\it(i)} to \textit{(iv)} refer  \cite[Lemma 13]{MR3790080}, \cite[Lemma 36.1]{MR4269305}, \cite[Section 4.3]{CCSP20}, and \cite[Section 4.4]{CCSP20}, respectively. 
The proof of property \textit{(v)} is included below. Use the definition of discrete divergence operator, integral mean value property as mentioned in {\it (i)} (component wise) and the property of projection operator to obtain
\begin{equation*}
	(\mbox{div}_h I_h {\bf v}_h,q_h)=\sum_{K\in\T}\int_{K}{\rm div} I_h {\bf v}_h q_h\,dx=\sum_{K\in\T}\int_{K}\Pi_{0}{\rm div} {\bf v}_h q_h\,dx=\sum_{K\in\T}\int_{K}{\rm div} {\bf v}_h q_h\,dx=({\rm div}_h {\bf v}_h,q_h).
\end{equation*}
\end{proof}
 
\begin{lemma}(Error estimate for the Stokes equation).
	\label{lemma 4.1}
	For a given ${\bf g}\in {\bf L}^2(\Om)$, let $({\bf y},r)\in {\bf V}\times Q$ and $({\bf y}_h,r_h)\in {\bf V}_h\times Q_h$ satisfy \eqref{4.1} and \eqref{4.2}, respectively. Then for ${\bf y}\in {\bf H}^{1+s}(\Om)$ and $r\in H^s(\Omega)$ with $s\in (0,1]$, the a priori error estimates stated below hold true.
    \begin{equation*}
		\enorm{{\bf y}-{\bf y}_h}_{\pw}+\|r-r_h\|\leq C_{\rm{ae}}h^{s} (\|{\bf y}\|_{1+s}+\|r\|_{s}+\|(1-\Pi_{0}){\bf g}\|).
	\end{equation*}
    Further, $L^2$ error estimate reads
    \begin{equation*}
        \|{\bf y}-{\bf y}_h\|\leq C_{\rm ael} h^s(\enorm{({\bf y}-{\bf y}_h,r-r_h)}_{\rm pw}+\sum_{K\in\T}h_K\|(1-\Pi_0){\bf g}\|).
    \end{equation*}
\end{lemma}
These estimates for the Stokes system are based on the best approximation result from \cite[Theorem 3.1]{MR3349689} along with standard interpolation estimates stated in Lemma \ref{interpolation} and \cite[Proposition 1.135]{ErnJLU_2004}. For $L^2$-estimates we can refer \cite[Theorem 4.1]{MR3194820}.


\subsubsection{A priori error estimates of optimal control problem (variational approach)}
\begin{theorem}
	\label{Theorem-4.1}
	Let $(\bar{\bf y},\rb,\bar{\bf p},\sbar,\ub)$ be the solution of \eqref{cstate}-\eqref{cop} such that $(\bar{\bf y},\rb,\bar{\bf p},\sbar)\in [({\bf V}\cap {\bf H}^{1+s}(\Om)\times Q\cap H^{s}(\Omega))]^2$ and $\ub\in {\bf U}_{\rm{ad}}$. Let $(\bar{\bf y}_h,\rhb,\bar{\bf p}_h,\shbar,\uhb)\in {\bf V}_h\times Q_h\times {\bf V}_h\times Q_h\times {\bf U}_{\rm ad}$ solve \eqref{dstate}-\eqref{dop}. Then the a priori error estimates mentioned below hold true
	\begin{equation*}
		\|E_{\rm C}\|+\enorm{E_{\rm S}}_{\rm pw}+\enorm{E_{\rm A}}_{\rm pw}\leq C_{\rm{AE}} h^s (\|{\bf f}\|+|\Om|^{1/2}\max\{|\U_a|,|\U_b|\}+\|{\bf y}_d\|).
	\end{equation*}
\end{theorem}
\begin{proof}
\noindent Let $(\widetilde{\y},\widetilde{r},\yhb,\rhb)\in {\bf V}\times Q\times {\bf V}_h\times Q_h$ solve \eqref{cvstate} and \eqref{dstate} and $(\widetilde{\pee},\widetilde{s},\phb,\shbar)\in {\bf V}\times Q\times {\bf V}_h\times Q_h$ solve \eqref{cvadjoint} and \eqref{dadjoint}. Theorem \ref{v error equivalence}{\it (i)} gives
\begin{equation}
\label{4.3}
	\begin{aligned}
		\|E_{\rm C}\|+\enorm{E_{\rm S}}_{\rm pw}+\enorm{E_{\rm A}}_{\rm pw}\leq C_{\rm{c1}} (\enorm{\widetilde{E}_{\rm S}}_{\rm pw}+\enorm{\widetilde{E}_{\rm A}}_{\rm pw}).
	\end{aligned}
\end{equation}
The standard a priori results mentioned in Lemma \ref{lemma 4.1} shows
\begin{equation}
\label{4.4}
\enorm{\widetilde{E}_{\rm S}}_{\rm pw}+\enorm{\widetilde{E}_{\rm A}}_{\rm pw}\leq C_{\rm ae}h^s (\|\widetilde{\y}\|_{1+s}+\|\widetilde{r}\|_{s}+\|\widetilde{\pee}\|_{1+s}+\|\widetilde{s}\|_{s}+
        \|(1-\Pi_{0})({\bf f}+\uhb)\|+\|(1-\Pi_{0})(\yhb-{\bf y}_d)\|).
\end{equation}
Regularity results for standard Stokes problem and stability of projection operator $\Pi_{0}$ results in
\begin{align*}
	&\|\widetilde{\y}\|_{1+s}+\|\widetilde {r}\|_s\leq C_{\rm{reg}} \|{\bf f}+\bar{\U}_h\|, \|\widetilde{\pee}\|_{1+s}+\|\widetilde{s}\|_s\leq C_{\rm{reg}} \|\bar{\y}_h-\y_d\|,\\
    &\|(1-\Pi_{0})({\bf f}+\uhb)\|\leq(1+C_{\rm pi})\|{\bf f}+\uhb\|,\mbox{ and }\|(1-\Pi_{0})(\bar{{\bf y}}_h-{\bf y}_d)\|\leq(1+C_{\rm pi})\|\bar{{\bf y}}_h-{\bf y}_d\|.
\end{align*}
Substitution of above estimate in \eqref{4.4} and use of \eqref{4.3} with $C_{\rm{REG}}:=(1+C_{\rm pi})+C_{\rm reg}$ give
\begin{equation}
	\|E_{\rm C}\|+\enorm{E_{\rm S}}_{\rm pw}+\enorm{E_{\rm A}}_{\rm pw}\leq C_{\rm{c1}}C_{\rm{ae}} C_{\rm{REG}} h^s (\|\textbf{f}+\uhb\|+\|\yhb-{\bf y}_d\|).
\end{equation}
Using $\|\yhb\|\leq C_{\rm{dP}} \enorm{\yhb}_{\pw}\leq C_{\rm{dP}}\|{\bf f}+\uhb\|$, $\|\uhb\|\leq |\Om|^{1/2}\max\{|\U_a|,|\U_b|\}$, and triangle inequality finally leads to
\begin{equation*}
	\begin{aligned}
		\|E_{\rm C}\|+\enorm{E_{\rm S}}_{\rm pw}+\enorm{E_{\rm A}}_{\rm pw}\leq  C_{\rm{AE}} h^s (\|{\bf f}\|+|\Om|^{1/2}\max\{|\U_a|,|\U_b|\}+\|\y_d\|)
	\end{aligned}
\end{equation*}
with $C_{\rm{AE}}:=(1+C_{\rm{dp}})C_{\rm{c1}}C_{\rm{ae}} C_{\rm{REG}}$. This completes the proof of desired estimate.
\end{proof}

\subsubsection{A priori error estimates for optimal control problem (discretised approach).} A priori error estimates in case of discretised control approach are similar to that of variational approach, but for the sake of completeness it has been mentioned here. 
\begin{theorem}
\label{theorem 4.2}
	Let $(\bar{\bf y},\rb,\bar{\bf p},\sbar,\ub)$ be the solution of \eqref{cstate}-\eqref{cop} such that $(\bar{\bf y},\bar{r},\bar{\bf p},\bar{s})\in [({\bf V}\cap {\bf H}^{1+s}(\Om)\times Q\cap H^{s}(\Omega))]^2$ and $\ub\in {\bf U}_{\rm{ad}}$. Let $(\bar{\bf y}_h,\rhb,\bar{\bf p}_h,\shbar,\uhb)\in {\bf V}_h\times Q_h\times {\bf V}_h\times Q_h\times {\bf U}_{h,\rm{ad}}$ solve \eqref{dstate}-\eqref{dop}. Then the a priori error estimates mentioned below hold true
	\begin{equation*}
		\|E_{\rm C}\|+\enorm{E_{\rm S}}_{\rm pw}+\enorm{E_{\rm A}}_{\rm pw}\leq C_{\rm{DAE}} h^s (\|{\bf f}\|+|\Om|^{1/2}\max\{|\U_a|,|\U_b|\}+\|{\bf y}_d\|).
	\end{equation*}
\end{theorem}
\begin{proof}
	From Theorem \ref{d error equivalence}{\it (i)} gives
	\begin{equation}
		\label{4.7}
		\begin{aligned}
			\|E_{\rm C}\|+\enorm{E_{\rm S}}_{\rm pw}+\enorm{E_{\rm A}}_{\rm pw}\leq C_{\rm{\rm d1}}
		    (\|\widetilde{E}_{\rm C}\|+\enorm{\widetilde{E}_{\rm S}}_{\rm pw}+\enorm{\widetilde{E}_{\rm A}}_{\rm pw}).
		\end{aligned}
	\end{equation}
Here observe that it is sufficient for us to bound the term $\|\widetilde{E}_{\rm C}\|$. Definition of $\uhb$ and $\widetilde{\bf u}$ from \eqref{defuh} and \eqref{2.21} respectively give
\begin{equation}
	\label{4.8}
	\|\widetilde{E}_{\rm C}\|\leq \Bigl\lVert\Pi_{[\U_a,\U_b]}\left(-\frac{\bar{\pee}_h}{\alpha}\right)-\Pi_{[\U_a,\U_b]}\Pi_{0}\left(-\frac{\bar{\pee}_h}{\alpha}\right)\Bigr\rVert\leq \Bigl\lVert\left(-\frac{\bar{\pee}_h}{\alpha}\right)-\Pi_{0}\left(-\frac{\bar{\pee}_h}{\alpha}\right)\Bigr\rVert \leq h\alpha^{-1} \enorm{\bar{\pee}_h}_{\pw}.
\end{equation}
The above estimate follows from Lipschitz continuity of $\Pi_{[\U_a,\U_b]}$ and approximation property of $\Pi_0$ from \cite[Proposition 1.135]{ErnJLU_2004}.
Choice of $\textbf{v}_h=\bar{\pee}_h$, $q_h=\shbar$ in equation \eqref{dadjoint} along with ellipticity of bilinear form $a_h(\cdot,\cdot)$ leads to
\begin{equation}
	\label{4.9}
	\|\widetilde{E}_{\rm C}\|\leq h\alpha^{-1}\enorm{\phb}_{\pw}\leq h\alpha^{-1}\|\bar{\y}_h-\y_d\|\leq h\alpha^{-1}(C_{\rm dP}\|{\bf f}+\uhb\|+\|{\bf y}_d\|).
\end{equation}
 Therefore for $C_{\rm DAE}:=2\max\{C_{\rm AE},\alpha^{-1}C_{\rm dP}|\Om|^{(1-s)/2}, |\Om|^{(1-s)/2}\alpha^{-1}\}$, \eqref{4.7}, \eqref{4.9}, and Theorem \ref{Theorem-4.1} give the desired result.
\end{proof}

\subsection{A posteriori error estimates} 
\label{a pos}
The a posteriori error control for the optimal control problem \eqref{nfunc}-\eqref{nst} relies fundamentally on the error equivalence results presented in Theorem \ref{v error equivalence} and \ref{d error equivalence} along with the standard a posteriori estimates established for the Stokes problem \eqref{4.1}-\eqref{4.2}. These estimates for the Stokes problem are provided in Lemma \ref{lemma 4.2} below.

\begin{lemma}[Efficiency and reliability of estimator for Stokes problem]\cite[Section 3]{DEDR95}
\label{lemma 4.2}
For a given ${\bf g}\in{\bf L}^2(\Om)$, let $({\bf y},r)\in {\bf V}\times Q$ and $({\bf y}_h,r_h)\in {\bf V}_h\times Q_h$ satisfy \eqref{4.1} and \eqref{4.2}, respectively. Then, for $C_{\rm{rel}}, C_{\rm{eff}}>0$
\begin{equation*}
	C_{\rm{rel}}^{-2}\enorm{{\bf y}-{\bf y}_h}_{\pw}^2+\|r-r_h\|^2\leq \eta_{\rm{aux}}^2(\T)\leq C_{\rm{eff}}^2(\enorm{{\bf y}-{\bf y}_h}_{\pw}^2+\|r-r_h\|^2+\rm{osc}^2({\bf g},\T))\mbox{ with }
\end{equation*}
$\eta_{\rm{aux}}^2(\T):=\sum_{K\in\T}(h_{K}^2\|{\bf g}\|^2+\sum_{E\in\E_{K}}h_{K}\|[\n_h\bar{\bf y}_h\cdot t_E]\|^2)$ and $\rm{osc}^2({\bf g},\T):=\sum_{{\it K}\in\T} {\it h}_{\it K}^2\|(1-\Pi_{0}){\bf g}\|^2$.
\end{lemma}
\subsubsection{A posteriori error estimates for optimal control problem (variational approach)}
\begin{theorem}
\label{theorem 4.3}
	Let $(\bar{\bf y},\rb,\bar{\bf p},\sbar,\ub)\in {\bf V}\times Q\times {\bf V}\times Q\times {\bf U}_{\rm ad}$ be the solution of \eqref{cstate}-\eqref{cop}. Let $(\bar{\bf y}_h,\rhb,\bar{\bf p}_h,\shbar,\uhb)\in {\bf V}_h\times Q_h\times {\bf V}_h\times Q_h\times {\bf U}_{\rm ad}$ solve \eqref{dstate}-\eqref{dop}. Then there exists constants $C_{\rm{veff}},C_{\rm{vrel}}>0$ such that for $\eta^2(\T)=\eta_{\rm st}^2(\T)+\eta_{\rm adj}^2(\T)$ following estimates holds. 
	\begin{equation*}
		C_{\rm{vrel}}^{-2}(\|E_{\rm C}\|^2+\enorm{E_{\rm S}}_{\rm pw}^2+\enorm{E_{\rm A}}_{\rm pw}^2)\leq\eta^2(\T)\leq C_{\rm{veff}}^{2} (\|E_{\rm C}\|^2+\enorm{E_{\rm S}}_{\rm pw}^2+\enorm{E_{\rm A}}_{\rm pw}^2+ \rm{osc}^2(\T)).
	\end{equation*}
\end{theorem}
\begin{proof}
	Theorem \ref{v error equivalence}{\it (i)} gives
	\begin{equation}
		\label{4.10}
		\begin{aligned}
			\|E_{\rm C}\|^2+\enorm{E_{\rm S}}_{\rm pw}^2+\enorm{E_{\rm A}}_{\rm pw}^2\leq 4C_{\rm{\rm c1}}^2
			(\enorm{\widetilde{E}_{\rm S}}_{\rm pw}^2+\enorm{\widetilde{E}_{\rm A}}_{\rm pw}^2).
		\end{aligned}
	\end{equation}
Referring to Lemma \ref{lemma 4.2} leads to
\begin{equation*}
	\enorm{\widetilde{E}_{\rm S}}_{\rm pw}^2\leq C_{\rm rel}^2\eta_{\rm st}^2(\T)\mbox{ and }\enorm{\widetilde{E}_{\rm A}}_{\rm pw}^2\leq C_{\rm{rel}}^2 \eta_{\rm adj}^2(\T).
\end{equation*}
Substitution of this in \eqref{4.10} leads to desired reliability as
\begin{equation*}
	C_{\rm{vrel}}^{-2}(\|E_{\rm C}\|^2+\enorm{E_{\rm S}}_{\rm pw}^2+\enorm{E_{\rm A}}_{\rm pw}^2)\leq \\ 
	\eta^2(\T)\mbox{ with }C_{\rm vrel}^2:=4C_{\rm c1}^2C_{\rm rel}^2.
\end{equation*}
 For proving efficiency, use Lemma \ref{lemma 4.2} to get
 \begin{equation*}
     \eta_{\rm st}^2(\T)\leq C_{\rm eff}^2(\enorm{\widetilde{E}_{\rm st}}_{\rm pw}^2+{\rm osc_{\rm st}^2(\T)})\mbox{ and }\eta_{\rm adj}^2(\T)\leq C_{\rm eff}^2(\enorm{\widetilde{E}_{\rm adj}}_{\rm pw}^2+{\rm osc_{\rm adj}^2(\T)}).
 \end{equation*}
Addition of above inequalities yields
 \begin{equation*}
     \eta^2(\T)=\eta_{\rm st}^2(\T)+\eta_{\rm adj}^2(\T)\leq C_{\rm eff}^2(\enorm{\widetilde{E}_{\rm st}}_{\rm pw}^2+\enorm{\widetilde{E}_{\rm adj}}_{\rm pw}^2+{\rm osc_{\rm adj}^2(\T)}+{\rm osc_{\rm adj}^2(\T)}).
 \end{equation*}
 Finally, refer to Theorem \ref{v error equivalence}{\it (i)} to get the desired efficiency as
 \begin{align*}
	\eta^2(\T)&\leq C_{\rm{veff}}^{2} (\|E_{\rm C}\|^2+\enorm{E_{\rm S}}_{\rm pw}^2+\enorm{E_{\rm A}}_{\rm pw}^2 + \rm{osc}^2(\T))
\end{align*}
 with $C_{\rm veff}:=C_{\rm eff}^2\max\{1,5C_{\rm c3}^2\}.$
 \end{proof}


\subsubsection{A posteriori error estimates for optimal control problem (discretised approach)} 
\begin{theorem}
\label{theorem 4.4}
	Let $(\bar{\bf y},\rb,\bar{\bf p},\sbar,\ub)\in {\bf V}\times Q\times {\bf V}\times Q\times {\bf U}_{\rm ad}$ be the solution of \eqref{cstate}-\eqref{cop}. Let $(\bar{\bf y}_h,\rhb,\bar{\bf p}_h,\shbar,\uhb)\in {\bf V}_h\times Q_h\times {\bf V}_h\times Q_h\times {\bf U}_{h,{\rm ad}}$ solve \eqref{dstate}-\eqref{dop} and $\widetilde{\bf u}_h$ be as defined in \eqref{2.21}. Then for $\eta^2(\T)=\eta_{\rm st}^2(\T)+\eta_{\rm adj}^2(\T)+\eta_{\rm C}^2(\T)$ there exists constants $C_{\rm{deff}},C_{\rm{drel}}>0$ such that
	\begin{equation*}
		C_{\rm{drel}}^{-2}(\|E_{\rm C}\|^2+\enorm{E_{\rm S}}_{\rm pw}^2+\enorm{E_{\rm A}}_{\rm pw}^2)\leq\eta^2(\T)\leq C_{\rm{deff}}^{2} (\|E_{\rm C}\|^2+\enorm{E_{\rm S}}_{\rm pw}^2+\enorm{E_{\rm A}}_{\rm pw}^2 + \rm{osc}^2(\T)).
	\end{equation*}
\end{theorem}
\begin{proof}
	Theorem \ref{d error equivalence}{\it (i)} gives
	\begin{equation}
		\label{4.12}
	\begin{aligned}
		\|E_{\rm C}\|^2+\enorm{E_{\rm S}}_{\rm pw}^2+\enorm{E_{\rm A}}_{\rm pw}^2\leq 5 C_{\rm d1}^2(\|\widetilde{E}_{\rm C}\|^2+\enorm{\widetilde{E}_{\rm S}}_{\rm pw}^2+\enorm{\widetilde{E}_{\rm A}}_{\rm pw}^2).
	\end{aligned}
   \end{equation}
Use of Lemma \ref{lemma 4.2} yields
$
	\enorm{\widetilde{E}_{\rm S}}_{\rm pw}^2\leq C_{\rm rel}^2\eta_{\rm st}^2(\T)\mbox{ and }\enorm{\widetilde{E}_{\rm A}}_{\rm pw}^2\leq C_{\rm{rel}}^2 \eta_{\rm adj}^2(\T).
$
Using \eqref{4.8} and the fact $\bar{\bf u}_h\in {\bf \mathcal{P}_0}(\T)$, we obtain
\begin{equation}
	\|\widetilde{E}_{\rm C}\|^2\leq h^2\alpha^{-2}\enorm{\bar{\pee}_h}_{\pw}^2\leq \sum_{K\in\T}h_{K}^2\|\n_h(\bar{\bf u}_h-\alpha^{-1}\bar{\bf p}_h)\|_{K}^2=\eta_{\rm C}^2(\T).
\end{equation}
Substituting these into \eqref{4.12} yields the desired reliability as
\begin{equation*}
	C_{\rm{drel}}^{-2}(\|E_{\rm C}\|^2+\enorm{E_{\rm S}}_{\rm pw}^2+\enorm{E_{\rm A}}_{\rm pw}^2)\leq
	\eta^2(\T)\mbox{ with }C_{\rm drel}^2:=5C_{\rm d1}^2\max\{1,C_{\rm rel}^2\}.
\end{equation*}
The efficiency proof follows the same steps as in the previous theorem; it remains only to establish a bound for the additional control estimator $\eta_{C}^2$. Using the discrete stability of $\bar{\bf p}_h$ and the fact that $\bar{\bf u}_h\in{\bf \mathcal{P}}_0(\T)$, we obtain
\begin{equation}
	\label{4.15}
	\begin{aligned}
	\eta_C^2(\T)&=\sum_{K\in\T}h_K^2\|\n_h(\uhb-\alpha^{-1}\phb)\|_{K}^2
		\leq \alpha^{-2}\sum_{K\in\T} h_{K}^2\|\bar{\y}_h-\y_d\|_{K}^2=\alpha^{-2}\sum_{K\in\T}\mu_{A,K}^2\leq \alpha^{-2}\eta_{\rm adj}^2(\T).
	\end{aligned}
\end{equation} 
Then, the efficiency of $\eta_{\rm adj}(\T)$  concludes the proof with $C_{\rm deff}^2=5 C_{\rm d2}^2C_{\rm eff}^2(1+\alpha^{-2})$.
\end{proof}

\section{Verification of axioms of adaptivity (Variational approach)}
\label{verification of variational axioms}
\noindent This section aims to verify the adaptivity axioms mentioned in Section \ref{adaptive convergence}. Recall from \eqref{estidef} that, for the variational formulation, the total error estimator is given by 
$\eta^2(\mathcal T)=\eta_{\mathrm{st}}^2(\mathcal T)+\eta_{\mathrm{adj}}^2(\mathcal T)$. 
%
  Recall that $({\bar{\bf y}}_h,{\bar{\bf p}}_h)$ and $({\widehat{\bar{\bf y}}}_h,{\widehat{\bar{\bf p}}}_h)$ are the solutions of \eqref{dstate}-\eqref{dop} on $\T$ and $\widehat{\T}$, respectively, and ${\bf d}^2(\T,\widehat{\T})$ is defined as
\begin{equation*}
{\bf d}^2(\T,\widehat{\T}):=\enorm{\widehat{\bar{\bf y}}_h-\bar{\bf y}_h}_{\rm pw}^2+\enorm{\widehat{\bar{\bf p}}_h-\bar{\bf p}_h}_{\rm pw}^2. 
\end{equation*}
\begin{lemma}(Discrete jump control)\cite[Lemma 5.2]{CCRH17}
	\label{discrete jump control}
	For $k\in\mathbb{N}$ there exists $C_{\rm{djc}}>0$ depending on the shape regularity of $\T$ such that, for all $g\in \mathcal{P}_{k}(\T)$, following estimate is true
	\begin{equation}
		\sum_{K\in\T}|K|^{1/2}\sum_{E\in\E_{K}} \|[g]\|_{E}^2 \leq C_{\rm{djc}}^2\|g\|^2.
	\end{equation}
\end{lemma}

\subsection{Proof of (A1)}
Recall that the complete estimator over triangulation $\T$ and $\widehat{\T}$ is denoted as $\eta$ and $\widehat{\eta}$, respectively. Recall from Section \ref{adaptive convergence}, \textbf{(A1)} states
\begin{equation*}
    |\widehat{\eta}(\T\cap\widehat{\T})-\eta(\T\cap\widehat{\T})|\leq \Lambda_1 {\bf d}(\T,\widehat{\T}).
\end{equation*}
\begin{proof}
	 Reverse triangle inequality shows
	\begin{equation}
		\label{5.2}
		\begin{aligned}
			|\widehat{\eta}(\T\cap\widehat{\T})-\eta(\widehat{\T}\cap\T)|^2\leq &\sum_{K\in\T\cap\widehat{\T}}\big((\widehat{\mu}_{S,K}-\mu_{S,K})^2+(\widehat{\mu}_{A,K}-\mu_{A,K})^2+(\widehat{\rho}_{S,\E(K)}-\rho_{S,\E(K)})^2\\
			&+(\widehat{\rho}_{A,\E(K)}-\rho_{A,\E(K)})^2\big).
		\end{aligned}
	\end{equation}
Reverse triangle inequality, definition of estimator from Section \ref{adaptive convergence}, and Lemma \ref{discrete sobolev}{\it (ii)} give
\begin{align*}
	(\widehat{\mu}_{S,K}-\mu_{S,K})^2&\leq h_{K}^2 \|\widehat{\bar{\U}}_h-\bar{\U}_h\|_{K}^2\mbox{ and }(\widehat{\mu}_{A,K}-\mu_{A,K})^2\leq h_{K}^2 \|\widehat{\bar{\y}}_h-\bar{\y}_h\|_{K}^2\leq h_{K}^2 C_{\rm{PI}}^2 \enorm{\widehat{\bar{\y}}_h-\bar{\y}_h}_{\rm{pw}}^2.
\end{align*}
Taking summation over $K\in \T\cap\widehat{\T}$ and using the Lipschitz continuity of the projection operator $\Pi_{[\U_a,\U_b]}$ with the Lipschitz constant $\alpha^{-1}$ yield
\begin{equation}
	\label{5.3}
	\begin{aligned}
		\sum_{K\in\T\cap\widehat{\T}}(\widehat{\mu}_{S,K}-\mu_{S,K})^2+(\widehat{\mu}_{A,K}-\mu_{A,K})^2 \leq C_{\rm{PI}}^2 h^2 C_{\rm m\pi}^2 (\enorm{\widehat{\bar{\pee}}_h-\bar{\pee}_h}_{\rm{pw}}^2+\enorm{\widehat{\bar{\y}}_h-\bar{\y}_h}_{\rm{pw}}^2)
	\end{aligned}
\end{equation}
with $C_{\rm m\pi}^2:=\max\{1,\alpha^{-2}\}$. Reverse triangle inequality and Lemma \ref{discrete jump control} leads to
\begin{equation}
\label{edge1}
	\begin{aligned}
		\sum_{K\in\T\cap\widehat{\T}}(\widehat{\rho}_{S,\E(K)}-\rho_{S,\E(K)})^2&\leq \sum_{K\in\T\cap\widehat{\T}} h_{K} \sum_{E\in\E_{K}} \|[\n_h(\widehat{\bar{\y}}_h-\bar{\y}_h) \tau_E]\|_{E}^2\leq C_{\rm{djc}}^2\enorm{\widehat{\bar{\y}}_h-\bar{\y}_h}_{\pw}^2
	\end{aligned}
\end{equation}
and
\begin{equation}
\label{edge2}
	\begin{aligned}
		\sum_{K\in\T\cap\widehat{\T}}(\widehat{\rho}_{A,\E(K)}-\rho_{A,\E(K)})^2&\leq \sum_{K\in\T\cap\widehat{\T}} h_{K} \sum_{E\in\E_{K}} \|[\n_h(\widehat{\bar{\pee}}_h-\bar{\pee}_h) \tau_E]\|_{E}^2\leq C_{\rm{djc}}^2\enorm{\widehat{\bar{\pee}}_h-\bar{\pee}_h}_{\pw}^2.
	\end{aligned}
\end{equation}

Substituting \eqref{5.3}, \eqref{edge1}, and \eqref{edge2} in \eqref{5.2} results in ${\bf (A1)}$ with $\Lambda_1^2:= |\Omega|C_{\rm{PI}}^2C_{m\pi}^2+C_{\rm{djc}}^2$.
\end{proof}

\subsection{Proof of (A2)}
Recall \textbf{(A2)} from Section \ref{adaptive convergence}
\begin{equation*}
		\widehat{\eta}(\widehat{\T} \setminus \T)\leq \rho \eta(\T \setminus \widehat{\T})+\Lambda_2{\bf d} (\T,\widehat{\T})\mbox{ with }\rho\in(0,1).
	\end{equation*} 
\begin{proof}
For $K\in\T$ define $\widehat{T}(K):=\{T\in \widehat{\T} : T\subset K\}$. Application of reverse triangle inequality gives
\begin{equation*}
	\begin{aligned}
		\widehat{\eta}(\widehat{\T} \setminus \T)\leq& \biggl(\sum_{K\in\T\setminus\widehat{\T}}\sum_{T\in\widehat{T}(K)}\big((\widehat{\mu}_{S,T}-\mu_{S,T})^2+(\widehat{\mu}_{A,T}-\mu_{A,T})^2+(\widehat{\rho}_{S,\E(T)}-\rho_{S,\E(T)})^2\\
		&+(\widehat{\rho}_{A,\E(T)}-\rho_{A,\E(T)})^2\big)\biggr)^{1/2}+\biggl(\sum_{K\in\T\setminus\widehat{\T}}\sum_{T\in\widehat{T}(K)}\big({\mu}_{S,T}^2+{\rho}_{S,\E(T)}^2+{\mu}_{A,T}^2+{\rho}_{A,\E(T)}^2\big)\biggr)^{1/2}=: S_1+S_2.\\
	\end{aligned}
\end{equation*}
As a consequence of \eqref{5.3}, \eqref{edge1}, and \eqref{edge2} over $\widehat{\T}\setminus\T$ one can observe that $S_1^2\leq \Lambda_1^2(\enorm{\widehat{\bar{\pee}}_h-\bar{\pee}_h}_{\rm{pw}}^2+\enorm{\widehat{\bar{\y}}_h-\bar{\y}_h}_{\rm{pw}}^2)$.
Since $T\in\widehat{T}(K)$, utilize $h_T= |T|^{1/2}\leq 2^{-1/2} |K|^{1/2}=2^{-1/2}h_{K}$ to get
\begin{equation}
\label{new5.6}
\begin{aligned}
	S_2^2&\leq \sum_{K\in\T\setminus\widehat{\T}}  2^{-1}({\mu}_{S,K}^2+\mu^2_{A,K})+\sum_{K\in\T\setminus\widehat{\T}}2^{-1/2}({\rho}_{S,\E(K)}^2+{\rho}_{A,\E(K)}^2)\\
    &\leq 2^{-1/2}\sum_{K\in\T\setminus\widehat{\T}}({\mu}_{S,K}^2+{\rho}_{S,\E(K)}^2+{\mu}_{A,K}^2+{\rho}_{A,\E(K)}^2).
    \end{aligned}
\end{equation}
Above estimate implies $S_2^2\leq 2^{-1/2}\eta^2(\T\setminus\widehat{\T})$.
Combination of bounds on $S_1$ and $S_2$ concludes the proof of \textbf{(A2)} with $\rho:= 2^{-1/4}$ and $\Lambda_2^2=\Lambda_1^2$.
\end{proof}
Observe that in \eqref{edge1}-\eqref{edge2} the bounds on the edge estimator terms are without $h$, whereas the bounds on volume estimators in equation \eqref{5.3} contains $h^2$. Using this, one can have the following Corollary. 
\begin{corollary}
	\label{corollary 5.1}
	Let the volume estimator with respect to $\widehat{\T}$ be $\widehat{\mu}^2 := \sum_{K\in\widehat{\T}}(\widehat{\mu}_{S,K}^2+\widehat{\mu}_{A,K}^2).$ Then for $\Lambda_0:= C_{m\pi}C_{\rm{PI}}$
	\begin{align}
		\label{5.11}
		|\widehat{\mu}(\T\cap\widehat{\T})-\mu(\T\cap\widehat{\T})|&\leq h\Lambda_0 {\bf d}(\T,\widehat{\T})\hspace{0.2cm}\mbox{and}\hspace{0.2cm}\widehat{\mu}(\widehat{\T}\backslash \T)\leq 2^{-1/2}\mu(\T\backslash\widehat{\T})+h\Lambda_0{\bf d}(\T,\widehat{\T}).
	\end{align}
\end{corollary}
\begin{proof}
    The proof follows from \eqref{5.3} and \eqref{new5.6}.
\end{proof}
\subsection{Proof of (A3)}
\label{a3 proof v}
Recall \textbf{(A3)} from Section~\ref{adaptive convergence}, which states that
${\bf d}^2(\T,\widehat{\T}) \leq \Lambda_3 \eta^2(\mathcal{R}(\T,\widehat{\T}))$, 
where $\mathcal{R}(\T,\widehat{\T}) := \{ K \in \T : \text{there exists } T \in \T \setminus \widehat{\T} \text{ with } \mathrm{dist}(K,T)=0 \}$.
The proof of ${\bf (A3)}$ uses the Lemma \ref{interpolation} and certain properties of companion operator as mentioned below in Lemma \ref{companion}.
\begin{lemma}(Companion/Enrichment operator)
	\label{companion}
	For any ${\bf v}_h\in {\bf V}_h$, the companion or enrichment operator $J : {\bf V}_h\to{\bf V}$ satisfy the following properties. 
	\begin{enumerate}[label=(\roman*)]
		\item For all ${\bf v}_h\in {\bf V}_h$
        \begin{align*}
            h_{K}^{-2}\|(1-J){\bf v}_h\|_{K}^{2}+\|\n_h(1-J){\bf v}_h\|_{K}^{2}\leq C_{\rm{J}}\sum_{E\in\omega(K)}h_{K}\|[\n_h{\bf v}_h]_E t_E\|_{E}^2\leq \Lambda_{\rm{J}}\min_{{\bf v}\in{\bf V}}\|\n_h({\bf v}_h-{\bf v})\|_{\omega(K)}^2.
        \end{align*}
		\item Any ${\bf v}_h\in{\bf V}_h$ and ${\widehat{\bf v}}_h^*:= \widehat{I}_h J {\bf v}_h\in {\widehat{\bf V}}_h$ satisfy
		\begin{equation*}
			\|\n_h({\widehat{\bf v}}_h^*-{\bf v}_h)\|^2\leq C_{\rm{IJ}}^2 \sum_{K\in\mathcal{R(\T,\widehat{\T})}} h_{K}\sum_{E\in\E_{K}}\|[\n_h {\bf v}_h]_Et_E\|_{E}^2.
		\end{equation*}
		\item $\Pi_{0}(1-J){\bf v}_h=0$ and $\Pi_{0} \nabla_h(1-J){\bf v}_h=0$ for all ${\bf v}_h\in {\bf V}_h$.
	\end{enumerate}
\end{lemma}
\noindent For the details of the proof refer \cite[Theorem 5.1]{CMRC12}, \cite[Theorem 3.2]{CCSP20}, and \cite[Proposition 2.3]{CGS2015}.

\begin{proof}[Proof of {\bf (A3)}]
	Use equation \eqref{3.11}, \eqref{3.12}, and \eqref{c15} to get the bound for the discrete error as 
	\begin{equation}
		\label{5.8}
		\2+\3\leq C_{\Pi} (\6+\7)
	\end{equation}
with $C_{\Pi}:=\max\{C_9+C_9C_{14}+C_{\rm dP}, C_1+C_9+C_{\rm dP}\}$. 
Therefore it is sufficient to find the bound on the terms $\6$ and $\7$. Introducing $\widehat{\bar{\y}}_h^{*}=\widehat{I}_hJ {\bar{\bf y}}_h\in\widehat{\bf V}_h$ and $I_h\widehat{\widetilde{\bf y}}_h$ give
\begin{align*}
	\6^2&=a_h(\widehat{\widetilde{\y}}_h,\widehat{\widetilde{\y}}_h-\widehat{\bar{\y}}_h^{*})+a_h(\widehat{\widetilde{\y}}_h,\widehat{\bar{\y}}_h^{*}-\bar{\y}_h)-a_h(\bar{\y}_h,\widehat{\widetilde{\y}}_h-I_h\widehat{\widetilde{\y}}_h)-a_h(\bar{\y}_h,I_h\widehat{\widetilde{\y}}_h-\bar{\y}_h).
\end{align*}
Application of Lemma \ref{interpolation}{\it (iii)} gives $a_h(\bar{\y}_h,\widehat{\widetilde{\y}}_h-I_h\widehat{\widetilde{\y}}_h)=0$. Therefore introducing $\bar{\bf y}_h$ and re-arrangement of terms give
\begin{equation*}
	\6^2=a_h(\widehat{\widetilde{\y}}_h-\bar{\y}_h,\widehat{\bar{\y}}_h^{*}-\bar{\y}_h)+a_h(\bar{\y}_h,\widehat{\bar{\y}}_h^{*}-\bar{\y}_h)+a_h(\widehat{\widetilde{\y}}_h,\widehat{\widetilde{\y}}_h-\widehat{\bar{\y}}_h^{*})-a_h(\bar{\y}_h,I_h\widehat{\widetilde{\y}}_h-\bar{\y}_h).
\end{equation*}
Again applying Lemma \ref{interpolation}{\it (iii)} for $\bar{\bf y}_h=I_{h}\widehat{\bar{\bf y}}_h^{*}$ \cite[(C4)]{CCSP20}, we  observe $a_h(\bar{\y}_h,\widehat{\bar{\y}}_h^{*}-\bar{\y}_h)=0$. Then
\begin{equation}
	\label{5.9}
	\6^2=a_h(\widehat{\widetilde{\y}}_h-\bar{\y}_h,\widehat{\bar{\y}}_h^{*}-\bar{\y}_h)+a_h(\widehat{\widetilde{\y}}_h,\widehat{\widetilde{\y}}_h-\widehat{\bar{\y}}_h^{*})-a_h(\bar{\y}_h,I_h\widehat{\widetilde{\y}}_h-\bar{\y}_h).
\end{equation}
By utilizing  \eqref{dvstate} for $\widehat{\bf v}_h=\widehat{\widetilde{\y}}_h-\widehat{\bar{\y}}_h^{*}$ and \eqref{dstate} for ${\bf v}_h=I_h\widehat{\widetilde{\y}}_h-\bar{\y}_h$ we obtain
\begin{align}
\label{new5.10}
	a_h(\widehat{\widetilde{\y}}_h,\widehat{\widetilde{\y}}_h-\widehat{\bar{\y}}_h^{*})&=({\bf f}+\bar{\U}_h,\widehat{\widetilde{\y}}_h-\widehat{\bar{\y}}_h^{*})+b_h(\widehat{\widetilde{\y}}_h-\widehat{\bar{\y}}_h^{*},\widehat{\widetilde{r}}_h)\mbox{ and }\\
    \label{new5.11}
	a_h(\bar{\y}_h,I_h\widehat{\widetilde{\y}}_h-\bar{\y}_h)&=({\bf f}+\bar{\U}_h,I_h\widehat{\widetilde{\y}}_h-\bar{\y}_h)+b_h(I_h\widehat{\widetilde{\y}}_h-\bar{\y}_h,\bar{r}_h).
\end{align}
Use the definition of bilinear form $b_h(\cdot,\cdot)$, Lemma \ref{companion}{\it (ii)} to get $\widehat{\bar{\bf y}}_h^{*}=\widehat{I}_h J\bar{\bf y}_h$, and $\widehat{q}_h=\widehat{\widetilde{r}}_h$ in \eqref{dvstate}  
 to obtain
\begin{equation*}
    b_h(\widehat{\widetilde{\y}}_h-\widehat{\bar{\y}}_h^{*},\widehat{\widetilde{r}}_h)=\sum_{K\in\widehat{\T}}\int_{K}\widehat{\widetilde{r}}_h\,{\rm div}(\widehat{\widetilde{\y}}_h-\widehat{\bar{\y}}_h^{*})\,dx=\sum_{K\in\widehat{\T}}\int_{K}\widehat{\widetilde{r}}_h\,{\rm div}(\widehat{\widetilde{\y}}_h-\widehat{I}_h J{\bar{\y}}_h)\,dx=-\sum_{K\in\widehat{\T}}\int_{K}\widehat{\widetilde{r}}_h\,{\rm div}\widehat{I}_h J{\bar{\y}}_h\,dx.
\end{equation*}
Using above estimate, Lemma \ref{interpolation}{\it (i)}, and Lemma \ref{companion}{\it (iii)} give
\begin{equation*}
    b_h(\widehat{\widetilde{\y}}_h-\widehat{\bar{\y}}_h^{*},\widehat{\widetilde{r}}_h)=-\sum_{K\in\widehat{\T}}\int_{K}\widehat{\widetilde{r}}_h\,{\rm div}\widehat{I}_h J{\bar{\y}}_h\,dx=-\sum_{K\in\widehat{\T}}\int_{K}\widehat{\widetilde{r}}_h\,\widehat{\Pi}_0{\rm div} J{\bar{\y}}_h\,dx=-\sum_{K\in\widehat{\T}}\int_{K}\widehat{\widetilde{r}}_h\,\widehat{\Pi}_0{\rm div} {\bar{\y}}_h\,dx.
\end{equation*}
Since $\bar{\bf y}_h\in{\bf V}_h,$ we can observe ${\rm div}\bar{\bf y}_h|_K\in \mathcal{P}_0(K)$.Therefore, $\widehat{\Pi}_0{\rm div}\bar{\bf y}_h|_K=\Pi_0{\rm div}\bar{\bf y}_h|_K={\rm div}\bar{\bf y}_h|_K$. Using this and choosing $q_h=\Pi_0\widehat{\widetilde{r}}_h$ in \eqref{dstate} we can observe
\begin{equation*}
    b_h(\widehat{\widetilde{\y}}_h-\widehat{\bar{\y}}_h^{*},\widehat{\widetilde{r}}_h)=-\sum_{K\in\widehat{\T}}\int_{K}\widehat{\widetilde{r}}_h\,\widehat{\Pi}_0{\rm div} {\bar{\y}}_h\,dx=-\sum_{K\in\widehat{\T}}\int_{K}\widehat{\widetilde{r}}_h\,{\rm div} {\bar{\y}}_h\,dx=-\sum_{K\in\widehat{\T}}\int_{K}\Pi_0\widehat{\widetilde{r}}_h\,{\rm div} {\bar{\y}}_h\,dx=0.
\end{equation*}
The analogous calculations by choosing $q_h=\bar{r}_h$ in \eqref{dstate} , using Lemma \ref{interpolation}{\it (i)}, and by choosing $\widehat{q}_h=\bar{r}_h$ in \eqref{dvstate} yield 
\begin{equation*}
b_h(I_h\widehat{\widetilde{\y}}_h-\bar{\y}_h,\bar{r}_h)=\sum_{K\in\T}\int_{K}\bar{r}_h{\rm div}(I_h\widehat{\widetilde{\y}}_h-\bar{\y}_h)\,dx=\sum_{K\in\T}\int_{K}\bar{r}_h \Pi_0{\rm div} \widehat{\widetilde{\y}}_h\,dx=\sum_{K\in\T}\int_{K}\bar{r}_h {\rm div} \widehat{\widetilde{\y}}_h\,dx=0.
\end{equation*}
Therefore, last two estimates suggest that $b_h(\widehat{\widetilde{\y}}_h-\widehat{\bar{\y}}_h^{*},\widehat{\widetilde{r}}_h)=b_h(I_h\widehat{\widetilde{\y}}_h-\bar{\y}_h,\bar{r}_h)=0$. Substituting this in \eqref{new5.10}-\eqref{new5.11} and using the fact that $\bar{\bf y}_h=I_h\widehat{\bar{\bf y}}_h^*$ along with some elementary calculation leads to
\begin{equation*}
	a_h(\widehat{\widetilde{\y}}_h,\widehat{\widetilde{\y}}_h-\widehat{\bar{\y}}_h^{*})-a_h(\bar{\y}_h,I_h\widehat{\widetilde{\y}}_h-\bar{\y}_h)=({\bf f}+\bar{\U}_h,(1-I_h)(\widehat{\widetilde{\y}}_h-\widehat{\bar{\y}}_h^{*})).
\end{equation*}
Substituting this in \eqref{5.9} gives
\begin{equation}
	\6^2=a_h(\widehat{\widetilde{\y}}_h-\bar{\y}_h,\widehat{\bar{\y}}_h^{*}-\bar{\y}_h)+({\bf f}+\bar{\U}_h,(1-I_h)(\widehat{\widetilde{\y}}_h-\widehat{\bar{\y}}_h^{*})).
\end{equation}
As a consequence of Lemma \ref{interpolation}{\it (iv)} we observe $\|(1-I_h)(\widehat{\widetilde{\y}}_h-\widehat{\bar{\y}}_h^{*})\|_{\T\cap\widehat{\T}}=0$. Using this, Cauchy-Schwartz inequality, and Lemma \ref{interpolation}{\it (ii)} result in
\begin{align*}
	\6^2\leq \enorm{\widehat{\widetilde{\y}}_h-\bar{\y}_h}_{\pw}&\enorm{\widehat{\bar{\y}}_h^{*}-\bar{\y}_h}_{\pw}+C_{\rm{I}}\sum_{K\in \T\setminus\widehat{\T}}h_K\|\textbf{f}+\bar{\U}_h\|_{K}\|\n_h(\widehat{\widetilde{\y}}_h-\widehat{\bar{\y}}_h^{*})\|_{K}.
\end{align*}
A use of H\"{o}lder's inequality and Lemma \ref{companion}{\it (ii)} yield
\begin{equation}
	C_{\rm{SI}}^{-1}\6^2\leq (\enorm{\widehat{\widetilde{\y}}_h-\bar{\y}_h}_{\pw}+\enorm{\widehat{\widetilde{\y}}_h-\widehat{\bar{\y}}_h^{*}}_{\pw})\eta_{\rm{st}}(\mathcal{R}(\T,\widehat{\T}))\mbox{ with }C_{\rm{SI}}:=\max\{C_{\rm{IJ}},C_{\rm{I}}\}.
\end{equation}
Triangle inequality and Lemma \ref{companion}{\it (ii)} gives
\begin{align*}
	C_{\rm{S2}}^{-1}\6^2\leq (\enorm{\widehat{\widetilde{\y}}_h-\bar{\y}_h}_{\pw}\eta_{\rm{st}}(\mathcal{R}(\T,\widehat{\T}))+\eta_{\rm{st}}^2(\mathcal{R}(\T,\widehat{\T})))\mbox{ with }C_{\rm{S2}}:= C_{\rm{SI}}\max\{2,C_{\rm{IJ}}\}.
\end{align*}
Young's inequality with $\epsilon=C_{\rm S2}$ shows that $\6^2\leq C_{\rm{S}}\eta_{\rm st}^2(\mathcal{R}(\T,\widehat{\T}))$ with $C_{\rm{S}}:=C_{\rm{S2}}(2+C_{\rm{S2}})$. The analogous calculations leads to the estimate for $\7^2\leq C_{\rm{A}}\eta_{\rm adj}^2(\mathcal{R}(\T,\widehat{\T}))$ with $C_{\rm A}:=(C_{\rm IJ}+C_{\rm I})^2+2C_{\rm I}C_{\rm IJ}$. Therefore \eqref{5.8} concludes the proof of \textbf{(A3)} with $\Lambda_3^2=2C_{\Pi}^2\max\{C_{\rm{S}},C_{\rm A}\}$.
\end{proof}

\subsection{Proof of (A4)}
\label{proof of a4v}
The proof of $\textbf{(A4)}$ follows from $(\textbf{A4}_{\varepsilon})$ as mentioned in \cite[Theorem 3.1]{CCRH17}. This section aims to prove $(\textbf{A4}_{\varepsilon})$. Assume $0<\varepsilon<8C_{\rm rl}^2$ and choose maximum $\delta$ such that following estimate holds
\begin{equation}
	\label{constants}
	\max\{\Delta_1,\Delta_2 ,\Delta_3 ,\Delta_4 \}\leq \varepsilon,
\end{equation}
with $\Delta_i$ for $i=1,2,3,4$ as defined in Table \ref{tab:constants-horizontal}. Moreover, the other constants mentioned in Table \ref{tab:constants-horizontal} will be used in the proof of Theorem \ref{A4ep} mentioned below. 
\begin{table}[hbt!]
\footnotesize
\centering
\renewcommand{\arraystretch}{1.4}
\begin{tabular}{|c|c|c|c|c|}
\hline
\textbf{Constant} & \textbf{Definition} & \textbf{Constant} & \textbf{Definition} \\
\hline
$C_{\rm A41}$ & $\max\{2C_{\rm I},\alpha^{-1}\}$ & $C_{\rm A42}$ & $\max\{C_1C_7, (1+C_1)\}$ \\
\hline
$C_{\rm A43}$ & $4C_{\rm A42}C_{\rm ael}\max\{1, C_{\rm rst}, C_{\rm sadj}\}$ & $C_{\rm A44}$ & $2(1+\alpha^{-1}C_{\rm PJ})\max\{1, C_{\rm cst}, C_{\rm cad}C_{\rm PJ}\}$ \\
\hline
$C_{\rm A45}$ & $C_{\rm A43}\max\{1, C_{\rm A44}\}$ & $C_{\rm sa1}$ & $C_{\rm A41}\max\{1,C_{\rm A45}\alpha^{-1}\}$ \\
\hline
$C_{\rm sa}$ & $4C_{\rm sa1}$ & $C_{\rm rl}$ & $C_{\rm vrel}$ \\
\hline
$\Lambda_5$ & $C_{\rm rl}C_{\rm sa}$ & $\Delta_1$ & $2^5C_{\rm sa}^2\delta^{2s}C_{\rm rl}^2$ \\
\hline
$\Delta_2$ & $2^5C_{\rm sa}^2\Lambda_{\rm mon}^2\delta^{2s}$ & $\Delta_3$ & $3\times 2^6\Lambda_5^2\Lambda_0^2\delta^2$ \\
\hline
$\Delta_4$ & $2^{13/3}\Lambda_5^{4/3}\Lambda_0^{2/3}\delta^{2/3}(1+\Lambda_{\rm mon}^2)^{1/3}$ & & \\
\hline
\end{tabular}
\caption{Definitions of various constants used in proof of $(\textbf{A4}_{\varepsilon})$.}
\label{tab:constants-horizontal}
\end{table}

\begin{theorem} [${\rm\bf A4}_{\varepsilon}$]
	\label{A4ep}
	Let $(\bar{\bf y}_l, \bar{\bf p}_l, \bar{\U}_l)$ denote the solution of the optimality system \eqref{dstate}-\eqref{dadjoint} associated with the $l$-th refinement of $\T_0$, and define ${\bf d}_{l,l+1} := {\bf d}(\T_l,\T_{l+1})$. Then, for any $\varepsilon > 0$, there exist $\delta > 0$ and $0 < \Lambda_4(\varepsilon) < \infty$ such that, for all $l,m \in \mathbb{N} \cup \{0\}$,
	\begin{equation}
		\tag{$\textbf{A4}_{\varepsilon}$} 
		\sum_{k=l}^{l+m}{\bf d}_{k,k+1}^2\leq \Lambda_{4(\varepsilon)}\eta_{l}^2+\varepsilon\sum_{k=l}^{l+m}\eta_{k}^2.
	\end{equation}
\end{theorem}

\noindent The proof is divided into 4 steps.  
\begin{proof}
{\bf Step 1 : (Key estimate)}
Given any $l\in \mathbb{N}\cup\{0\}$ define $e_l^2:=\enorm{\bar{\y}-\bar{\y}_l}_{\rm{pw}}^2+\enorm{\bar{\pee}-\bar{\pee}_l}_{\rm{pw}}^2.$ For any $k\in\mathbb{N}_0$, definition of ${\bf d}_{k,k+1}$ and $e_k$ plus a re-grouping of the terms shows that  
\begin{equation}
	\label{5.13}
	\begin{aligned}
		{\bf d}_{k,k+1}^2+e_{k+1}^2-e_{k}^2&=(\enorm{\bar{\y}_{k+1}-\bar{\y}_k}_{\rm{pw}}^2+\enorm{\bar{\y}-\bar{\y}_{k+1}}_{\rm{pw}}^2-\enorm{\bar{\y}-\bar{\y}_{k}}_{\rm{pw}}^2)\\
		&+(\enorm{\bar{\pee}_{k+1}-\bar{\pee}_k}_{\rm{pw}}^2+\enorm{\bar{\pee}-\bar{\pee}_{k+1}}_{\rm{pw}}^2-\enorm{\bar{\pee}-\bar{\pee}_{k}}_{\rm{pw}}^2)\\
		&:= E_1+E_2.
	\end{aligned}
\end{equation}
Observe that term $\enorm{\bar{\y}-\bar{\y}_k}_{\rm{pw}}^2$ can be written as
\begin{equation*}
	\enorm{\bar{\y}-\bar{\y}_{k}}_{\rm{pw}}^2=a_h(\bar{\y}-\bar{\y}_{k},\bar{\y}-\bar{\y}_{k})=a_h(\bar{\y}-\bar{\y}_{k+1},\bar{\y}-\bar{\y}_{k})+a_h(\bar{\y}_{k+1}-\bar{\y}_{k},\bar{\y}-\bar{\y}_{k}).
\end{equation*}
This and an elementary algebra in equation \eqref{5.13} give
\begin{equation*}
	E_1=2a_h(\bar{\y}_{k+1}-\bar{\y}_k,\bar{\y}_{k+1}-\bar{\y})\mbox{ and }E_2=2a_h(\bar{\pee}_{k+1}-\bar{\pee}_k,\bar{\pee}_{k+1}-\bar{\pee}).
\end{equation*}
Substituting above estimate in \eqref{5.13} results in
\begin{equation}
	\label{5.14}
	\frac{1}{2}({\bf d}_{k,k+1}^2+e_{k+1}^2-e_{k}^2)=a_h(\bar{\y}_{k+1}-\bar{\y}_k,\bar{\y}_{k+1}-\bar{\y})+a_h(\bar{\pee}_{k+1}-\bar{\pee}_k,\bar{\pee}_{k+1}-\bar{\pee}).
\end{equation}
{\bf Step 2 : (Estimates for the right hand side of \eqref{5.14})}
Using the orthogonality condition from Lemma \ref{interpolation}{\it (iii)} as $a_h(\bar{\y}_{k+1},\bar{\y}_{k+1}-\bar{\y})=a_h(\bar{\y}_{k+1},I_{k+1}(\bar{\y}_{k+1}-\bar{\y}))$ and $a_h(\bar{\y}_k,\bar{\y}_{k+1}-\bar{\y})=a_h(\bar{\y}_k,I_k(\bar{\y}_{k+1}-\bar{\y}))$ leads to
\begin{equation}
\label{ik}
	a_h(\bar{\y}_{k+1}-\bar{\y}_k,\bar{\y}_{k+1}-\bar{\y})=a_h(\bar{\y}_{k+1},I_{k+1}(\bar{\y}_{k+1}-\bar{\y}))-a_h(\bar{\y}_k,I_k(\bar{\y}_{k+1}-\bar{\y})).
\end{equation}
Use of \eqref{dstate} at $\T_{k+1}$ and $\T_{k}$ level for the choice of test functions as ${\bf v}_{k+1}=I_{k+1}(\bar{\bf y}_{k+1}-\bar{\bf y})$ and ${\bf v}_k=I_{k}(\bar{\bf y}_{k+1}-\bar{\bf y})$, respectively, to obtain
\begin{equation}
\label{aa}
\begin{aligned}
    a_h(\bar{\y}_{k+1},I_{k+1}(\bar{\y}_{k+1}-\bar{\y}))&=({\bf f}+\bar{\bf u}_{k+1}, I_{k+1}(\bar{\y}_{k+1}-\bar{\y}))+b_{h}(I_{k+1}(\bar{\y}_{k+1}-\bar{\y}),\bar{r}_{k+1})\mbox{ and }\\
    a_h(\bar{\y}_{k},I_{k}(\bar{\y}_{k+1}-\bar{\y}))&=({\bf f}+\bar{\bf u}_{k}, I_{k}(\bar{\y}_{k+1}-\bar{\y}))+b_{h}(I_{k}(\bar{\y}_{k+1}-\bar{\y}),\bar{r}_{k}).
\end{aligned}
\end{equation}
Using definition of $b_{h}(\cdot,\cdot)$ and Lemma \ref{interpolation}{\it (i)}, we obtain
\begin{equation*}
    b_{h}(I_{k+1}(\bar{\y}_{k+1}-\bar{\y}),\bar{r}_{k+1})=\sum_{K\in\T_{k+1}}\int_{K}{\rm div}(I_{k+1}(\bar{\y}_{k+1}-\bar{\y}))\bar{r}_{k+1}\,dx=\sum_{K\in\T_{k+1}}\int_{K}{\Pi}_{k+1}{\rm div}(\bar{\y}_{k+1}-\bar{\y})\bar{r}_{k+1}\,dx.
\end{equation*}
Here, $\Pi_{k+1}$ denotes the piecewise constant $L^2$-projection operator over $\mathcal{P}_0(\T_{k+1})$. As $\bar{\bf y}_{k+1}\in {\bf V}_h$ is the solution of \eqref{dstate} at $\T_{k+1}$, we have ${\rm div}\bar{\bf y}_{k+1}|_{K}\in \mathcal{P}_0(\T_{k+1})$ for all $K\in\T_{k+1}$. Therefore, we have $\Pi_{k+1}{\rm div}\bar{\bf y}_{k+1}|_K={\rm div}\bar{\bf y}_{k+1}|_K$ for all $K\in\T_{k+1}$. Using this, choosing test function ${q}_{k+1}=\bar{r}_{k+1}$ in \eqref{dstate} at $\T_{k+1}$, and by introducing ${\rm div}\bar{\bf y}$, we obtain
\begin{equation*}
    b_{h}(I_{k+1}(\bar{\y}_{k+1}-\bar{\y}),\bar{r}_{k+1})=-\sum_{K\in\T_{k+1}}\int_{K}\Pi_{k+1}{\rm div}\bar{\bf y}\bar{r}_{k+1}\,dx=-\sum_{K\in\T_{k+1}}\int_{K}{\rm div}\bar{\bf y}\bar{r}_{k+1}\,dx=0.
\end{equation*}
Here the last step utilizes the choice of test function $q=\bar{r}_{k+1}$ in \eqref{cstate}. Analogous calculation leads to $b_{h}(I_{k}(\bar{\y}_{k+1}-\bar{\y}),\bar{r}_{k})=0$. Substitution of this in \eqref{aa} and \eqref{ik} and introducing $\bar{\bf u}_k$ yield
\begin{equation}
\label{5115}
a_h(\bar{\y}_{k+1}-\bar{\y}_k,\bar{\y}_{k+1}-\bar{\y})=({\bf f}+\bar{\U}_{k},(I_{k+1}-I_k)(\bar{\y}_{k+1}-\bar{\y}))+{(\bar{\U}_{k+1}-\bar{\U}_{k},I_{k+1}(\bar{\y}_{k+1}-\bar{\y})).}
\end{equation}

\noindent Lemma \ref{interpolation}{\it (iv)} and Lemma \ref{interpolation}{\it (ii)} (used twice for $(I_{k+1}-1)(\bar{\y}_{k+1}-\bar{\y}))$ and $(1-I_k)(\bar{\y}_{k+1}-\bar{\y}))$, use of \eqref{defuh} at $\T_k$ and $\T_{k+1}$ levels, Lipschitz continuity of projection operator, and stability of operator $I_{k+1}$ show
\begin{equation}
	\label{5.15}
	a_h(\bar{\y}_{k+1}-\bar{\y}_k,\bar{\y}_{k+1}-\bar{\y})\leq C_{\rm A41}\big(\sum_{K\in \T_k\backslash \T_{k+1}} h_{K} \|{\bf f}+\bar{\U}_k\| \enorm{\bar{\y}_{k+1}-\bar{\y}}_{\rm{pw}}+\alpha^{-1} {\|\bar{\pee}_{k+1}-\bar{\pee}_k\| \|\bar{\y}_{k+1}-\bar{\y}\|}\big).
\end{equation}
Applying the H\"{o}lder's inequality in the first term of \eqref{5.15} results in
\begin{equation}
	\label{5.16}
	\sum_{K\in \T_k\backslash \T_{k+1}}h_{K}\|{\bf f}+\bar{\U}_k\| \enorm{\bar{\y}_{k+1}-\bar{\y}}_{\rm{pw}}\leq \mu_{S,k}(\T_{k}\backslash\T_{k+1})\enorm{\bar{\y}_{k+1}-\bar{\y}}_{\pw}\leq\mu_{S,k}(\T_{k}\backslash\T_{k+1})e_{k+1}.
\end{equation}
Here, $\mu_{S,k}(\T_{k}\backslash\T_{k+1})$ denotes the volume estimator corresponding to the state equation at the level $k$ in the region $\T_{k}\backslash\T_{k+1}$.  For estimating the second term from \eqref{5.15}, observe that \eqref{3.3} and \eqref{3.8} give $\|\bar{\y}_{k+1}-\bar{\y}\|\leq C_{\rm A42}(\|\widetilde{\bf y}-\bar{\bf y}_{k+1}\|+\|\widetilde{\bf p}-\bar{\bf p}_{k+1}\|)$. Now use $L^2$- error estimates from Lemma \ref{lemma 4.1} to get
\begin{equation*}
    \|\widetilde{\bf y}-\bar{\bf y}_{k+1}\|+\|\widetilde{\bf p}-\bar{\bf p}_{k+1}\|\leq C_{\rm A43}\delta^s(\enorm{\widetilde{\bf y}-\bar{\bf y}_h}_{\rm pw}+\enorm{\widetilde{\bf p}-\bar{\bf p}_h}_{\rm pw}+\eta_{\rm st}(\T_{k+1})+\eta_{\rm adj}(\T_{k+1})).
\end{equation*}
Here the last estimate utilizes the Lemma \ref{lemma 4.2}. Use of above estimate yields
\begin{equation}
\label{hg}
\|\bar{\y}_{k+1}-\bar{\y}\|\leq C_{\rm A43}\delta^{s}(\enorm{\widetilde{\bf y}-\bar{\bf y}_{k+1}}_{\rm pw}+\enorm{\widetilde{\bf p}-\bar{\bf p}_{k+1}}_{\rm pw}+\eta_{\rm st}(\T_{k+1})+\eta_{\rm adj}(\T_{k+1})).
\end{equation}
 Use of triangle inequality, \eqref{cst}, \eqref{cad}, Lipschitz continuity of $\Pi_{[{\bf u}_a,{\bf u}_b]}$, and Lemma \ref{discrete sobolev}{\it (i)} results in
\begin{equation*}
    \enorm{\widetilde{\bf y}-\bar{\bf y}_{k+1}}_{\rm pw}+\enorm{\widetilde{\bf p}-\bar{\bf p}_{k+1}}_{\rm pw}\leq C_{\rm A44}(\enorm{\bar{\bf y}-\bar{\bf y}_{k+1}}_{\rm pw}+\enorm{\bar{\bf p}-\bar{\bf p}_{k+1}}_{\rm pw})=C_{\rm A44}e_{k+1}. 
\end{equation*}
Substitution of above estimate in \eqref{hg} yields
\begin{equation}
\label{518}
    \|\bar{\y}_{k+1}-\bar{\y}\|\leq C_{\rm A45}\delta^{s}(e_{k+1}+\eta_{k+1}). 
\end{equation}
Use of \eqref{518} along with Lemma \ref{discrete sobolev}{\it (ii)} gives
\begin{equation}
\label{519}
    \alpha^{-1}\|\bar{\bf p}_{k+1}-\bar{\bf p}_k\|\|\bar{\y}_{k+1}-\bar{\y}\|\leq C_{\rm A45}\alpha^{-1}\delta^{s}{\bf d}_{k,k+1}(e_{k+1}+\eta_{k+1}).
\end{equation}
Substitute \eqref{5.16} and \eqref{519} in \eqref{5.15} to obtain
\begin{equation}
\label{ay}
    a_h(\bar{\y}_{k+1}-\bar{\y}_k,\bar{\y}_{k+1}-\bar{\y})\leq C_{\rm sa1}(\mu_{S,k}(\T_k\setminus\T_{k+1})+\delta^{s}{\bf d}_{k,k+1})e_{k+1}+C_{\rm sa1}\delta^{s}{\bf d}_{k,k+1}\eta_{k+1}. 
\end{equation}
Repeating the analogous calculation for term $a_h(\bar{\pee}_{k+1}-\bar{\pee}_k,\bar{\pee}_{k+1}-\bar{\pee})$ gives
\begin{equation}
\label{ap}
    a_h(\bar{\pee}_{k+1}-\bar{\pee}_k,\bar{\pee}_{k+1}-\bar{\pee})\leq C_{\rm sa1}(\mu_{A,k}(\T_k\setminus\T_{k+1})+\delta^{s}{\bf d}_{k,k+1})e_{k+1}+C_{\rm sa1}\delta^{s}{\bf d}_{k,k+1}\eta_{k+1}. 
\end{equation}
Here $\mu_{A,k}(\T_{k}\backslash\T_{k+1})$ denotes the volume estimator corresponding to the adjoint equation at the level $k$ in the region $\T_{k}\backslash\T_{k+1}$. Substitution of \eqref{ay} and \eqref{ap} in  \eqref{5.14} gives
\begin{equation}
	\label{5.17}
	({\bf d}_{k,k+1}^2+e_{k+1}^2-e_{k}^2)\leq C_{\rm{sa}} (\mu_k(\T_k\backslash\T_{k+1})+\delta^{s} {\bf d}_{k,k+1}){e}_{k+1}+C_{\rm sa}\delta^{s}{\bf d}_{k,k+1}\eta_{k+1}
\end{equation}
with $\mu_{k}^2=\mu_{S,k}^2+\mu_{A,k}^2$.\vspace{0.1cm}\\
{\bf Step 3 : (Estimate for the term $\mu_{k}(\T_k\backslash \T_{k+1})$)} Using the Youngs inequality for $\epsilon=2$ in Corollary \ref{corollary 5.1} and adding the term $\mu_{k+1}^2(\T_k\cap\T_{k+1})$ both the sides leads to
\begin{equation*}
	\mu_{k+1}^2\leq \mu_{k+1}^2(\T_{k}\cap\T_{k+1})+\frac{3}{4}\mu_{k}^2(\T_{k}\backslash \T_{k+1})+3\delta^{2}\Lambda_0^2{\bf d}_{k,k+1}^2.
\end{equation*}

\noindent Use the identity $\mu_{k}^2(\T_{k}\backslash\T_{k+1})=\mu_{k}^2-\mu_{k}^2(\T_{k}\cap\T_{k+1})$ and some elementary algebra to conclude,
\begin{equation}
	\label{5.18}
	\frac{1}{4}\mu_{k}^2(\T_{k}\backslash \T_{k+1})\leq\mu_{k}^2-\mu_{k+1}^2+\mu_{k+1}^2(\T_{k}\cap\T_{k+1})-\mu_{k}^2(\T_{k}\cap\T_{k+1})+3\Lambda_0^2\delta^{2}{\bf d}_{k,k+1}^2.
\end{equation} Multiply both sides by $\mu_{k+1}(\T_{k}\cap\T_{k+1})+\mu_{k}(\T_{k}\cap\T_{k+1})$, in Corollary \ref{corollary 5.1}, to get
\begin{equation*}
	\mu_{k+1}^2(\T_{k}\cap\T_{k+1})-\mu_{k}^2(\T_{k}\cap\T_{k+1})\leq\Lambda_0\delta{\bf d}_{k,k+1}(\mu_{k+1}(\T_{k}\cap\T_{k+1})+\mu_{k}(\T_{k}\cap\T_{k+1})).
\end{equation*}
Substitute above in \eqref{5.18}, to get,
\begin{equation*}
	\frac{1}{4}\mu_{k}^2(\T_{k}\backslash \T_{k+1})\leq\mu_{k}^2-\mu_{k+1}^2+\Lambda_0\delta{\bf d}_{k,k+1}(\mu_{k+1}(\T_{k}\cap\T_{k+1})+\mu_{k}(\T_{k}\cap\T_{k+1}))+3\Lambda_0^2\delta^{2}{\bf d}_{k,k+1}^2.
\end{equation*}
Using the fact that $\mu_{k+1}^2(\T_{k}\cap\T_{k+1})\leq \mu_{k+1}^2$ and $\mu_{k}^2(\T_{k}\cap\T_{k+1})\leq \mu_{k}$ one can finally conclude,
\begin{equation}
	\label{5.19}
	\frac{1}{4}\mu_{k}^2(\T_{k}\backslash \T_{k+1})\leq\mu_{k}^2-\mu_{k+1}^2+\Lambda_0\delta{\bf d}_{k,k+1}(\mu_{k+1}+\mu_{k})+3\Lambda_0^2\delta^{2}{\bf d}_{k,k+1}^2.
\end{equation}
{\bf Step 4 : (Crucial estimates)}
The Young's inequality used thrice with $\epsilon=2^{-3}\varepsilon C_{\rm{vrel}}^{-2}, a=C_{\rm{sa}} \mu_{k}(\T_{k}\backslash \T_{k+1})$, $b=e_{k+1}$, $\epsilon=2^2C_{\rm sa}^2\delta^{2\gamma}$, $a=C_{\rm sa}\delta^{s}{\bf d}_{k,k+1}$, $b=e_{k+1}$, and $\epsilon=C_{\rm sa}^22^2\delta^{2\gamma}$, $a=C_{\rm sa}\delta^{s}{\bf d}_{k,k+1}$, $b=\eta_{k+1}$ in \eqref{5.17} gives 
\begin{align*}
	({\bf d}_{k,k+1}^2+e_{k+1}^2-e_{k}^2)\leq& 2^2 \varepsilon^{-1}C_{\rm rl}^2 C_{\rm sa}^2\mu_{k}^2(\T_{k}\backslash\T_{k+1})+2^{-4}\varepsilon C_{\rm{rl}}^{-2} e_{k+1}^2+2C_{\rm sa}^2\delta^{2\gamma}e_{k+1}^2+2C_{\rm sa}^2\delta^{2\gamma}\eta_{k+1}^2+2^{-2}{\bf d}_{k,k+1}^2\\
	=&2^2 \varepsilon^{-1}C_{\rm rl}^2 C_{\rm sa}^2\mu_{k}^2(\T_{k}\backslash\T_{k+1})+(2^{-4}\varepsilon C_{\rm{rl}}^{-2}+2C_{\rm sa}^2\delta^{2\gamma}) e_{k+1}^2+2C_{\rm sa}^2\delta^{2\gamma}\eta_{k+1}^2+2^{-2}{\bf d}_{k,k+1}^2.
\end{align*}
The estimate from \eqref{constants} shows {{$2C_{\rm{sa}}^2\delta^{2\gamma}\leq 2^{-4} \varepsilon C_{\rm{rl}}^{-2}$}} and {{$2C_{\rm sa}^2\delta^{2\gamma}\leq 2^{-4}\Lambda_{\rm mon}^{-2}\varepsilon$}}
\begin{equation*}
	\frac{3}{4}{\bf d}_{k,k+1}^2+e_{k+1}^2-e_k^2\leq 2^2 \varepsilon^{-1}\Lambda_5^2\mu_{k}^2(\T_{k}\backslash\T_{k+1})+2^{-3}\varepsilon C_{\rm rl}^{-2} e_{k+1}^2+2^{-4}\Lambda_{\rm mon}^{-2}\varepsilon\eta_{k+1}^2.
\end{equation*}
Using quasi-monotonicity property as $\eta_{k+1}^2\leq \Lambda_{\rm mon}^2\eta_{k}^2$ \cite[Theorem 4.1]{CCRH17} we get
\begin{equation}
    \label{k+1}
    \frac{3}{4}{\bf d}_{k,k+1}^2+e_{k+1}^2-e_k^2\leq 2^2 \varepsilon^{-1}\Lambda_5^2\mu_{k}^2(\T_{k}\backslash\T_{k+1})+2^{-3}\varepsilon C_{\rm rl}^{-2} e_{k+1}^2+2^{-4}\varepsilon\eta_{k}^2.
\end{equation}
\noindent Substitute \eqref{5.19} in the first term of above displayed inequality to get
\begin{align*}
	\frac{3}{4}{\bf d}_{k,k+1}^2+e_{k+1}^2-e_k^2&\leq 2^4 \varepsilon^{-1}\Lambda_5^2(\mu_{k}^2-\mu_{k+1}^2)+2^4 \varepsilon^{-1}\Lambda_5^2\Lambda_0\delta{\bf d}_{k,k+1}(\mu_{k}+\mu_{k+1})\\
	&+3\times2^4 \varepsilon^{-1}\Lambda_5^2\Lambda_0^2\delta^2{\bf d}_{k,k+1}^2+2^{-3}\varepsilon C_{\rm rl}^{-2} e_{k+1}^2+2^{-4}\varepsilon\eta_{k}^2.
\end{align*}

\noindent Estimate {{$3\times2^4 \varepsilon^{-1}\Lambda_5^2\Lambda_0^2\delta^2\leq 2^{-2}$}} from \eqref{constants} leads to
\begin{equation}
	\label{5.20}
	2^{-1}{\bf d}_{k,k+1}^2+e_{k+1}^2-e_k^2\leq 2^{4}\varepsilon^{-1}\Lambda_5^2 (\mu_{k}^2-\mu_{k+1}^2)+2^{4}\varepsilon^{-1}\Lambda_5^2\Lambda_0 \delta{\bf d}_{k,k+1}(\mu_{k}+\mu_{k+1})+2^{-3}\varepsilon C_{\rm rl}^{-2} e_{k+1}^2+2^{-4}\varepsilon\eta_{k}^2.
\end{equation}
Now a use of Young's inequality with $\epsilon=2^{-4}\varepsilon/(1+\Lambda_{\rm mon}^2), a=2^{4}\varepsilon^{-1}\Lambda_5^2\Lambda_0\delta{\bf d}_{k,k+1}$, and $b=(\mu_{k}+\mu_{k+1})$ in the second term of the above displayed inequality yields
\begin{equation*}	2^4\varepsilon^{-1}\Lambda_5^2\Lambda_0\delta{\bf d}_{k,k+1}(\mu_{k}+\mu_{k+1})\leq 2^{11}\varepsilon^{-3}\Lambda_5^4\Lambda_0^2\delta^2(1+\Lambda_{\rm mon}^2){\bf d}_{k,k+1}^2 +2^{-4}\varepsilon(1+\Lambda_{\rm mon}^2)^{-1}(\mu_{k+1}^2+\mu_{k}^2).
\end{equation*}
The bound \eqref{constants} implies { $2^{11}\varepsilon^{-3}\Lambda_5^4\Lambda_0^2\delta^2(1+\Lambda_{\rm mon}^2)\leq 2^{-2}$}, and hence
\begin{equation}
	\label{5.21}
    2^4\varepsilon^{-1}\Lambda_5^2\Lambda_0\delta{\bf d}_{k,k+1}(\mu_{k}+\mu_{k+1})\leq 2^{-2}{\bf d}_{k,k+1}^2 +2^{-4}\varepsilon(1+\Lambda_{\rm mon}^2)^{-1}(\mu_{k+1}^2+\mu_{k}^2).
\end{equation}
Substitute \eqref{5.21} in \eqref{5.20} to get
\begin{equation*}
	{\bf d}_{k,k+1}^2+4(e_{k+1}^2-e_k^2)\leq 2^{6}\varepsilon^{-1}\Lambda_5^{2}(\mu_{k}^2-\mu_{k+1}^2)+2^{-2}\varepsilon (1+\Lambda_{\rm mon}^2)^{-1}(\mu_{k+1}^2+\mu_{k}^2)+2^{-1}\varepsilon C_{\rm rl}^{-2} e_{k+1}^2+2^{-2}\varepsilon\eta_{k}^2.
\end{equation*}
\noindent Taking summation from $l$ to $l+m$ both sides leads to $e_{l+m+1}^2-e_l^2$ on the left hand side and $\mu_{l}^2-\mu_{l+m+1}^2$ on the right hand side, resulting in
\begin{align}
	\sum_{k=l}^{l+m}{\bf d}_{k,k+1}^2+4(e_{l+m+1}^2-e_l^2)\leq & 2^{6}\varepsilon^{-1}\Lambda_5^{2}(\mu_{l}^2-\mu_{l+m+1}^2)+2^{-2}\varepsilon(1+\Lambda_{\rm mon}^2)^{-1} \sum_{k=l}^{l+m}(\mu_{k+1}^2+\mu_{k}^2)\nonumber\\
    &+2^{-1}\varepsilon C_{\rm rl}^{-2}\sum_{k=l}^{l+m} e_{k+1}^2+2^{-2}\varepsilon\sum_{k=l}^{l+m}\eta_{k}^2.\label{5.23}
\end{align}
\noindent Rearrange the terms from equation \eqref{5.23} to obtain
\begin{align*}
	\sum_{k=l}^{l+m}{\bf d}_{k,k+1}^2\leq 
    &4e_l^2+2^{6}\varepsilon^{-1}\Lambda_5^{2}\mu_{l}^2+2^{-2}\varepsilon(1+\Lambda_{\rm mon}^2)^{-1}\sum_{k=l}^{l+m}(\mu_{k+1}^2+\mu_{k}^2)+2^{-2}\varepsilon\sum_{k=l}^{l+m}\eta_k^2\\
    &+2\varepsilon^{-1}C_{\rm rl}^{-2}\sum_{k=l}^{l+m-1}e_{k+1}^2+(2^{-1}\varepsilon C_{\rm rl}^{-2}-4)e_{l+m+1}^{2}.
\end{align*}
The bound on $\varepsilon$ makes $(2^{-1}\varepsilon C_{\rm rl}^{-2}-4)<0$. Theorem \ref{theorem 4.3} shows $e_l^2\leq C_{\rm rl}^2\eta_{l}^2$ ($C_{\rm rl}=C_{\rm vrel})$. From the definition of estimator from Section \ref{adaptive convergence} we observe $\mu_l\leq \eta_{l}$. Moreover, the quasi monotonicity property of estimator shows $\mu_{k+1}^2\leq \eta_{k+1}^2\leq \Lambda_{\rm mon}^2\eta_k^2$ for all $k\in\mathbb{N}$. Using these properties in above mentioned estimates result in
\begin{align*}
    \sum_{k=l}^{l+m}{\bf d}_{k,k+1}^2&\leq (4+2^{6}\varepsilon^{-1}\Lambda_{5}^2)\eta_{l}^2+2^{-2}\varepsilon\sum_{k=l}^{l+m}\eta_k^2+2^{-2}\varepsilon\sum_{k=l}^{l+m}\eta_k^2+2^{-1}\varepsilon\sum_{k=l}^{l+m-1}\eta_{k+1}^2
    \\
    &\leq(4+2^{6}\varepsilon^{-1}\Lambda_{5}^2)\eta_{l}^2+2^{-1}\varepsilon\sum_{k=l}^{l+m}\eta_k^2+2^{-1}\varepsilon\sum_{k=l}^{l+m}\eta_{k}^2.
\end{align*}
This completes the proof for ${\bf (A4_{\varepsilon}})$ for $\Lambda_{4(\varepsilon)}:=4+2^{6}\varepsilon^{-1}\Lambda_{5}^2.$
\end{proof}
\section{Verification of axiom of adaptivity (discretised approach)}
\label{verification of discretised axioms}
\noindent Recall from \eqref{estidef} from Section \ref{adaptive convergence} for discretised approach, the estimator is defined as $\eta^2(\T):=\eta_{\rm st}^2(\T)+\eta_{\rm adj}^2(\T)+\eta_{\rm C}^2(\T)$. Furthermore, ${\bf d}^2(\T,\widehat{\T})$ is defined as  ${\bf d}^2(\T,\widehat{\T}):=\|\widehat{\bar{\U}}_h-\bar{\U}_h\|^2+\enorm{\widehat{\bar{\y}}_h-\bar{\y}_h}_{\rm{pw}}^2+\3^2$.

\subsection{Proof of (A1)}
\label{disca1}
\begin{proof}
Reverse triangle inequality shows
\begin{equation}
	\begin{aligned}
		|\widehat{\eta}(\T\cap\widehat{\T})-\eta(\widehat{\T}\cap\T)|^2\leq & \sum_{K\in\T\cap\widehat{\T}}\big((\widehat{\mu}_{S,K}-\mu_{S,K})^2+(\widehat{\mu}_{A,K}-\mu_{A,K})^2+(\widehat{\eta}_{C,K}-\eta_{C,K})^2\\
        &+(\widehat{\rho}_{S,\E(K)}-\rho_{S,\E(K)})^2
		+(\widehat{\rho}_{A,\E(K)}-\rho_{A,\E(K)})^2\big). 
	\end{aligned}
\end{equation}
Again use of reverse triangle inequality and Lemma \ref{discrete sobolev}{\it (ii)} yield
\begin{align}
	\label{6.2}
	(\widehat{\mu}_{S,K}-\mu_{S,K})^2&\leq h_{K}^2 \|\widehat{\bar{\U}}_h-\bar{\U}_h\|_{K}^2\mbox{ and }(\widehat{\mu}_{A,K}-\mu_{A,K})^2\leq h_{K}^2 C_{\rm{PI}}^2 \enorm{\widehat{\bar{\y}}_h-\bar{\y}_h}_{\rm{pw}}^2.
\end{align}
Recall $\widehat{\eta}_{C,K}=h_K\|\n_h(\widehat{\bar{\U}}_h-\alpha^{-1}\widehat{\bar{\pee}}_h)\|_{K}$ and $\eta_{C,K}=h_K\|\n_h(\bar{\U}_h-\alpha^{-1}\bar{\pee}_h)\|_{K}$. Using these definitions and applying reverse triangle inequality we observe, $(\widehat{\eta}_{C,K}-\eta_{C,K})^2\leq h_K^2\alpha^{-2}\3^2$. Therefore 
\begin{equation}
	\label{6.3}
	\sum_{K\in\T\cap\widehat{\T}}\big((\widehat{\mu}_{S,K}-\mu_{S,K})^2+(\widehat{\mu}_{A,K}-\mu_{A,K})^2+{(\widehat{\eta}_{C,K}-\eta_{C,K})^2}\big)\leq h^2C_{\rm M\Pi}^2{\bf d}^2(\T,\widehat{\T}) 
\end{equation}
with $C_{\rm M\Pi}^2:=\max\{1,C_{\rm PI}^2,\alpha^{-2}\}$. Using \eqref{6.3} for estimating the volume estimators and using \eqref{edge1}-\eqref{edge2} for getting estimate on edge estimators, we get the desired estimate as $|\widehat{\eta}(\T\cap\widehat{\T})-\eta(\widehat{\T}\cap\T)|\leq \Lambda_{1} {\bf d}(\T,\widehat{\T})$ with $\Lambda_1^2:=|\Omega|C_{\rm M\Pi}^2+C_{\rm djc}^2$. 
\end{proof}

\subsection{Proof of (A2)}
\begin{proof}
	Reverse triangle inequality and some elementary calculation shows
	\begin{equation*}
		\begin{aligned}
			\widehat{\eta}(\widehat{\T} \setminus \T)\leq & \biggl(\sum_{K\in\T\setminus\widehat{\T}}\sum_{T\in\widehat{T}(K)}(\widehat{\mu}_{S,T}-\mu_{S,T})^2+(\widehat{\mu}_{A,T}-\mu_{A,T})^2+(\widehat{\eta}_{C,K}-\eta_{C,K})^2+\\
			&+(\widehat{\rho}_{S,\E(T)}-\rho_{S,\E(T)})^2+(\widehat{\rho}_{A,\E(T)}-\rho_{A,\E(T)})^2\biggr)^{1/2}+\biggl(\sum_{K\in\T\setminus\widehat{\T}}\sum_{T\in\widehat{T}(K)}{\mu}_{S,T}^2+{\rho}_{S,\E(T)}^2\\
			&+{\mu}_{A,T}^2+{\rho}_{A,\E(T)}^2+\eta_{C,K}^2\biggr)^{1/2}=: S_1+S_2.
		\end{aligned}
	\end{equation*}
\noindent For estimating $S_1$, we follow analogous steps as mentioned in Section \ref{disca1} over $\widehat{\T}\setminus\T$ to get 
\begin{equation}
\label{6.4}
	S_1^2\leq (|\Om| C_{\rm M\Pi}^2+C_{\rm djc}^2){\bf d}^2(\T,\widehat{\T}).
\end{equation}
 As $T\in\widehat{T}(K)$ we can use the relation $h_T= |T|^{1/2}\leq 2^{-1/2} |K|^{1/2}= 2^{-1/2} h_{K}$ to obtain
\begin{equation}
\label{6.5}
	\begin{aligned}
	S_2^2&\leq \sum_{K\in\T\setminus\widehat{\T}} 2^{-1}({\mu}_{S,K}^2+{\mu}_{A,K}^2+{\eta}_{C,K}^2)+2^{-1/2}({\rho}_{S,\E(K)}^2+{\rho}_{A,\E(K)}^2)\\
	&\leq2^{-1/2}\sum_{K\in\T\setminus\widehat{\T}}({\mu}_{S,K}^2+{\rho}_{S,\E(K)}^2+{\mu}_{A,K}^2+{\rho}_{A,\E(K)}^2+{\eta}_{C,K}^2).
	\end{aligned}
\end{equation}
\noindent Bounds on $S_1$ and $S_2$ leads to desired estimate as $\widehat{\eta}(\widehat{\T} \setminus \T)\leq \Lambda_2 {\bf d}(\T,\widehat{\T})
	+\rho
	_2\eta(\T\setminus\widehat{\T})\mbox{ with }\Lambda_2^2:=|\Omega|C_{\rm M\Pi}^2+C_{\rm djc}^2\mbox{ and }\rho=2^{-1/4}$.
\end{proof} 

\begin{corollary}
	\label{corollary 6.1}
	Let the volume estimator with respect to $\widehat{\T}$ be $\widehat{\mu}^2 := \sum_{K\in\widehat{\T}}(\widehat{\mu}_{S,K}^2+\widehat{\mu}_{A,K}^2+\widehat{\eta}_{C,K}^2).$ Then for $\Lambda_0:= C_{\rm M\Pi}$ following holds.
	\begin{align}
		|\widehat{\mu}(\T\cap\widehat{\T})-\mu(\T\cap\widehat{\T})|&\leq h\Lambda_0 {\bf d}(\T,\widehat{\T})\hspace{0.2cm}\mbox{and}\hspace{0.2cm}\widehat{\mu}(\widehat{\T}\backslash \T)\leq  2^{-1/2}\mu(\T\backslash\widehat{\T})+h\Lambda_0{\bf d}(\T,\widehat{\T}).
	\end{align}
\end{corollary}
\begin{proof}
    The proof follows from estimate mentioned in \eqref{6.3} and \eqref{6.5}.
\end{proof}
\subsection{Proof of (A3)} 
The proof of ${\bf (A3)}$ follows from Section \ref{a3 proof v} for the error terms $\6$ and $\7$. The bound for $\normtwo$ is estimated below.
\begin{proof}
Theorem \ref{d error equivalence}{\it (iii)} gives the following estimate on ${\bf d}^2(\T,\widehat{\T})$ as
\begin{equation}
\label{cr}
	\|\widehat{\bar{\U}}_h-\bar{\U}_h\|+\2+\3\leq C_{\rm R}(\normtwo+\6+\7)
\end{equation}
\noindent with $C_{\rm R}:=\max\{C_{27}(C_{22}+C_{24})+C_{24}, (1+C_{27})(1+C_{22})\}$. From Section \ref{a3 proof v} we conclude
\begin{equation}
	\label{6.8}
	\6^2\leq C_{\rm{S}}\eta_{\rm st}^2(\mathcal{R}(\T,\widehat{\T}))\mbox{ and }\7^2\leq C_{\rm{A}}\eta_{\rm adj}^2(\mathcal{R}(\T,\widehat{\T})). 
\end{equation}
In order to find estimate on the term $\|\widetilde{\bf u}_h-\bar{\bf u}_h\|$, use definitions of $\bar{\bf u}_h$ and $\widetilde{\bf u}_h$ from \eqref{defuh} and \eqref{2.21}, respectively to observe $\|\widetilde{\bf u}_h-\bar{\bf u}_h\|=\|\widetilde{\bf u}_h-\bar{\bf u}_h\|_{\T\setminus\widehat{\T}}$. Therefore
\begin{equation}
	\label{6.7}
	\begin{aligned}
\normtwo^2&=\normtwo_{\T\setminus\widehat{\T}}^2=\Bigl\lVert\Pi_{[\U_a,\U_b]}\widehat{\Pi}_0\left(-\frac{\bar{\pee}_h}{\alpha}\right)-\Pi_{[\U_a,\U_b]}\Pi_0\left(-\frac{\bar{\pee}_h}{\alpha}\right)\Bigr\rVert_{\T\setminus\widehat{\T}}^2 \\
		& \leq \alpha^{-2}\Bigl\lVert\widehat{\Pi}_0\left(-\frac{\bar{\pee}_h}{\alpha}\right)-\Pi_{0}\left(-\frac{\bar{\pee}_h}{\alpha}\right)\Bigr\rVert_{\T\setminus\widehat{\T}}^2=\alpha^{-2}\Bigl\lVert\widehat{\Pi}_{0}\left(-\frac{\bar{\pee}_h}{\alpha}\right)-\widehat{\Pi}_{0}\Pi_{0}\left(-\frac{\bar{\pee}_h}{\alpha}\right)\Bigr\rVert_{\T\setminus\widehat{\T}}^2\\
		&\leq \alpha^{-2}\Bigl\lVert\left(-\frac{\bar{\pee}_h}{\alpha}\right)-\Pi_{0}\left(-\frac{\bar{\pee}_h}{\alpha}\right)\Bigr\rVert_{\T\setminus\widehat{\T}}^2 \leq  \alpha^{-4} C_{\rm{PI}}^2\enorm{h_{\T} \phb}_{\T\setminus\widehat{\T}}^2\leq \alpha^{-2}C_{\rm{PI}}^2\eta_{C}^2(\mathcal{R}(\T,\widehat{\T})).
	\end{aligned}
\end{equation} 
\noindent Use \eqref{cr}-\eqref{6.7} to conclude the discrete reliability for $\Lambda_3^2:=3C_{\rm R}^2\max\{C_{\rm S}^2, C_{\rm A}^2, \alpha^{-2}C_{\rm PI}^2\}$. 
\end{proof} 

\subsection{Proof of (A4)}
\label{proof of a4d}
The proof of $\textbf{(A4)}$ follows from $(\textbf{A4}_{\epsilon})$ as mentioned in \cite[Theorem 3.1]{CCRH17}. Assume $0<\varepsilon<8C_{\rm drl}^2$ and choose maximum $\delta$ such that
\begin{equation}
    \label{constants2}
    \max\{\Delta_5, \Delta_6, \Delta_7, \Delta_8\}\leq \varepsilon
\end{equation}
with $\Delta_i$ for $i=1,2,3,4$ as defined in Table \ref{tab:constants-horizontal-2}. 

\begin{table}[hbt!]
\footnotesize
\centering
\renewcommand{\arraystretch}{1.4}
\begin{tabular}{|c|c|c|c|}
\hline
\textbf{Constant} & \textbf{Definition} & \textbf{Constant} & \textbf{Definition} \\
\hline
$C_{\rm A46}$ & $\max\{2C_{\rm I}, C_{\rm A45}\}$ & $C_{\rm A47}$ & $\max\{\max\{1, C_{\rm dP}\}, 2C_{\rm rel}+C_{\rm rst}(1+\Lambda_{\rm mon})\}$ \\
\hline
$C_{\rm A48}$ & $\max\{\max\{1, C_{\rm dP}\}, 2C_{\rm rel}+C_{\rm sadj}(1+\Lambda_{\rm mon})\}$ & $C_{\rm A49}$ & $2\max\{C_{\rm PI}, C_{\rm A47}, C_{\rm A48}\}$ \\
\hline
$C_{\rm sa}$ & $2\max\{C_{\rm A46}, C_{\rm A47}\}$ & $C_{\rm drl}$ & $C_{\rm drel}$ \\
\hline
$\Lambda_6$ & $C_{\rm drl}C_{\rm sa}$ & $\Delta_5$ & $2^4C_{\rm drl}^2C_{\rm sa}^2\delta^{s}(2\delta^{s}+1)$ \\
\hline
$\Delta_6$ & $2^4(2^{-2}\delta^{s}+2C_{\rm sa}^2\Lambda_{\rm mon}^2\delta^{2s})$ & $\Delta_7$ & $3\times 2^6\Lambda_6^2\Lambda_0^2\delta^2$ \\
\hline
$\Delta_8$ & $2^{2/3}2^{11/3}\Lambda_6^{4/3}\Lambda_0^{2/3}\delta^{2/3}(1+\Lambda_{\rm mon}^2)^{1/3}$ & & \\
\hline
\end{tabular}
\caption{Definitions of various constants used in proof of $(\textbf{A4}_{\epsilon})$.}
\label{tab:constants-horizontal-2}
\end{table}

\noindent Here $(\textbf{A4}_{\epsilon})$ is as mentioned in Theorem \ref{A4ep}. Most of the steps detailed below closely follow the proof of \textbf{(A4)} for the variational approach, as presented in Section \ref{proof of a4v}. For the discretised control approach, the proof of \textbf{(A4)} is outlined below, highlighting the key estimates and referencing the corresponding results from the variational case.
\begin{proof}
{\bf Step 1 : (Key estimates)} Given any $l\in \mathbb{N}\cup\{0\}$ define $e_l^2:=\|\bar{\bf u}-\bar{\bf u}_l\|^2+\enorm{\bar{\y}-\bar{\y}_l}_{\rm{pw}}^2+\enorm{\bar{\pee}-\bar{\pee}_l}_{\rm{pw}}^2$ and ${\bf d}_{k,k+1}:=\|\bar{\bf u}_{k+1}-\bar{\bf u}_k\|^2+\enorm{\bar{\bf y}_{k+1}-\bar{\bf y}_k}_{\rm pw}^2+\enorm{\bar{\bf p}_{k+1}-\bar{\bf p}_k}_{\rm pw}^2$. For any $k\in\mathbb{N}\cup\{0\}$, definition of ${\bf d}_{k,k+1}$ and $e_k$ plus a re-grouping of the terms show 
\begin{equation}
\begin{aligned}
	\label{6.10}
		\frac{1}{2}({\bf d}_{k,k+1}^2+e_{k+1}^2-e_{k}^2)&=a_h(\bar{\y}_{k+1}-\bar{\y}_k,\bar{\y}_{k+1}-\bar{\y})+a_h(\bar{\pee}_{k+1}-\bar{\pee}_k,\bar{\pee}_{k+1}-\bar{\pee})\\
        &+(\bar{\bf u}_{k+1}-\bar{\bf u}_k, \bar{\bf u}_{k+1}-\bar{\bf u}):=E_1+E_2+E_3. 
\end{aligned}
\end{equation}
{\bf Step 2 : (Estimates for $E_1$ and $E_2$)}
Estimate from \eqref{5115}, Lemma \ref{interpolation}{\it (ii)} and {\it (iv)} along with stability of $I_{k}$ operator result in
\begin{equation*}
	a_h(\bar{\y}_{k+1}-\bar{\y}_k,\bar{\y}_{k+1}-\bar{\y})\leq 2C_{\rm{I}} \sum_{K\in \T_k\backslash \T_{k+1}} h_{K} \|f+\bar{\U}_k\| \enorm{\bar{\y}_{k+1}-\bar{\y}}_{\rm{pw}}+\|\bar{\U}_{k+1}-\bar{\U}_k\|\|\bar{\y}_{k+1}-\bar{\y}\|.
\end{equation*}
Use \eqref{5.16} to get
\begin{equation*}
	a_h(\bar{\y}_{k+1}-\bar{\y}_k,\bar{\y}_{k+1}-\bar{\y})\leq 2C_{\rm I} \mu_{S,k}(\T_{k}\backslash\T_{k+1})e_{k+1}+\|\bar{\U}_{k+1}-\bar{\U}_k\|\|\bar{\y}_{k+1}-\bar{\y}\|. 
\end{equation*}
Use of \eqref{518} gives
\begin{equation}
    \label{6.11}
	E_1=a_h(\bar{\y}_{k+1}-\bar{\y}_k,\bar{\y}_{k+1}-\bar{\y})\leq C_{\rm A46}(\mu_{S,k}(\T_{k}\backslash\T_{k+1})e_{k+1}+\delta^{s}{\bf d}_{k,k+1}(e_{k+1}+\eta_{k+1})). 
\end{equation}
\noindent Analogous calculation leads to
\begin{equation}
	\label{6.12}
	E_2=a_h(\bar{\pee}_{k+1}-\bar{\pee}_k,\bar{\pee}_{k+1}-\bar{\pee})\leq C_{\rm A46}(\mu_{A,k}(\T_{k}\backslash\T_{k+1})e_{k+1}+\delta^{s}{\bf d}_{k,k+1}(e_{k+1}+\eta_{k+1})).
\end{equation}
{\bf Step 3 : (Estimates for $E_3$)}
\noindent Using Theorem \ref{d error equivalence}{\it (iv)} one can observe
\begin{equation}
	\label{6.13}
	\|\bar{\U}_{k+1}-\bar{\U}_k\|\leq C_{\rm 26}(\|\widetilde{\bf u}_{k+1}-\bar{\bf u}_k\|+\|\widehat{\widetilde{\bf y}}_{k+1}-\bar{\bf y}_k\|+\|\widehat{\widetilde{\bf p}}_{k+1}-\bar{\bf p}_k\|):=C_{\rm 26}(I_1+I_2+I_3).
\end{equation}
\noindent Estimate in \eqref{6.7} gives bound on $I_1$ as 
\begin{equation}
\label{i1}
	I_1=\|\widetilde{\bf u}_{k+1}-\bar{\bf u}_k\|\leq C_{\rm PI}\eta_{C,k}(\T_{k}\setminus\T_{k+1}). 
\end{equation}
The use of the triangle inequality by introducing the term $\widetilde{\bf y}$ in $I_2$ along with \cite[Theorem 4.1]{MR3194820}, Lemma \ref{a pos} give
\begin{equation*}
    \|\widehat{\widetilde{\bf y}}_{k+1}-\bar{\bf y}_k\|\leq \delta^{s}(\enorm{\widehat{\widetilde{\bf y}}_{k+1}-\widetilde{\bf y}}_{\rm pw}+\enorm{{\bar{\bf y}}_{k}-\widetilde{\bf y}}_{\rm pw}+C_{\rm rst}(\eta_{\rm st}(\T_{k+1})+\eta_{\rm st}(\T_k)).
\end{equation*}
A use of triangle inequality, by introducing the term $\bar{\bf y}_k$ in first term of above estimate gives
\begin{equation*}
    \|\widehat{\widetilde{\bf y}}_{k+1}-\bar{\bf y}_k\|\leq \delta^{s}(\enorm{\widehat{\widetilde{\bf y}}_{k+1}-\bar{\bf y}_k}_{\rm pw}+2\enorm{{\bar{\bf y}}_{k}-\widetilde{\bf y}}_{\rm pw}+C_{\rm rst}(\eta_{\rm st}(\T_{k+1})+\eta_{\rm st}(\T_k)).
\end{equation*}
Lemma \ref{lemma 4.2} gives
\begin{equation*}
    \|\widehat{\widetilde{\bf y}}_{k+1}-\bar{\bf y}_k\|\leq \delta^{s}(\enorm{\widehat{\widetilde{\bf y}}_{k+1}-\bar{\bf y}_k}_{\rm pw}+2C_{\rm rel}\eta_{\rm st}(\T_k)+C_{\rm rst}(\eta_{\rm st}(\T_{k+1})+\eta_{\rm st}(\T_k)).
\end{equation*}
Use of triangle inequality by introducing the term $\bar{\bf y}_{k+1}$, \eqref{dst} along with monotonicity property as $\eta_{\rm st}
(\T_{k+1})\leq \Lambda_{\rm mon}\eta(\T_k)$ yield
\begin{equation}
\label{i2}
   I_2=\|\widehat{\widetilde{\bf y}}_{k+1}-\bar{\bf y}_k\|\leq C_{\rm A47}\delta^{s}({\bf d}_{k,k+1}+\eta_{\rm st}(\T_k)).
\end{equation}
Similar calculation yields
\begin{equation}
\label{i3}
    I_3=\|\widehat{\widetilde{\bf p}}_{k+1}-\bar{\bf p}_k\|\leq C_{\rm A48}\delta^{s}({\bf d}_{k,k+1}+\eta_{\rm adj}(\T_k)).
\end{equation}
Substitution of \eqref{i1}-\eqref{i3} in \eqref{6.13} leads to
\begin{equation}
\label{baru}
        \|\bar{\U}_{k+1}-\bar{\U}_k\|\leq C_{\rm A49}(\eta_{C,k}(\T_k\setminus\T_{k+1})+\delta^{s}({\bf d}_{k,k+1}+\eta_k)).
\end{equation}
Now consider term $E_3=(\bar{\bf u}_{k+1}-\bar{\bf u}_k, \bar{\bf u}_{k+1}-\bar{\bf u})$. Use of Cauchy Schwartz inequality and \eqref{baru} gives
\begin{equation}
\label{e3}
    E_3\leq C_{\rm A49}(\eta_{C,k}(\T_k\setminus\T_{k+1})+\delta^{s}({\bf d}_{k,k+1}+\eta_k))e_{k+1}.
\end{equation}
{\bf Step 4 : (Crucial estimates)}
Substitution of \eqref{6.11}, \eqref{6.12}, and \eqref{e3} in \eqref{6.10} gives
\begin{equation*}
    ({\bf d}_{k,k+1}^2+e_{k+1}^2-e_{k}^2)\leq C_{\rm sa}(\mu_{k}(\T_k\setminus\T_{k+1})+\delta^{s}{\bf d}_{k,k+1}+\delta^{s}\eta_{k})e_{k+1}+C_{\rm sa}\delta^{s}\eta_{k+1}{\bf d}_{k,k+1}.
\end{equation*}
Use of quasi-monotonicity property for $\eta_{k+1}$ results in
\begin{equation}
\label{csa}
    ({\bf d}_{k,k+1}^2+e_{k+1}^2-e_{k}^2)\leq C_{\rm sa}(\mu_{k}(\T_k\setminus\T_{k+1})+\delta^{s}{\bf d}_{k,k+1}+\delta^{s}\eta_{k})e_{k+1}+C_{\rm sa}\delta^{s}\Lambda_{\rm mon}\eta_{k}{\bf d}_{k,k+1}.
\end{equation}
Apply Young's inequality in \eqref{csa} for $a=C_{\rm sa}\mu_k(\T_k\setminus\T_{k+1}), b=e_{k+1}$ for $\epsilon=2^{-3}\varepsilon C_{\rm drl}^{-2}$, $a=C_{\rm sa}\delta^{s}{\bf d}_{k,k+1}, b=e_{k+1}$ for $\epsilon=2^{2}C_{\rm sa}^2\delta^{2s}$, $a=C_{\rm sa}\delta^{s}\eta_k, b=e_{k+1}$ for $\epsilon=2C_{\rm sa}^2\delta^{s}$ and $a=C_{\rm sa}\Lambda_{\rm mon}\delta^{s}{\bf d}_{k,k+1}$, $b=\eta_{k}$ for $\epsilon=2^2C_{\rm sa}^2\Lambda_{\rm mon}^2\delta^{2s}$ gives
\begin{align}
    ({\bf d}_{k,k+1}^2+e_{k+1}^2-e_{k}^2)\leq &2^2\varepsilon^{-1}C_{\rm drl}^{2}C_{\rm sa}^2\mu_{k}^2(\T_{k}\setminus\T_{k+1})+2^{-4}\varepsilon C_{\rm drl}^{-2}e_{k+1}^2+2^{-3}{\bf d}_{k,k+1}^2+2C_{\rm sa}^2\delta^{2s}e_{k+1}^2\nonumber\\
    &+2^{-2}\delta^{s}\eta_k^2+C_{\rm sa}^2\delta^{s}e_{k+1}^2+2^{-3}{\bf d}_{k,k+1}^2+2C_{\rm sa}^2\Lambda_{\rm mon}^2\delta^{2s}\eta_k^2. 
\end{align}
Some re-arrangement of terms yields
\begin{equation*}
    \frac{3}{4}{\bf d}_{k,k+1}^2+e_{k+1}^2-e_{k}^2\leq 2^2\varepsilon^{-1}\Lambda_6^2\mu_{k}^2(\T_{k}\setminus\T_{k+1})+(2^{-4}\varepsilon C_{\rm drl}^{-2}+2C_{\rm sa}^2\delta^{2s}+C_{\rm sa}^2\delta^{s})e_{k+1}^2+(2^{-2}\delta^{s}+2C_{\rm sa}^2\Lambda_{\rm mon}^2\delta^{2s})\eta_{k}^2.
\end{equation*}
Using {$2C_{\rm sa}^2\delta^{2s}+C_{\rm sa}^2\delta^{s}\leq 2^{-4}\varepsilon C_{\rm drl}^{-2}$} and {$2^{-2}\delta^{s}+2C_{\rm sa}^2\Lambda_{\rm mon}^2\delta^{2s}\leq 2^{-4}\varepsilon$} from \eqref{constants2} we can observe 
\begin{equation*}
    \frac{3}{4}{\bf d}_{k,k+1}^2+e_{k+1}^2-e_{k}^2\leq 2^2\varepsilon^{-1}\Lambda_6^2\mu_{k}^2(\T_{k}\setminus\T_{k+1})+2^{-3}\varepsilon C_{\rm drl}^{-2}e_{k+1}^2+2^{-4}\varepsilon\eta_k^2.
\end{equation*}
The estimate for the term $\mu_k^2(\T_k\setminus\T_{k+1})$ follows from \eqref{5.19}. Note that the above estimate mirrors the equation \eqref{k+1}, with constants defined as in \eqref{constants2}. Hence, by proceeding with similar calculations as outlined in Section \ref{proof of a4v} following \eqref{k+1}, desired result is obtained as
    \begin{equation*}
\sum_{k=l}^{l+m}{\bf d}_{k,k+1}^2\leq\Lambda_{4(\varepsilon)}\eta_{l}^2+\varepsilon \sum_{k=l}^{l+m}\eta_{k}^2
\end{equation*}
with $\Lambda_{4(\varepsilon)}:=4+2^{6}\varepsilon^{-1}\Lambda_{6}^2.$
\end{proof}
\section{Numerical Experiments}
\label{numerical experiments}
\noindent The primary aim of this section is to validate the theoretical results, including the a priori estimates stated in Theorems \ref{Theorem-4.1} and \ref{theorem 4.2}, and to assess the efficiency and reliability of the estimator introduced in Section \ref{adaptive convergence}, thereby confirming Theorems \ref{theorem 4.3} and \ref{theorem 4.4}.

 The projected gradient method, together with the standard adaptive algorithm presented in \cite[Algorithm 1]{MR4766712}, is applied to solve the variational control formulation. In contrast, the discretised control formulation is addressed using a primal-dual active set approach \cite[Section 2.12.4]{TF2010}.

\noindent Recall from Section \ref{main result}, the optimality system comprises of state equation \eqref{cstate}, adjoint equation \eqref{cadjoint}, and optimality condition mentioned in \eqref{cop}. The discrete control variable $\bar{\bf u}_h$ has representation in terms of $\bar{\bf p}_h$ as mentioned in \eqref{defuh}. Recall $(\bar{\bf y},\bar{\bf p})\in {\bf V}\times {\bf V}$ and $(\bar{r},\bar{s})\in Q\times Q$ are regarded optimal velocity and pressure variables respectively. Throughout the section total error comprises of state error, adjoint error, and control error defined as ${\rm TE}:=\|E_C\|+\enorm{E_S}_{\rm pw}+\enorm{E_A}_{\rm pw}$. Here notation follows from Table \ref{tab:error_notation_horizontal}. The convergence rate is evaluated with respect to the total number of degrees of freedom (NDOF). Since the velocity is discretised using lowest order Crouzeix–Raviart (CR) elements, its degrees of freedom are associated with the edges of the mesh. However, the pressure is approximated using piecewise constant elements. Consequently, the total NDOF is given by $2\times \#\E+\#\T$. 



\subsection{Convex domain example}
\label{convex domain}
The computational domain is defined as $\Omega = [0,1]^2$, with control constraints specified by the vectors $\mathbf{u}_a = (-200, -200)^T$ and $\mathbf{u}_b = (200, 200)^T$. The load function $\mathbf{f}$ and the desired state $\mathbf{y}_d$ are constructed to ensure that the exact solutions for the state and adjoint variables are :
\begin{equation*}
	\yb=\pb=[\sin^2(\pi x)\sin(\pi y)\cos(\pi y), -\sin^2(\pi y)\sin(\pi x)\cos(\pi x)]^T\mbox{ and } \rb=\sbar=\sin(2\pi x)\sin(2\pi y). 
\end{equation*} The exact control is obtained using the formula $\ub=\Pi_{[\U_a,\U_b]}(\dfrac{-\pb}{\alpha})$, with the regularization parameter $\alpha$ set to $10^{-3}$.\\ 

\noindent This benchmark problem is adapted from the work presented in \cite{MR3924212}. 
The numerical experiments commence with an initial triangulation $\mathcal{T}_0$ consisting of 8 elements.  The initial control $\mathbf{u}_0$ is initialized as the average of the lower and upper control bounds, i.e., $\mathbf{u}_0 = (\mathbf{u}_a + \mathbf{u}_b)/2$. The Dörfler marking criterion is employed with a marking parameter $\theta = 0.25$ to guide the adaptive mesh refinement process.
\begin{table}[hbt!]
\footnotesize
	\begin{tabular}{|c||c|c|c|c|c|c|c|c|c|c|c|c|}
		\hline
		\multirow{2}{*}{NDOF} &
		\multicolumn{12}{|c|}{Error and order of convergence from variational approach} \\ 
		\cline{2-13} 
		& $\enorm{\bar{\y}-\bar{\y}_h}$ & {order} & $\|\bar{r}-\bar{r}_h\|$ & order & $\enorm{\bar{\bf p}-\bar{\bf p}_h}$ & order & $\|\bar{s}-\shbar\|$ & order & $\|\bar{\bf u}-\bar
        {\bf u}_h\|$ & order & TE & order \\[2pt] 
         \hline
		 604 & 0.1280 & - & 0.1356 & - & 1.1281 & - & 0.1355 & - & 0.0841 & - & 1.1263 & -  \\[2pt] 
		 1083 & 0.0901  & 0.60  & 0.1093 & 0.49  & 0.0902 & 0.59 & 0.1092 & 0.49 & 0.0725 & 0.85 & 0.6371 & 0.84 \\[2pt] 
		3454 & 0.0730   & 0.42  & 0.0692  & 0.49  & 0.0730 & 0.43 & 0.0692 & 0.92 & 0.0681 & 0.73 & 0.3774 & 0.87   \\[2pt] 
		5909 & 0.0500 & 0.56  & 0.0564  & 0.59  & 0.0500 & 0.56 & 0.0564 & 0.59 & 0.0424 & 0.86 & 0.2177 & 0.59 \\ [2pt] 
		10139 & 0.0390 & 0.46 & 0.0426 & 0.59 & 0.0390 & 0.46 & 0.0352 & 0.64 & 0.0321 & 0.75 & 0.1958 & 0.55   \\[2pt] 
		16382 & 0.0216 & 0.50 & 0.2147 & 0.55 & 0.0211 & 0.59 & 0.0248 & 0.51 & 0.0198 & 1.06 &  0.0595 & 0.51   \\
		\hline
	\end{tabular}
	\caption{Error and observed order of convergence for the velocity, pressure, control variable, and total error in the variational approach for example \ref{convex domain}.}
    \label{table 4}
\end{table}

\begin{figure}[hbt!]
	\centering
	\subcaptionbox{approximate velocity profiles of state equation.}{\includegraphics[width=0.3\linewidth]{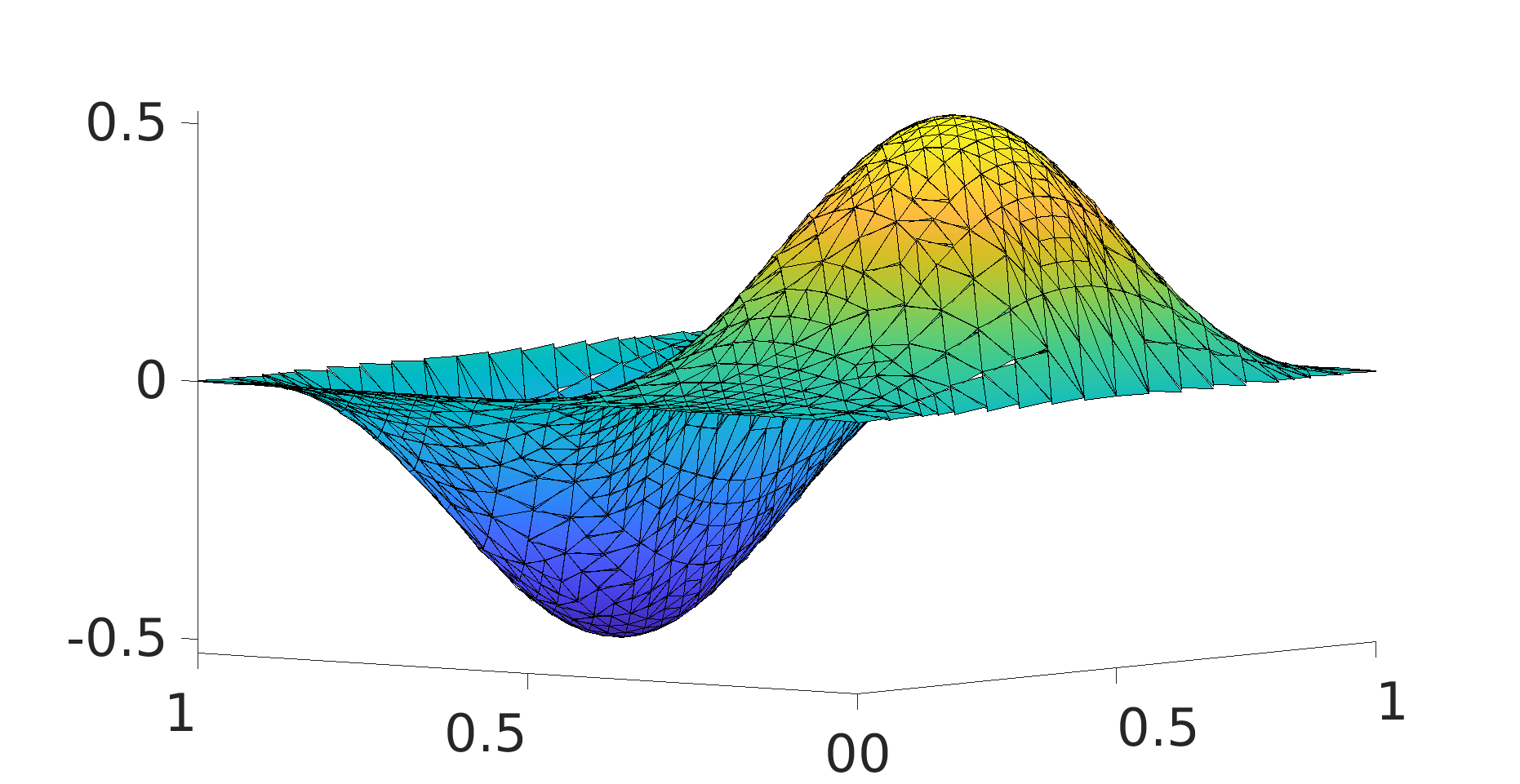}
		{\includegraphics[width=0.3\linewidth]{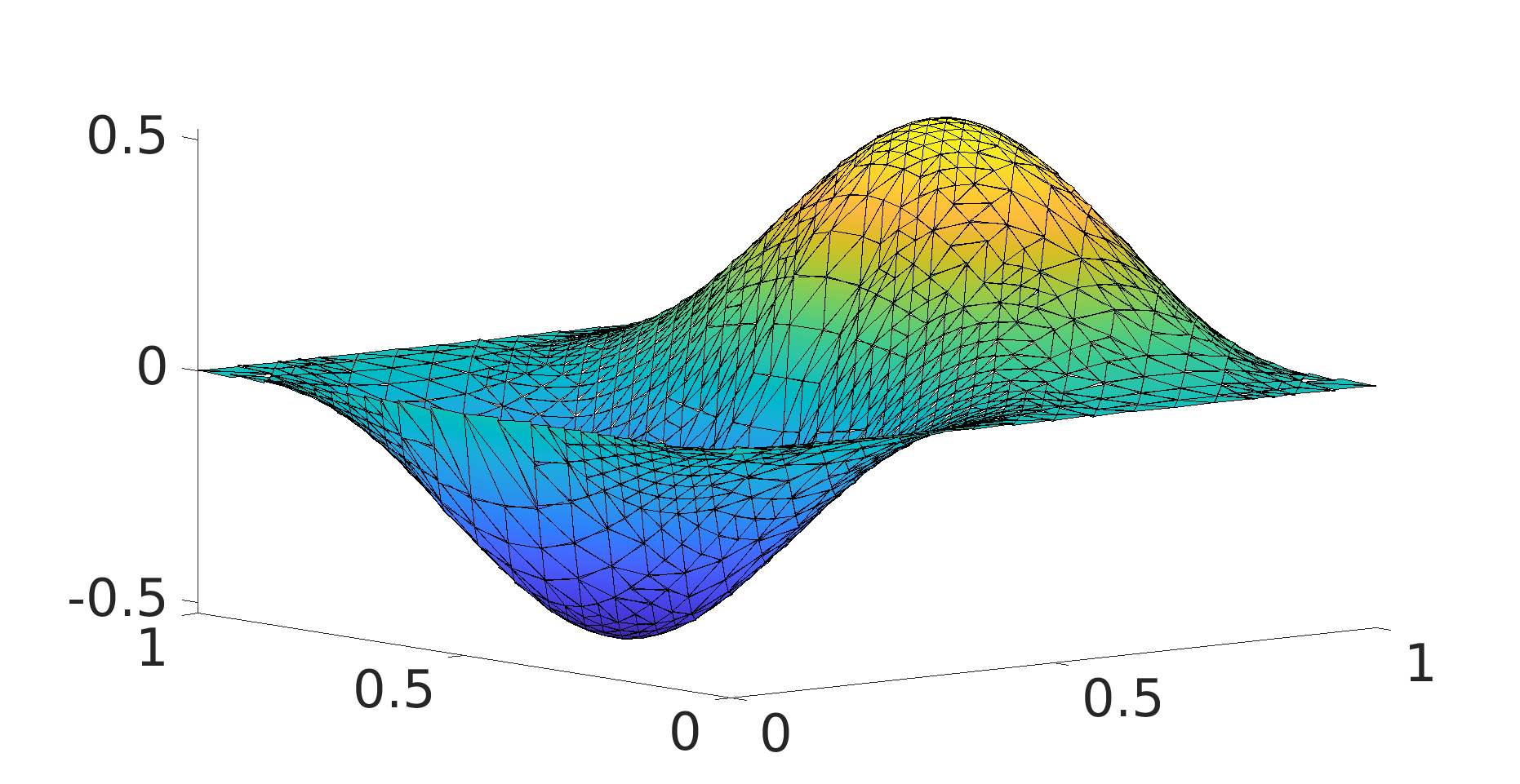}}}
	\subcaptionbox{pressure profile of state equation.}{\includegraphics[width=0.3\linewidth]{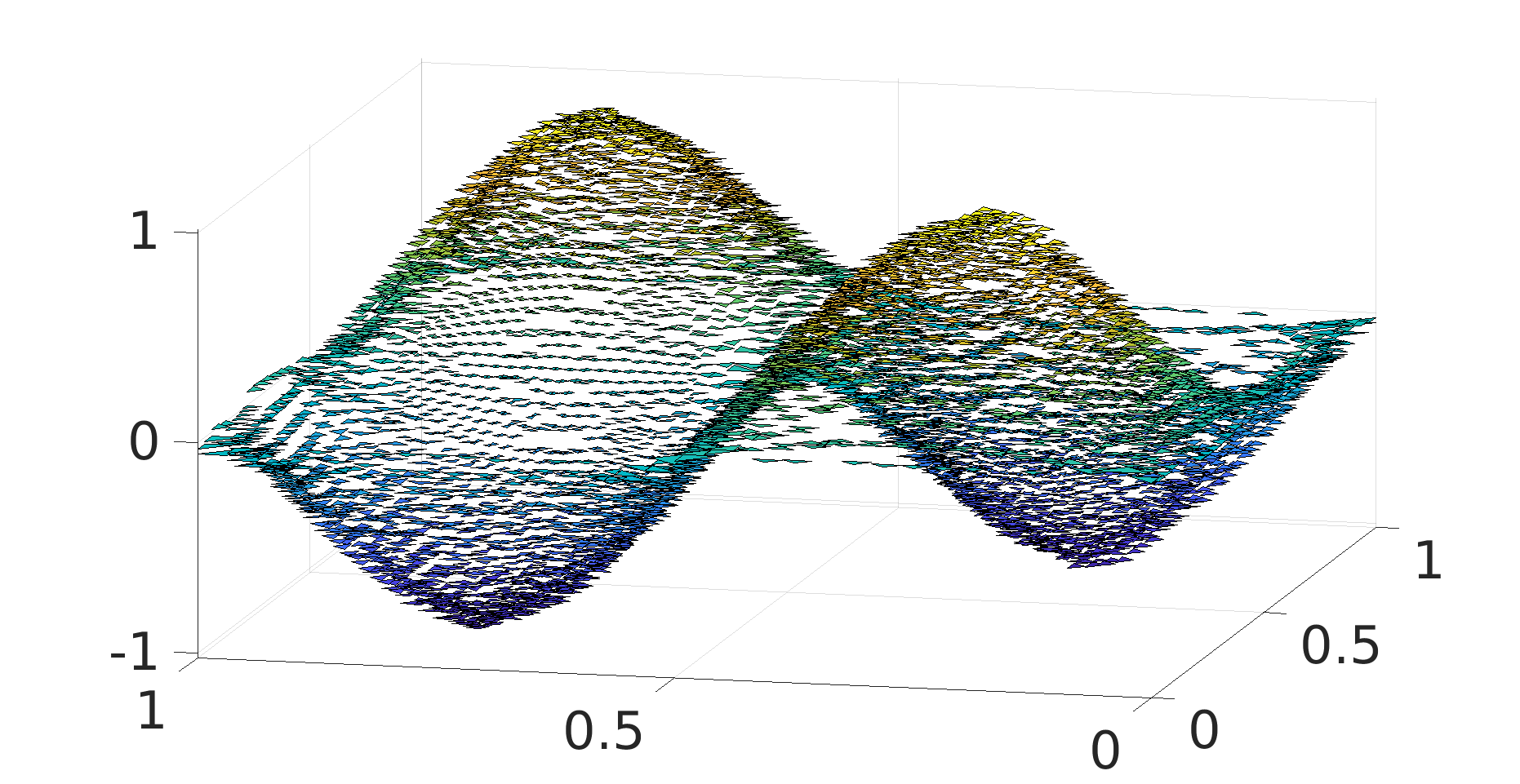}}
	\caption{Velocity and Pressure profiles of state equation with variational formulation.}
    \label{figure 1}
\end{figure}
\noindent Table \ref{table 4} presents the numerical errors and corresponding convergence orders for the velocity and pressure variables, along with the control and total error. The velocity variable demonstrates a convergence order of 0.5 ($h \approx \text{NDOF}^{-1/2}$) in the piecewise norm, while the pressure and control variables exhibit convergence orders of approximately 0.5 and 1 respectively, consistent with the theoretical result mentioned in Theorem \ref{Theorem-4.1}. Furthermore, the total error, as reported in Table \ref{table 4}, achieves the anticipated convergence rate of 0.5. A similar trend is observed in the discretised control approach, as evidenced by the results in Tables \ref{table 5}, where the velocity, pressure, control, and total errors also converge at the expected rate of 0.5. Table~\ref{table 7} presents the values of the state, adjoint, and control error estimators along with their rate of convergence obtained for the discretised approach. 
It can be observed that all estimators converge with the optimal rate of $0.5$, demonstrating the efficiency and reliability of the proposed estimator. A comparable behavior is also observed for the variational approach, although those results are not displayed here for brevity. \\
\noindent Figure \ref{figure 1} presents the velocity and pressure profiles corresponding to the state equation in variational formulation. Due to the close resemblance between the adjoint and state equation profiles, the adjoint profiles are omitted for brevity. Figures \ref{figure 2} and \ref{figure 3} display the control profiles obtained from the variational and discretised approaches, respectively. Recall, control variable is computed using the projection formula as $\uhb=\Pi_{[\U_a,\U_b]}(\frac{-\phb}{\alpha})$, where $\Pi_{[\mathbf{u}_a, \mathbf{u}_b]}$ denotes the projection onto the admissible control set defined by the bounds $\mathbf{u}_a$ and $\mathbf{u}_b$, and $\alpha$ is the regularization parameter. As observed from Figure \ref{figure 1}, $\bar{\bf p}_h$ varies between $-0.5$ and $0.5$. With the regularization parameter set to $\alpha = 10^{-3}$, the unprojected control variable $\bar{\mathbf{u}}$ would range between $-500$ and $500$. However, given the control bounds $\mathbf{u}_a = (-200, -200)^T$ and $\mathbf{u}_b = (200, 200)^T$, the projected control values are constrained within these limits. This projection effect is evident in Figures \ref{figure 2} and \ref{figure 3}, where the flattening of the control profile surfaces indicates the enforcement of the control constraints.
\begin{table}[hbt!]
\footnotesize
    \centering
	\begin{tabular}{|c||c|c|c|c|c|c|c|c|c|c|c|c|}
		\hline
		\multirow{2}{*}{NDOF} &
		\multicolumn{12}{|c|}{Error and order of convergence from discretised approach} \\ 
		\cline{2-13} 
		& $\enorm{\bar{\y}-\bar{\y}_h}$ & {order} & $\|\bar{r}-\bar{r}_h\|$ & order & $\enorm{\bar{\bf p}-\bar{\bf p}_h}$ & order & $\|\bar{s}-\shbar\|$ & order & $\|\bar{\bf u}-\bar{\bf u}_h\|$ & order & TE & order  \\[2pt] 
         \hline
		 72 & 0.9210 & - & 0.1694 & - & 0.9329 & - & 0.1707 & - & 0.0986 & - & 2.2928 & - \\[2pt] 
		 244 & 0.6928  & 0.41  & 0.3473 & 0.40  & 0.6931 & 0.42 & 0.1151 & 0.586 & 0.0825 & 0.58 & 2.1679 & 0.49 \\[2pt] 
		544 & 0.5602   & 0.74  & 0.2677  & 0.42  & 0.5603 & 0.75 & 0.0564 & 0.53 & 0.0436 & 0.52 & 0.9348 & 0.87 \\[2pt] 
		1056 & 0.2726 & 0.53  & 0.1126  & 0.52  & 0.3316 & 0.74 & 0.0277 & 0.51 & 0.0113 & 0.45 & 0.8264 & 0.84  \\ [2pt] 
		8320 & 0.1356 & 0.53 & 0.1253 & 0.49 & 0.2726 & 0.51 & 0.0305 & 0.41 & 0.0080 & 0.51 & 0.4608 & 0.44 \\[2pt] 
		16512 & 0.0677 & 0.42 & 0.0564 & 0.68 & 0.0809 & 0.41 & 0.0135 & 0.49 & 0.0042 & 0.49 & 0.2066 & 0.59 \\
		\hline
	\end{tabular}
	\caption{Error and observed order of convergence for the velocity, pressure, control variable, and total error in the discretised approach for example \ref{convex domain}.}
    \label{table 5}
\end{table}
\vspace{-0.5cm}
\begin{table}[hbt!]
\footnotesize
	\centering
	\begin{tabular}{|c||c|c|c|c|c|c|c|c|}
		\hline
		\multirow{2}{*}{NDOF} &
		\multicolumn{2}{|c|}{
			State estimator} &
		\multicolumn{2}{|c|}{Adjoint estimator} & 
		\multicolumn{2}{|c|}{Control estimator} & 
		\multicolumn{2}{|c|}{$\eta^2=\eta_{\rm st}^2+\eta_{\rm adj}^2+\eta_{C}^2$} \\ 
		\cline{2-9} 
		& $\eta_{\rm st}$ & {order} & $\eta_{\rm adj}$ & order & $\eta_{C}$ & order & $\eta$ & order  \\[2pt] 
         \hline
		72 & 9.5166 & - & 9.5576 & - & 0.7631 & - & 13.5092 & -  \\[2pt] 
		144 & 9.4090  & 0.57  & 9.5364 & 0.57  & 0.6445 & 0.57 & 11.3413 & 0.57   \\[2pt]
		544 & 6.5369   & 0.59  & 6.5392  & 0.59  & 0.4046 & 0.51 & 9.2551 & 0.59 \\[2pt] 
		1056 & 3.1035 & 0.50  & 4.3380  & 0.50  & 0.2717 & 0.48 & 4.3931 & 0.50   \\ [2pt]
		8320 & 1.1568 & 0.45 & 2.2685 & 0.45 & 0.1381 & 0.50 & 2.2474 & 0.45  \\[2pt] 
		16512 & 0.7998 & 0.52 & 1.1569 & 0.52 & 0.0971 & 0.49 & 1.2374 & 0.52  \\[2pt] 
		\hline 
	\end{tabular}
	\caption{Order of convergence of state, adjoint, and complete estimator with $\theta=0.25$ of example \ref{convex domain}.}
    \label{table 7}
\end{table}

\begin{figure}[hbt!]
	\centering
	\includegraphics[width=0.45\linewidth]{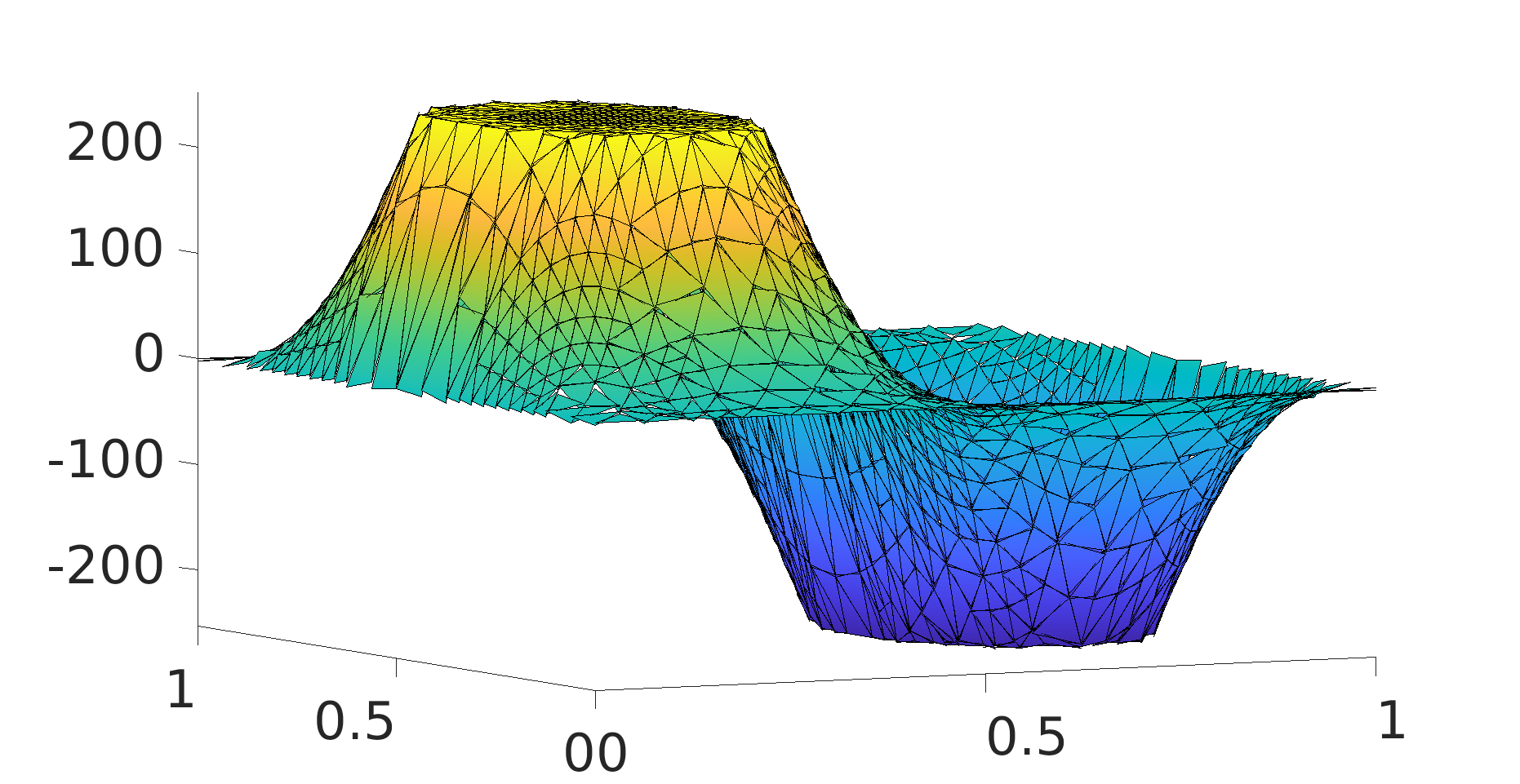}
	\includegraphics[width=0.45\linewidth]{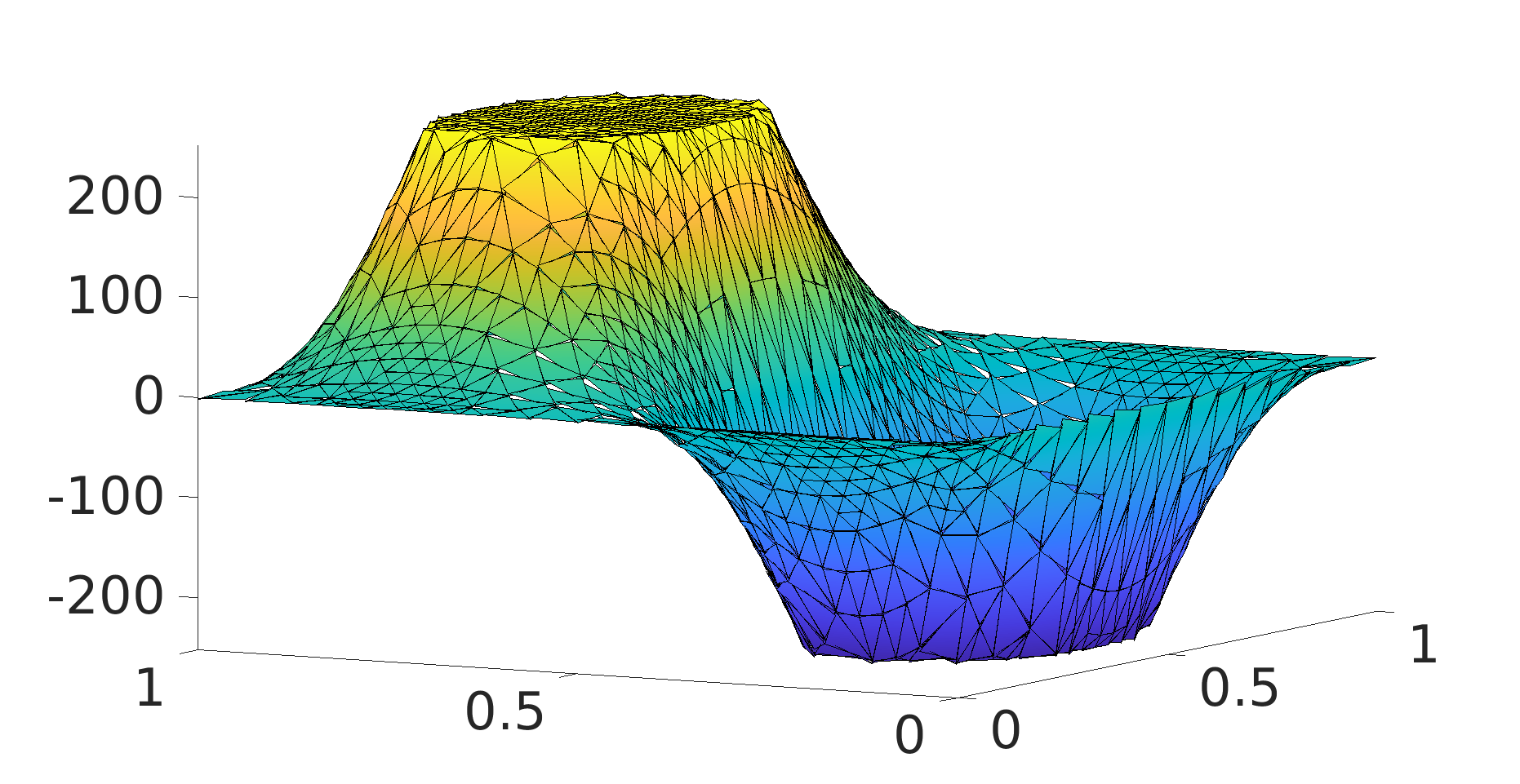}
	\caption{Variational control profiles for example \ref{convex domain}.}
    \label{figure 2}
\end{figure}

\begin{figure}[hbt!]
	\centering
	\includegraphics[width=0.45\linewidth]{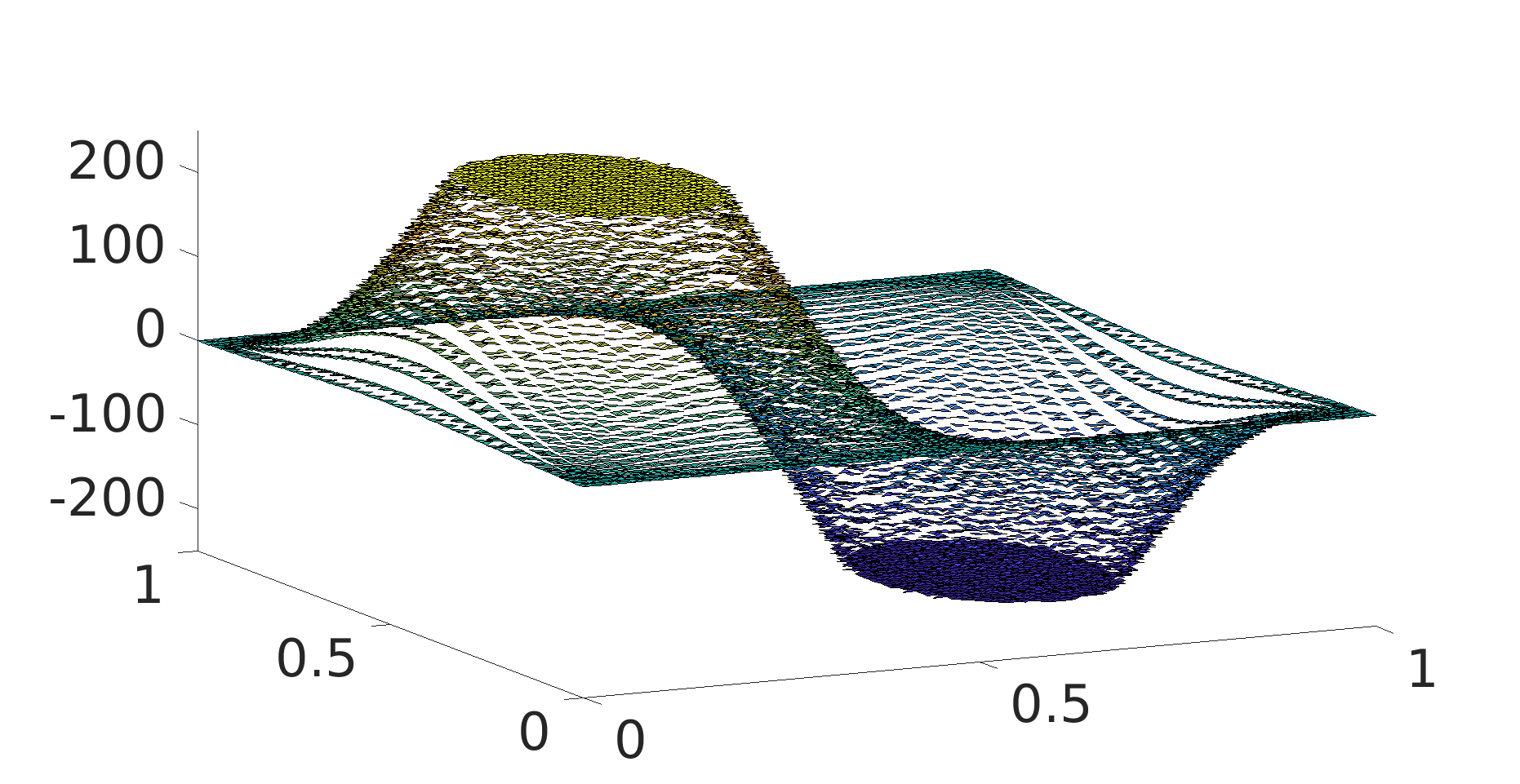}
	\includegraphics[width=0.45\linewidth]{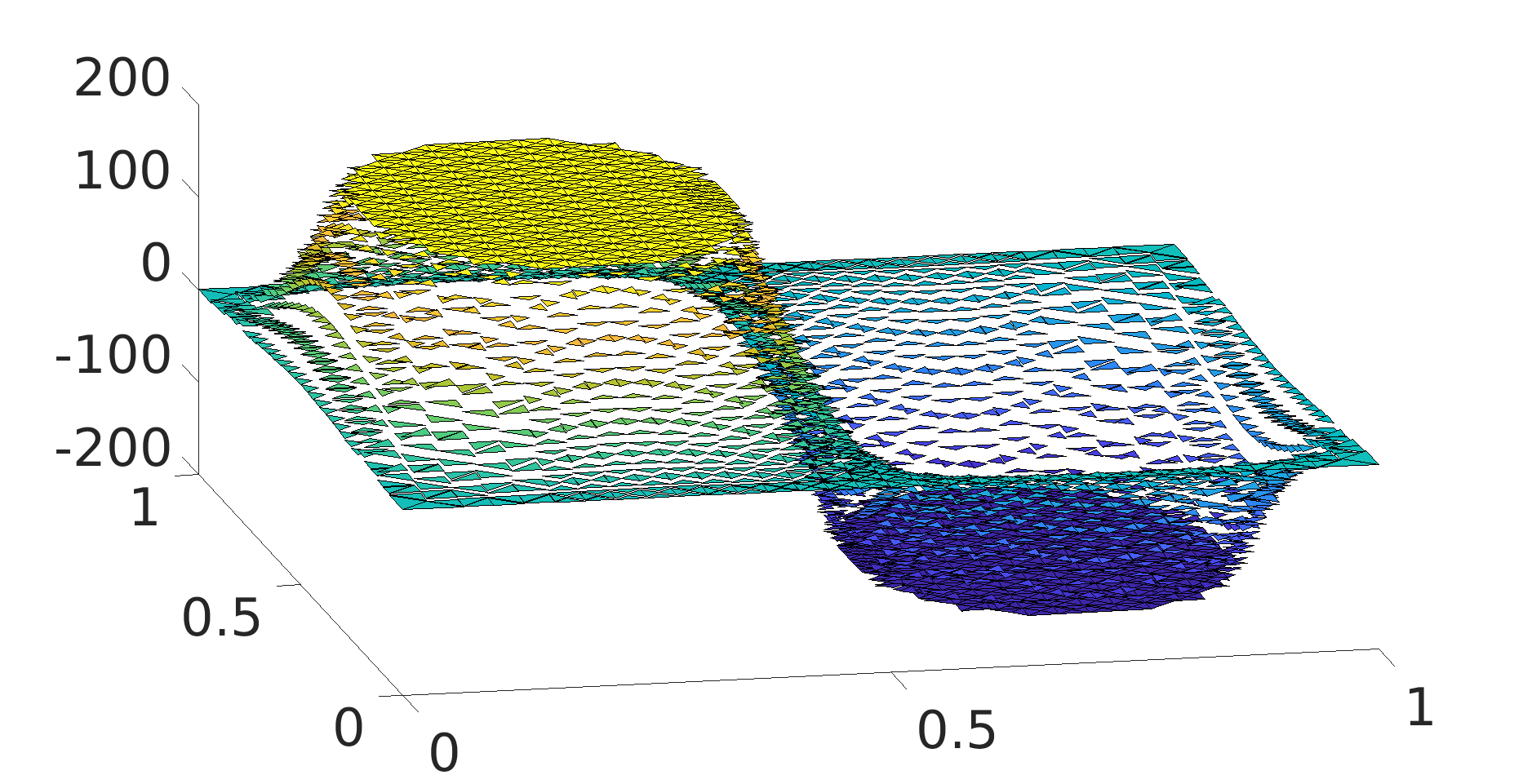}
	\caption{Discretised control profiles for example \ref{convex domain}.}
    \label{figure 3}
\end{figure}

\subsection{Non convex domain example}
\label{nonconvex domain}
Consider the L-shaped domain $\Om=(0,1)^2\setminus [0,1]\times [-1,0]$. Let $(r,\theta)$ be the polar co-ordinates and set $\omega=\frac{3\pi}{2}$, $\U_a=(-200,-200)^{T}$, and $\U_b=(-50,-50)^{T}$. Define the functions
\begin{align*}
	&\phi_1(\theta)=-\sin(\lambda\theta)\cos(\omega)-\lambda\sin(\theta)\cos(\lambda(\omega-\theta)+\theta)+\lambda\sin(\omega-\theta)\cos(\lambda\theta-\theta)+\sin(\lambda(\omega-\theta)),\\
	&\phi_2(\theta)=-\sin(\lambda\theta)\sin(\omega)-\lambda\sin(\theta)\sin(\lambda(\omega-\theta)+\theta)-\lambda\sin(\omega-\theta)\sin(\lambda\theta-\theta),\\
	&\phi_3(\theta)=2\lambda(\sin((\lambda-1)\theta+\omega)+\sin((\lambda-1)\theta-\lambda\omega)).
\end{align*}
Construct the load function $\textbf{f}$ and desired state $\y_d$ such that, the exact solution reads
\begin{equation}
\label{solution}
	\bar{\y}=\bar{\pee}=[r^{\lambda}\phi_1(\theta), r^{\lambda}\phi_2(\theta)]^{T}\mbox{ and } \bar{r}=-\bar{s}=r^{\lambda-1} \phi_3(\theta) 
\end{equation}
with $\lambda=0.5445$ \cite{YWN2022}. The exact control is obtained by using the formula $\bar{\U}=\Pi_{[\U_a,\U_b]}\left(\frac{-\bar{\pee}}{\alpha}\right)$ and regularization parameter $\alpha=10^{-3}$. The numerical experiments begins with an initial triangulation $\mathcal{T}_0$ consisting of 6 elements. The Dörfler marking criterion is used with a marking parameter $\theta = 0.25$ to guide the adaptive mesh refinement process. The initial control $\mathbf{u}_0$ is initialized as the average of the lower and upper control bounds, i.e., $\mathbf{u}_0 = (\mathbf{u}_a + \mathbf{u}_b)/2$.
\begin{figure}[hbt!]
	\centering
	\includegraphics[width=8cm,height=8cm]
	{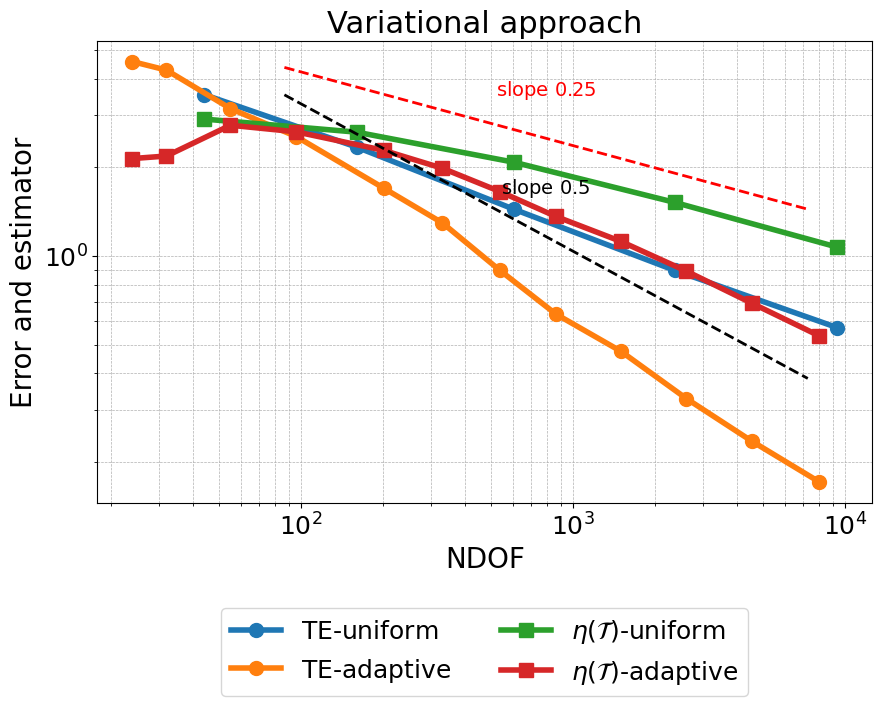}
	\includegraphics[width=8cm,height=8cm]
	{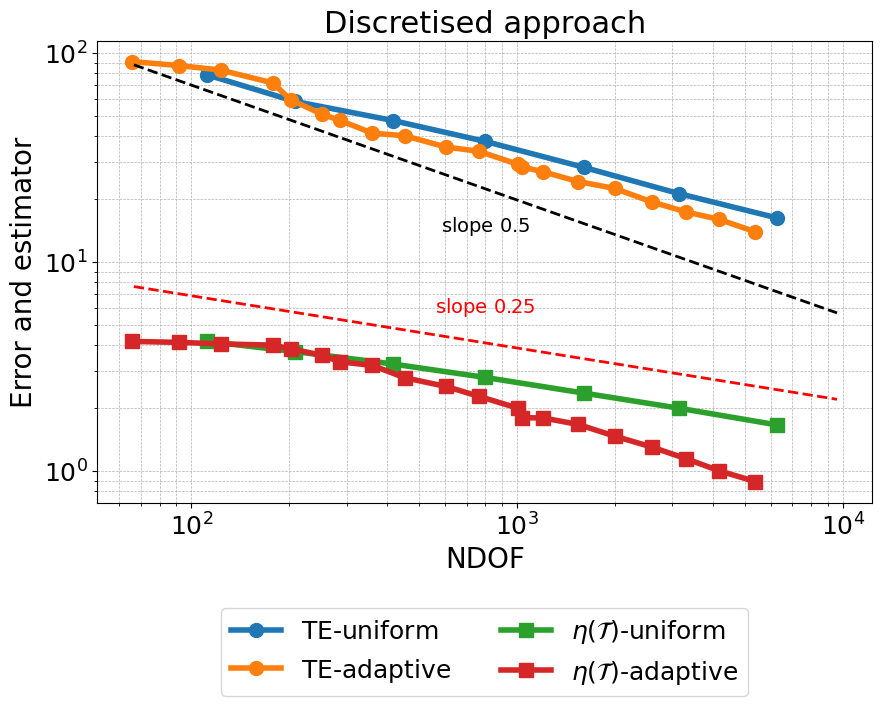}
	\caption{Error plots with uniform and adaptive refinement for total error and complete estimator for example \ref{nonconvex domain}.}
     \label{figure 4}
\end{figure}

The standard elliptic regularity theory implies that the exact solution defined in \eqref{solution} possesses a regularity index of order $1+s$ for all $s<\lambda=0.5445$ in the case of $(\bar{\mathbf y}, \bar{\mathbf p})$, and of order $s<\lambda=0.5445$ for $\bar r$ and $\bar s$. 

Consequently, the total error 
\[
{\rm TE}:=\|E_{\rm C}\|+\enorm{E_{\rm S}}+\enorm{E_{\rm A}}
\]
exhibits a suboptimal convergence rate of order $s$, as predicted by Theorem~\ref{Theorem-4.1} for the variational formulation and Theorem~\ref{theorem 4.2} for the discretised formulation.

In particular, under uniform mesh refinement, one expects a suboptimal rate of order $0.25$ with respect to the number of degrees of freedom (since $h^{0.5}\approx {\rm NDOF}^{-0.25}$). In contrast, adaptive mesh refinement improves this behaviour and recovers the optimal convergence rate of order $0.5$ (i.e., $h\approx {\rm NDOF}^{-1/2}$) for both the variational and discretised formulations as shown in Figure \ref{figure 4}. This confirms the a priori estimates established in Theorem~\ref{Theorem-4.1} for the variational case and Theorem~\ref{theorem 4.2} for the discretised case.

A similar improvement is observed for the complete estimator. In the variational formulation,
\[
\eta^2(\mathcal T)=\eta_{\rm st}^2(\mathcal T)+\eta_{\rm adj}^2(\mathcal T),
\]
while in the discretised formulation,
\[
\eta^2(\mathcal T)=\eta_{\rm C}^2(\mathcal T)+\eta_{\rm st}^2(\mathcal T)+\eta_{\rm adj}^2(\mathcal T).
\]
In both approaches, the estimator reflects the same convergence behaviour as the total error, thereby demonstrating its reliability and efficiency in accordance with Theorems~\ref{theorem 4.3} and~\ref{theorem 4.4} for the variational and discretised cases, respectively.

\begin{figure}[hbt!]
	\centering
	\subcaptionbox{Velocity profiles of state equation}{\includegraphics[width=0.3\linewidth]{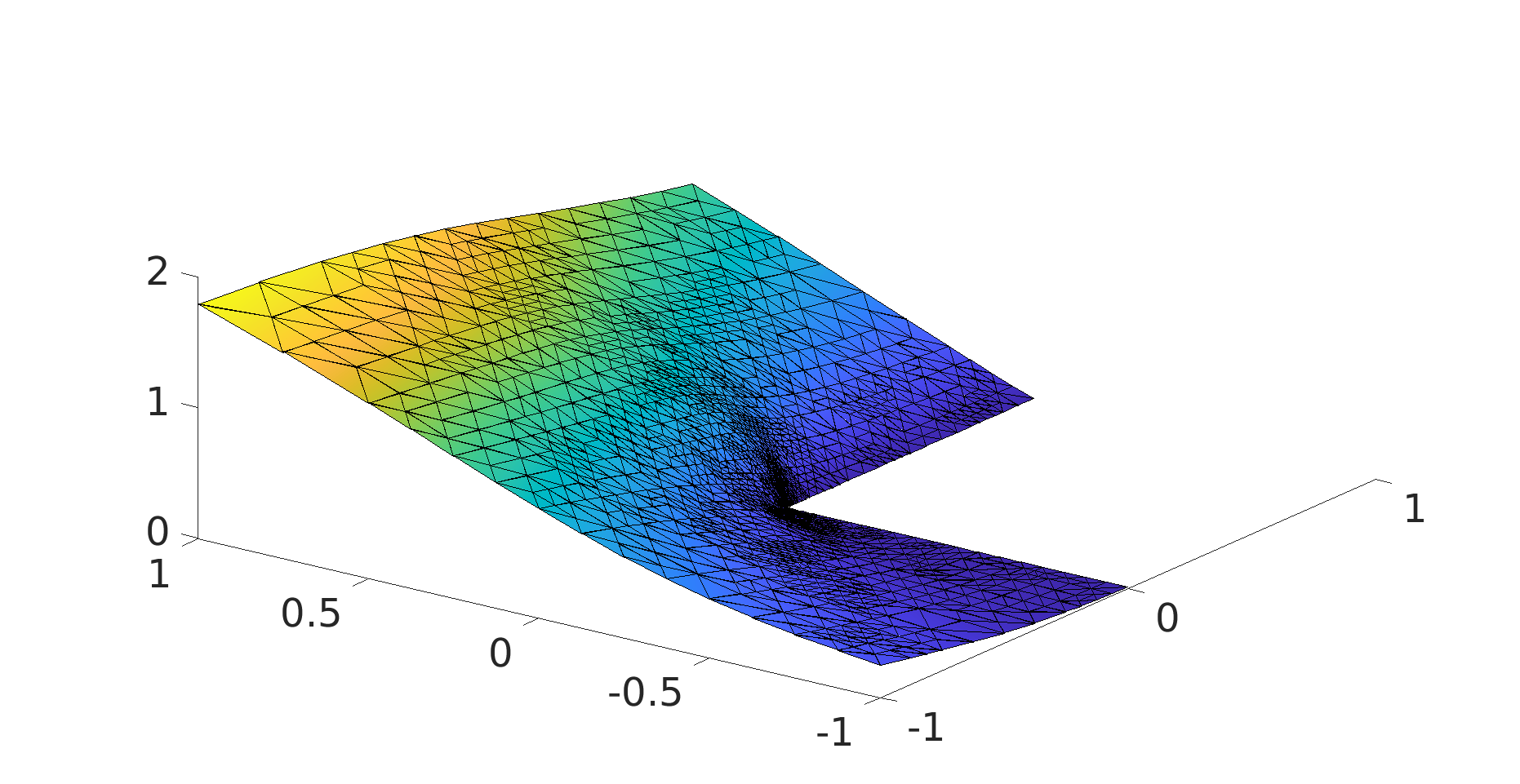}
		{\includegraphics[width=0.3\linewidth]{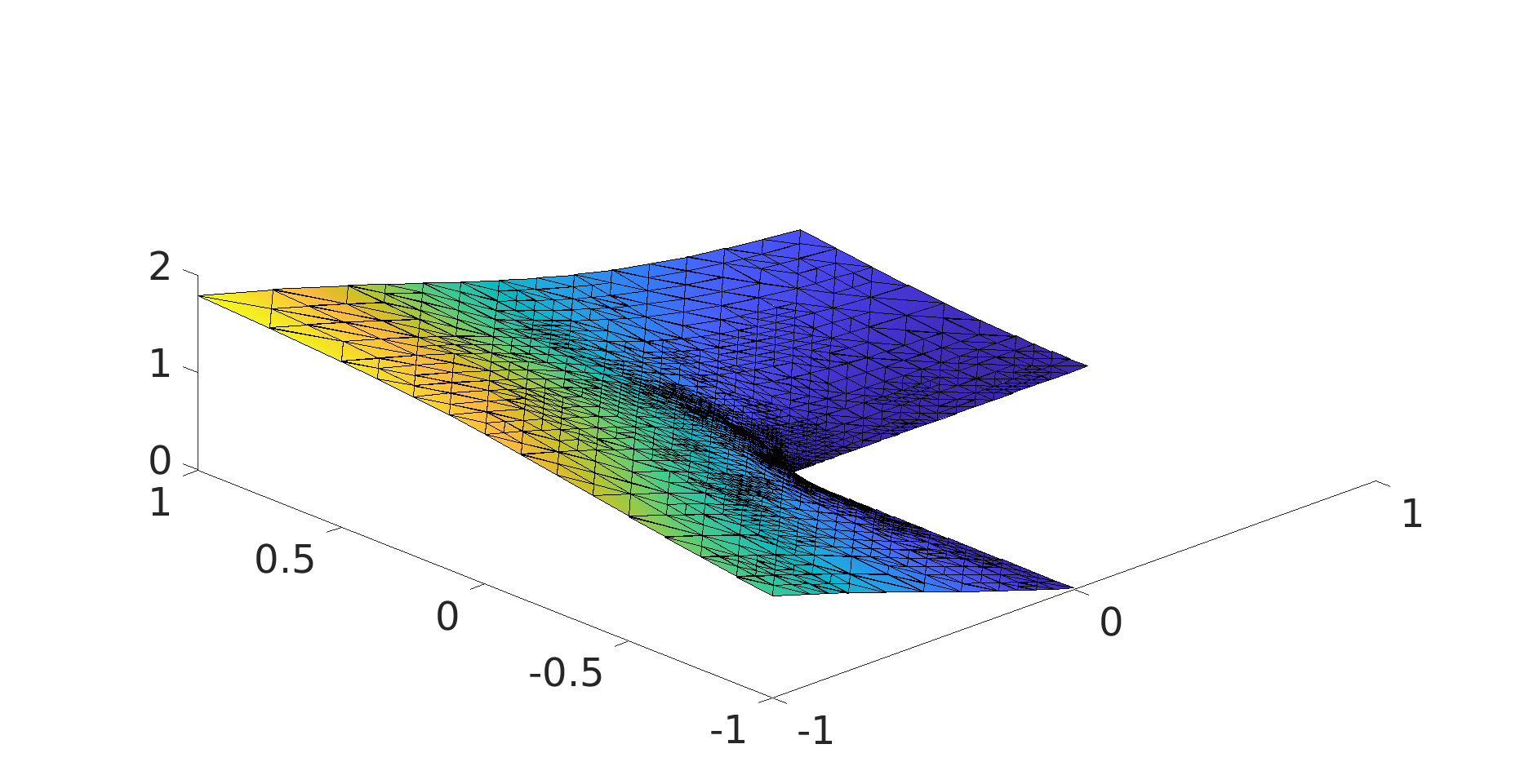}}}
	\subcaptionbox{Pressure profile of state equation}{\includegraphics[width=0.3\linewidth]{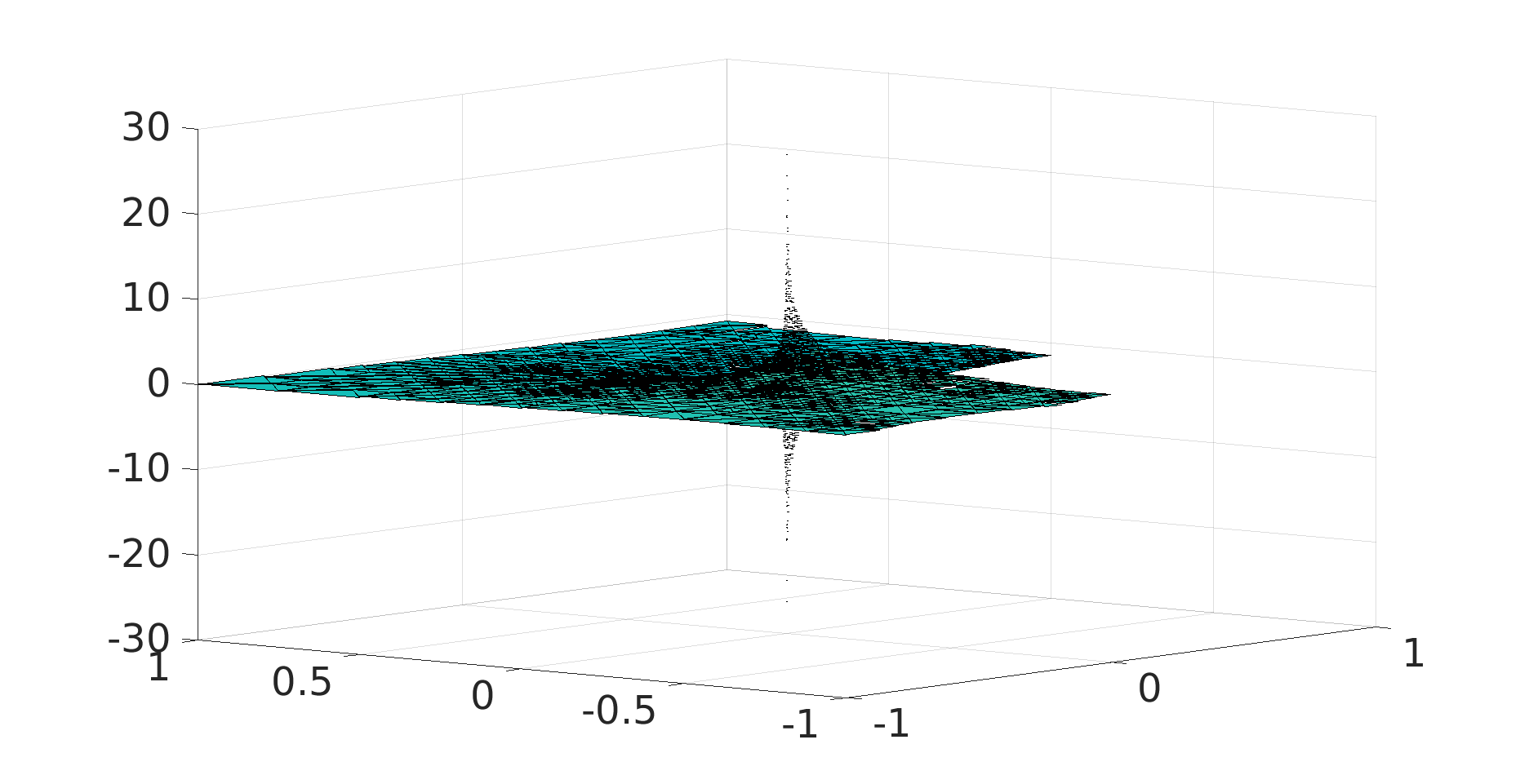}}
	\caption{Velocity and pressure profiles of state equation in variational approach for example \ref{nonconvex domain}.}
    \label{figure 6}
\end{figure}


\begin{figure}[hbt!]
	\centering	\includegraphics[width=0.45\linewidth]{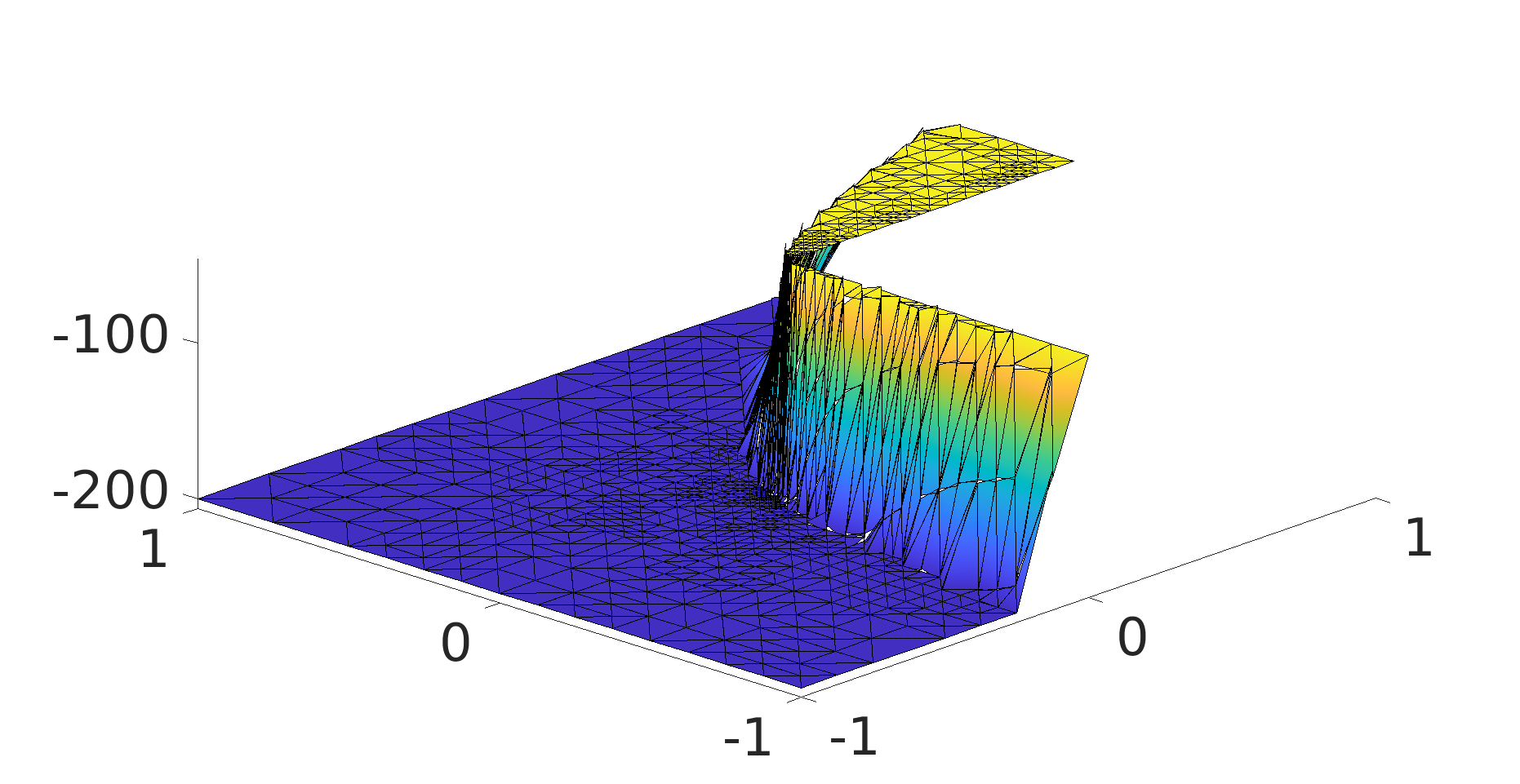}
	\includegraphics[width=0.45\linewidth]{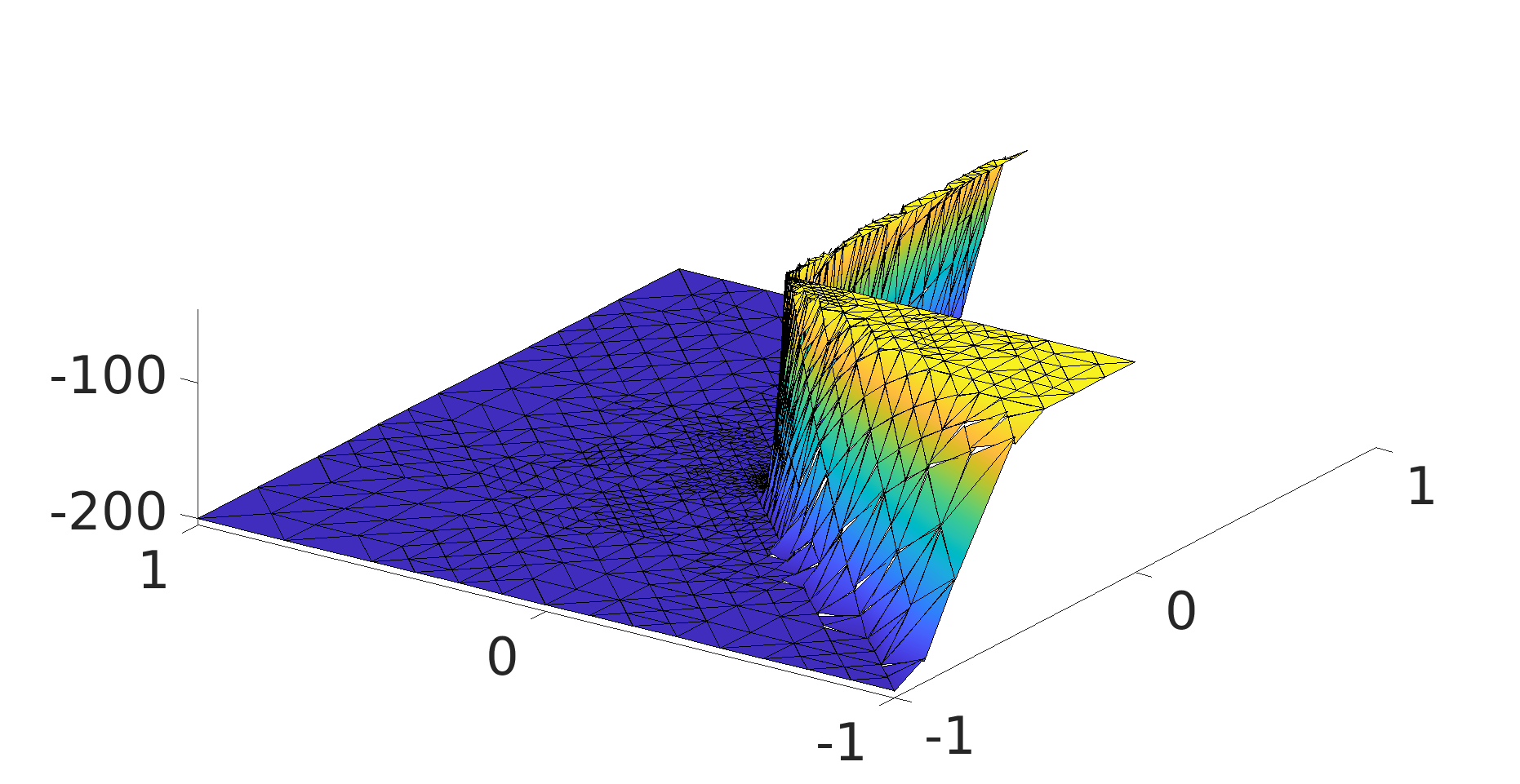}
	\caption{Variational control profiles for example \ref{nonconvex domain}.}
    \label{figure 7}
\end{figure}

\noindent Recall from \eqref{defuh}, the approximate control is computed using the formula $\bar{\U}_h(x)=\Pi_{[\U_a,\U_b]}\left(-\frac{1}{\alpha}\bar{\pee}_h(x)\right)$ for variational approach and $\bar{\U}_h(x)=\Pi_{[\U_a,\U_b]}\Pi_{0}\left(-\frac{1}{\alpha}\bar{\pee}_h(x)\right)$ for discretised approach. Figure \ref{figure 6} shows that the velocity variable ranges from 0 to 2. Given the selected value of $\alpha=10^{-3}$ the resulting control variable lies in the interval [-2000,0]. With the specified bounds $\U_a=(-200,-200)^{T}$, and $\U_b=(-50,-50)^{T}$, the projection of the control variable is confined to this range, causing the control surface to flatten near these limits. This behavior is clearly visible in Figures \ref{figure 7} and \ref{figure 8} for both the variational and discretised control formulations, respectively.

\begin{figure}[hbt!]
	\centering
	\includegraphics[width=0.45\linewidth]{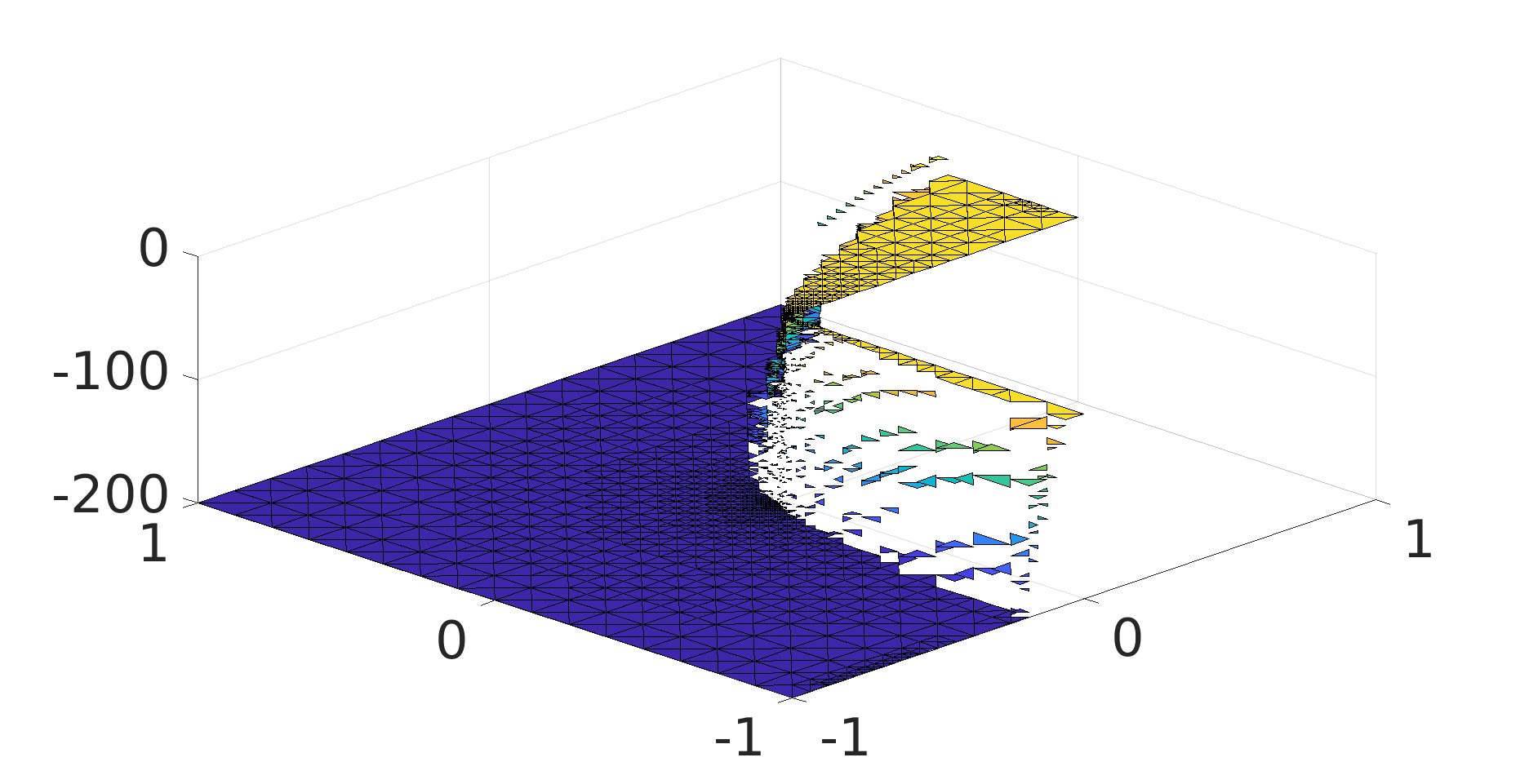}
	\includegraphics[width=0.45\linewidth]{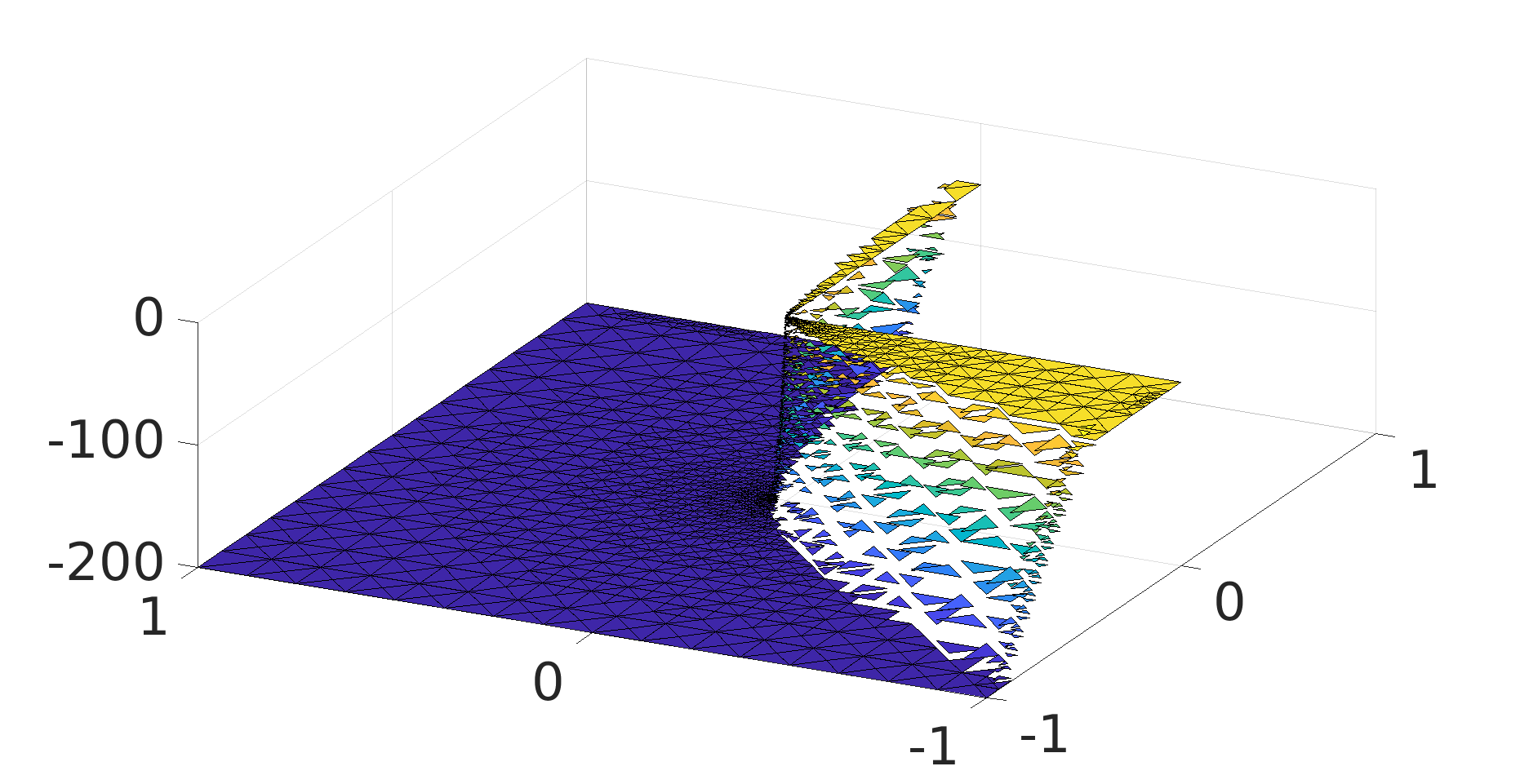}
	\caption{Discrete control profiles for Example \ref{nonconvex domain}}
    \label{figure 8}
\end{figure}


\section{Conclusion}
\label{conclusions}
This work investigates the convergence and quasi-optimality of adaptive finite element methods for an optimal control problem governed by the Stokes equations. The analysis is carried out within an axiomatic framework for both variational and discretised control formulations.

A key contribution of this study is the establishment of error equivalence results through the introduction of suitable auxiliary variables at both the continuous and discrete levels. In the discretised control formulation, the control variable is approximated using piecewise constant functions. Since the discrete control does not belong to the admissible control set ${\bf U}_{\mathrm{ad}}$, an auxiliary control variable $\widetilde{\mathbf{u}}$ is introduced (see~\eqref{2.21}). The error between $\widetilde{\mathbf{u}}$ and $\bar{\mathbf{u}}_h$ naturally leads to an additional control estimator $\eta_C(\mathcal{T})$, which is shown to be both reliable and efficient.

Furthermore, besides the control estimator proposed in this work, two alternative choices may also be considered. The first option is to directly include the term $\|\widetilde{\mathbf{u}}-\bar{\mathbf{u}}_h\|$ in the estimator, since $\widetilde{\mathbf{u}}$ is computable at the discrete level. However, as this choice does not explicitly involve the mesh-size parameter, establishing the estimator reduction property {\bf (A2)} becomes technically challenging. Another possible approach is to define the control estimator only over the inactive set $\mathcal{T}_{\mathrm{in}}:=\mathcal{T}\setminus\mathcal{T}_{\mathrm{a}}$, where $\mathcal{T}_{\mathrm{a}}$ denotes the active region. Nevertheless, due to the difficulty in relating the inactive sets across successive meshes, verifying the axioms {\bf (A1)} and {\bf (A2)} in this setting is nontrivial. Similar challenges are well known in the analysis of obstacle problems, particularly in the treatment of edge estimators on noncontact regions \cite[Remark 3.1]{MR3365763}. To overcome these issues, the control estimator in this work is defined over the entire domain $\Omega$.
Numerical experiments in the convex and nonconvex domains demonstrate the effectiveness of the resulting complete estimator.
\section*{Conflict of interest}
The authors declared that they have no conflict of interest.
\section{Acknowledgments}
\label{acknowledgements}
\noindent Asha K. Dond acknowledges the financial support of ANRF under Grant No. ANRF/ARG/009285. The authors express their gratitude to Mr. Subham Nayak for his valuable contributions and insights during the discussions.

\bibliographystyle{abbrv} 
\bibliography{vKeBib}


\section{Appendix}
\label{appendix}
\subsubsection{Proof of Theorem \ref{d error equivalence}{\it (i)-(ii)}}
\label{(2.2)(i)-(ii)a}
\begin{proof}
{\bf Step-1 : (Primary estimates)}
Choose ${\bf v}=\bar{\y}-\widetilde{\bf y}, q=\bar{r}-\widetilde{r}$ in \eqref{2.11}, ${\bf v}=\bar{\y}-\widehat{\y}, q=\bar{r}-\widehat{r}$ in \eqref{2.24}, and ${\bf v}=\bar{\pee}-\widetilde{\pee}, q=\bar{s}-\widetilde{s}$ in \eqref{2.12} along with ellipticity of bilinear form $a$ to get
\begin{equation}
	\label{3.17a}
	\enorm{\bar{\y}-\widetilde{\y}}_{\rm pw}\leq C_{\rm P}\|\bar{{\bf u}}-\bar{{\bf u}}_h\|, \enorm{\bar{\y}-\widehat{\y}}_{\rm pw}\leq C_{\rm P}\|\bar{{\bf u}}-\widetilde{{\bf u}}\|,\mbox{ and }\enorm{\bar{\pee}-\widetilde{\pee}}_{\rm pw}\leq C_{\rm P}\|\bar{{\bf y}}-\bar{{\bf y}}_h\|. 
\end{equation}
Lemma \ref{Poincare}{\it(i)} and \eqref{3.17a} give
\begin{equation}
	\label{3.18a}
	\|\bar{\y}-\widetilde{\y}\|\leq C_{\rm P}^2\|\bar{{\bf u}}-\bar{{\bf u}}_h\|, \|\bar{\y}-\widehat{\y}\|\leq C_{\rm P}^2\|\bar{{\bf u}}-\widetilde{{\bf u}}\|,\mbox{ and }\|\bar{\pee}-\widetilde{\pee}\|\leq C_{\rm P}^2\|\bar{{\bf y}}-\bar{{\bf y}}_h\|. 
\end{equation}

{\bf Step-2 : (Control based bounds for the state and adjoint errors)}
Application of triangle and \eqref{3.17a}-\eqref{3.18a} give
\begin{equation}
	\label{3.19a}
	\wsterr\leq C_{16}(\wcerr+\wauxsterr)\mbox{ and }\wlsterr\leq C_{17}(\wcerr+\wlauxsterr)
\end{equation}
with $C_{16}:=\max\{1,C_{\rm P}\}$ and $C_{17}:=\max\{1,C_1\}$. Triangle inequality, \eqref{3.17a}, and Lemma \ref{discrete sobolev}{\it(i)} give
\begin{equation}
	\label{3.20a}
	\waderr\leq C_{18}(\wcerr+\wauxsterr+\wauxaderr)\mbox{ with }C_{18}:=\max\{1,C_{\rm P}^2C_{\rm PJ}\}.
\end{equation}
Again application of triangle inequality and \eqref{3.18a} leads to
\begin{equation}
	\label{3.21a}
	\wladerr\leq C_{19}(\wcerr+\wlauxsterr+\wlauxaderr)\mbox{ with }C_{19}:=\max\{1,C_{\rm P}^4\}.
\end{equation}
 For the bounds on terms $\wtum$ and $\wmai$ one can refer to \eqref{3.6} and \eqref{3.7} respectively.\vspace{0.1cm} \\
{\bf Step-3 : (Upper bound on the control)}
Definition of $\widetilde{\bf u}$ from \eqref{2.21}, \cite[Theorem 7.1.2]{kesavann}, and \eqref{cop} result in
\begin{equation*}
	(\alpha\widetilde{\bf u}+\bar{\pee}_h, {\bf z}-\widetilde{\bf u})\geq 0\mbox{ and }(\alpha\bar{\bf u}+\bar{\pee},{\bf z}-\bar{\bf u})\geq 0\mbox{ for all }{\bf z}\in {\bf U}_{\rm ad}. 
\end{equation*}
Choice of  ${\bf z}=\bar{\bf u}\in {\bf U}_{\rm ad}\mbox{ and }{\bf z}=\widetilde{\bf u}\in {\bf U}_{\rm ad}$ in the above mentioned equation, respectively with some algebraic calculations give
\begin{equation*}
	\alpha\|\bar{\bf u}-\widetilde{\bf u}\|^2\leq (\bar{\bf p}-\bar{\bf p}_h, \widetilde{\bf u}-\bar{\bf u})=(\widetilde{\bf p}-\bar{\bf p},\bar{\bf u}-\widetilde{\bf u})+(\widetilde{\bf p}-\bar{\bf p}_h,\widetilde{\bf u}-\bar{\bf u}). 
\end{equation*}
Choice of ${\bf v}=\widetilde{\pee}-\bar{\bf p}$ in \eqref{2.24} and $q=\bar{r}-\widehat{r}$ in \eqref{2.12} give
\begin{equation*}
	\alpha\|\bar{\bf u}-\widetilde{\bf u}\|^2\leq a(\bar{\bf y}-\widehat{\bf y},\widetilde{\bf p}-\bar{\bf p})+(\widetilde{\bf p}-\bar{\bf p}_h,\widetilde{\bf u}-\bar{\bf u})=a(\widetilde{\bf p}-\bar{\bf p},\bar{\bf y}-\widehat{\bf y})+(\widetilde{\bf p}-\bar{\bf p}_h,\widetilde{\bf u}-\bar{\bf u}). 
\end{equation*}
Again choose ${\bf v}=\widehat{\bf y}-\bar{\bf y}$ in \eqref{2.12} and $q=\bar{s}-\widetilde{s}$ in \eqref{2.24} to get
\begin{equation*}
	\alpha\|\bar{\bf u}-\widetilde{\bf u}\|^2\leq(\bar{\bf y}-\bar{\bf y}_h,\widehat{\bf y}-\bar{\bf y})+(\widetilde{\bf p}-\bar{\bf p}_h,\widetilde{\bf u}-\bar{\bf u}).
\end{equation*}
Some algebraic calculations yield
\begin{equation*}
	\alpha\|\bar{\bf u}-\widetilde{\bf u}\|^2\leq (\bar{\bf y}-\widehat{\bf y},\widehat{\bf y}-\bar{\bf y}_h)+\|\widehat{\bf y}-\bar{\bf y}_h\|^2+(\widetilde{\bf p}-\bar{\bf p}_h,\widetilde{\bf u}-\bar{\bf u}).
\end{equation*}
Application of Cauchy Schwartz inequality, triangle inequality, and \eqref{3.18a} results in
\begin{equation*}
 \alpha\|\bar{\bf u}-\widetilde{\bf u}\|^2\leq C_{\rm P}^2\|\bar{\bf u}-\widetilde{\bf u}\|(\|\widehat{\y}-\widetilde{\y}\|+\|\widetilde{\y}-\yhb\|)+2(\|\widehat{\y}-\widetilde{\y}\|^2+\|\widetilde{\y}-\yhb\|^2)+\|\widetilde{\bf p}-\bar{\bf p}_h\|\|\widetilde{\bf u}-\bar{\bf u}\|.
\end{equation*}
Choice of ${\bf v}=\widehat{\y}-\widetilde{\y}$ and $q=\widehat{r}-\widetilde{r}$ in \eqref{332} along with ellipticity of bilinear form $a$ leads to
\begin{equation*}
	\alpha\|\bar{\bf u}-\widetilde{\bf u}\|^2\leq C_{\rm P}^2\|\bar{\bf u}-\widetilde{\bf u}\|\|\widetilde{\bf u}-\bar{\bf u}_h\|+C_{\rm P}^2\|\widetilde{\bf y}-\bar{\bf y}_h\|\|\bar{\bf u}-\widetilde{\bf u}\|+2\|\widetilde{\bf u}-\bar{\bf u}_h\|^2+2\|\widetilde{\bf y}-\bar{\bf y}_h\|^2+\|\widetilde{\bf p}-\bar{\bf p}_h\|\|\widetilde{\bf u}-\bar{\bf u}\|. 
\end{equation*}
Applying Young’s inequality with $a=C_{\rm P}^2\|\bar{\bf u}-\widetilde{\bf u}\|$, $b=\|\bar{\bf u}-\widetilde{\bf u}\|$, and $\epsilon={3C_{\rm P}^4}/{\alpha}$, with $a=C_{\rm P}^2\|\widetilde{\bf y}-\bar{\bf y}_h\|$, $b=\|\bar{\bf u}-\widetilde{\bf u}\|$, and $\epsilon={\alpha}/{3}$, and with $a=\|\widetilde{\bf p}-\bar{\bf p}_h\|$, $b=\|\widetilde{\bf u}-\bar{\bf u}\|$, and $\epsilon={\alpha}/{3}$ in the above equations, together with some algebraic manipulations, yields
\begin{equation}
	\label{3.22a}
	\|\bar{\bf u}-\widetilde{\bf u}\|\leq C_{20}(\|\widetilde{\bf u}-\bar{\bf u}_h\|+\|\widetilde{\bf y}-\bar{\bf y}_h\|+\|\widetilde{\bf p}-\bar{\bf p}_h\|)
\end{equation}
with $C_{20}:=\max\{\sqrt{\frac{2}{\alpha}(\frac{3C_{\rm P}^4}{\alpha}+2)},\frac{\sqrt{3}}{\alpha}\}$. Triangle inequality, \eqref{3.22}, and Lemma \ref{discrete sobolev}{\it (i)} give
\begin{equation}
	\label{3.23a}
	\wcerr\leq C_{21}(\normone+\wauxsterr+\wauxaderr)\mbox{ with }C_{21}:=C_{\rm PJ}(1+C_{20}).
\end{equation}

{\bf Step-4 : (Complete Upper and lower bounds)}
Use of \eqref{3.23a}, \eqref{3.19a}-\eqref{3.20a}, and \eqref{3.6}-\eqref{3.7} give
\begin{align*}
	\|E_{\rm C}\|+\enorm{E_{\rm S}}_{\rm pw}+\enorm{E_{\rm A}}_{\rm pw}\leq C_{\rm d1}(\|\widetilde{E}_{\rm C}\|+\enorm{\widetilde{E}_{\rm S}}_{\rm pw}+\enorm{\widetilde{E}_{\rm A}}_{\rm pw})
\end{align*}
with $C_{\rm d1}:=C_{21}(1+C_{16}+C_{18}+C_{4}+C_{5})$. Again use of \eqref{3.19a}, \eqref{3.21a}, \eqref{3.6}-\eqref{3.7}, and \eqref{3.23a} leads to
\begin{align*}
	\|E_{\rm C}\|+\|{E_{\rm S}}\|+\|{E_{\rm A}}\|\leq C_{\rm d2}(\|\widetilde{E}_{\rm C}\|+\|{\widetilde{E}_{\rm S}}\|+\widetilde{E}_{\rm A}\|)\mbox{ with }C_{\rm d2}:=C_{21}(1+C_{17}+C_{19}+C_4+C_6).
\end{align*}
 Applying triangle inequality, by introducing $(\bar{\bf y},\bar{r},\bar{\bf p},\bar{s},\bar{\bf u})$, Lemma \ref{discrete sobolev}{\it (i)}, \eqref{defu}, \eqref{2.21}, \eqref{cst}, \eqref{cad}, and Lipschitz continuity of projection operator, $\Pi_{[{\bf u}_a,{\bf u}_b]}$,(with Lipschitz constant $\alpha^{-1}$) gives
\begin{align*}
	\|\widetilde{E}_{\rm C}\|+\enorm{\widetilde{E}_{\rm S}}_{\rm pw}+\enorm{\widetilde{E}_{\rm A}}_{\rm pw}\leq C_{\rm d3}(\|E_{\rm C}\|+\enorm{E_{\rm S}}_{\rm pw}+\enorm{E_{\rm A}}_{\rm pw})
\end{align*}
with $C_{\rm d3}:=\max\{1, C_{\rm cst}, \alpha^{-1}C_{\rm PJ},C_{\rm cad}C_{\rm PJ}\}$ and
\begin{align*}
	\|\widetilde{E}_{\rm C}\|+\|{\widetilde{E}_{\rm S}}\|+\widetilde{E}_{\rm A}\|\leq C_{\rm d4}(\|E_{\rm C}\|+\|{E_{\rm S}}\|+\|{E_{\rm A}}\|)\mbox{ with }C_{\rm d4}:=\max\{1,C_{\rm cst},C_{\rm cad},\alpha^{-1}\}.
\end{align*}
\end{proof}
\subsubsection{Proof of Theorem \ref{d error equivalence}{\it (iii)-(iv)}}

\begin{proof}
{\bf Step-1 : (Primary estimates)}
Choose $\widehat{{\bf v}}_h=\widehat{\bar{\y}}_h-\widehat{\widetilde{\bf y}}_h, \widehat{{q}}_h={\widehat{\bar{r}}}_h-\widehat{\widetilde{r}}_h$ in \eqref{2.17}, $\widehat{{\bf v}}_h=\widehat{\bar{\y}}_h-\widehat{\y}_h, \widehat{q}_h=\widehat{\bar{r}}_h-\widehat{r}_h$ in \eqref{2.25}, and $\widehat{{\bf v}}_h=\widehat{\bar{\pee}}_h-\widehat{\widetilde{\pee}}_h, \widehat{{q}}_h=\widehat{\bar{s}}_h-\widehat{\widetilde{s}}_h$ in \eqref{2.18} to get
\begin{equation}
	\label{3.24a}
	\enorm{\widehat{\bar{\y}}_h-\widehat{\widetilde{\y}}_h}_{\rm pw}\leq C_{\rm dP}\|\widehat{\bar{{\bf u}}}_h-\bar{{\bf u}}_h\|, \enorm{\widehat{\bar{\y}}_h-\widehat{\y}_h}_{\rm pw}\leq C_{\rm dP}\|\widehat{\bar{{\bf u}}}_h-\widetilde{{\bf u}}_h\|,\mbox{ and }\enorm{\widehat{\bar{\pee}}_h-\widehat{\widetilde{\pee}}_h}_{\rm pw}\leq C_{\rm dP}\|\widehat{\bar{{\bf y}}}_h-\bar{{\bf y}}_h\|. 
\end{equation}

\noindent Application of Lemma \ref{discrete sobolev}{\it (ii)} and \eqref{3.24a} give
\begin{equation}
	\label{3.25a}
	\|\widehat{\bar{\y}}_h-\widehat{\widetilde{\y}}_h\|\leq C_{\rm dP}^2\|\widehat{\bar{{\bf u}}}_h-\bar{{\bf u}}_h\|, \|\widehat{\bar{\y}}_h-\widehat{\y}_h\|\leq C_{\rm dP}^2\|\widehat{\bar{{\bf u}}}_h-\widetilde{{\bf u}}_h\|,\mbox{ and }\|\widehat{\bar{\pee}}_h-\widehat{\widetilde{\pee}}_h\|\leq C_{\rm dP}^2\|\widehat{\bar{{\bf y}}}_h-\bar{{\bf y}}_h\|. 
\end{equation}

{\bf Step-2 : (Control based bounds for the state and adjoint errors)}

Applying triangle inequality, \eqref{3.24a}, and \eqref{3.25a} results in
\begin{equation}
	\label{3.26a}
	\2\leq C_{22}(\1+\6)\mbox{ and }\w\leq C_{23}(\1+\|{\widehat{\widetilde{\bf y}}}_h-\bar{\bf y}_h\|)
\end{equation}
with $C_{22}:=\max\{1,C_{\rm dP}\}$ and $C_{23}:=\max\{1,C_8\}$. Multiple use of triangle inequality, \eqref{3.24a}, and Lemma \ref{discrete sobolev}{\it (ii)} leads to
\begin{equation}
	\label{3.27a}
	\3\leq C_{24}(\1+\6+\7)\mbox{ with }C_{24}:=\max\{1,C_{\rm dP}^2C_{\rm PI}\}.
\end{equation}
Again application of triangle inequality and \eqref{3.25a} yield
\begin{equation}
	\label{3.28a}
	\x\leq C_{25}(\1+\q+\z)\mbox{ with }C_{25}:=\max\{1,C_{\rm P}^2,C_{\rm P}^2C_{\rm dP}^2\}. 
\end{equation}
 For bounds on terms $\|{\widehat{\bar r}}_h-\rhb\|$ and $\|{\widehat{\bar s}}_h-\shbar\|$ one can refer to \eqref{3.14}-\eqref{s}.

{\bf Step-3 : (Upper bound on the control)}
Use the definition of $\widetilde{\bf u}_h$ from \eqref{2.21}, \cite[Theorem-7.1.3]{kesavann}, and \eqref{dop}  ($\widehat{\T}$-level) to get
\begin{equation*}
	(\alpha\widetilde{\bf u}_h+\bar{\bf p}_h,\widehat{{\bf z}}_h-\widetilde{\bf u}_h)\geq 0\mbox{ and }(\alpha{\widehat{\bar{\bf u}}}_h+{\widehat{\bar{\pee}}}_h,\widehat{{\bf z}}_h-{\widehat{\bar{\bf u}}}_h)\geq 0\mbox{ for all }\widehat{{\bf z}}_h\in \widehat{\bf U}_{h,{\rm ad}}.
\end{equation*}
Choice of $\widehat{{\bf z}}_h=\widehat{\bar{\bf u}}_h\in \widehat{\bf U}_{h,{\rm ad}}$ and $\widehat{{\bf z}}_h=\widetilde{\bf u}_h\in \widehat{\bf U}_{h,{\rm ad}}$ in above mentioned inequalities along with some algebraic calculations give
\begin{equation*}
	\alpha\|{\widehat{\bar{\bf u}}}_h-\widetilde{\bf u}_h\|^2\leq ({\widehat{\bar{\bf p}}}_h-\bar{\bf p}_h,\widetilde{\bf u}_h-{\widehat{\bar{\bf u}}}_h)=(\widehat{\widetilde{\bf p}}_h-{\widehat{\bar{\bf p}}}_h,{\widehat{\bar{\bf u}}}_h-\widetilde{\bf u}_h)+(\widehat{\widetilde{\pee}}_h-\bar{\bf p}_h,\widetilde{\bf u}_h-{\widehat{\bar{\bf u}}}_h).
\end{equation*}
Choose $\widehat{{\bf v}}_h=\widehat{\widetilde{\bf p}}_h-\widehat{\bar{\bf p}}_h$ in equation \eqref{2.25} and $\widehat{{q}}_h=\widehat{\bar{r}}_h-\widehat{r}_h$ in \eqref{2.18} to get
\begin{equation*}
	\alpha\|{\widehat{\bar{\bf u}}}_h-\widetilde{\bf u}_h\|^2\leq a({\widehat{\bar{\bf y}}}_h-\widehat{\bf y}_h,\widehat{\widetilde{\bf p}}_h-{\widehat{\bar{\bf p}}}_h)+(\widehat{\widetilde{\pee}}_h-\bar{\bf p}_h,\widetilde{\bf u}_h-{\widehat{\bar{\bf u}}}_h). 
\end{equation*}
Choice of $\widehat{{\bf v}}_h=\widehat{\bf y}_h-{\widehat{\bar{\bf y}}}_h$ in \eqref{2.18} and $q_h=\widehat{\bar{s}}_h-\widehat{\widetilde{s}}_h$ in \eqref{2.25} leads to
\begin{equation*}
	\alpha\|{\widehat{\bar{\bf u}}}_h-\widetilde{\bf u}_h\|^2\leq({\widehat{\bar{\bf y}}}_h-\bar{\bf y}_h,\widehat{\bf y}_h-{\widehat{\bar{\bf y}}}_h)+(\widehat{\widetilde{\pee}}_h-\bar{\bf p}_h,\widetilde{\bf u}_h-{\widehat{\bar{\bf u}}}_h).
\end{equation*}
Some algebraic calculations give
\begin{equation*}
    \alpha\|{\widehat{\bar{\bf u}}}_h-\widetilde{\bf u}_h\|^2\leq (\widehat{\bar{\bf y}}_h-\widehat{\bf y}_h,\widehat{\bf y}_h-{\bar{\bf y}}_h)+\|\widehat{\bf y}_h-\yhb\|^2+(\widehat{\widetilde{\pee}}_h-\bar{\bf p}_h,\widetilde{\bf u}_h-{\widehat{\bar{\bf u}}}_h).
\end{equation*}
Applying Cauchy Schwartz inequality, triangle inequality, and \eqref{3.25a} results in
\begin{equation*}
    \alpha\|{\widehat{\bar{\bf u}}}_h-\widetilde{\bf u}_h\|^2\leq C_{\rm dP}^2\|\widehat{\bar{\bf u}}_h-\widetilde{\bf u}_h\|(\|\widehat{\bf y}_h-\widehat{\widetilde{\bf y}}_h\|+\|\widehat{\widetilde{\bf y}}_h-\yhb\|)+2(\|\widehat{\bf y}_h-\widehat{\widetilde{\bf y}}_h\|^2+\|\widehat{\widetilde{\bf y}}_h-\yhb\|^2)+(\widehat{\widetilde{\pee}}_h-\bar{\bf p}_h,\widetilde{\bf u}_h-{\widehat{\bar{\bf u}}}_h).
\end{equation*}
The choice of $\widehat{\bf v}_h=\widehat{\bf y}_h-\widehat{\widetilde{\bf y}}_h$ and $\widehat{q}_h=\widehat{r}_h-\widehat{\widetilde{r}}_h$ in \eqref{333} together with the ellipticity of bilinear form $a$ give
\begin{equation*}
    \alpha\|{\widehat{\bar{\bf u}}}_h-\widetilde{\bf u}_h\|^2\leq C_{\rm dP}^2\|\widehat{\bar{\bf u}}_h-\widetilde{\bf u}_h\|(\|\widetilde{\bf u}_h-\uhb\|+\|\widehat{\widetilde{\bf y}}_h-\yhb\|)+2(\|\widetilde{\bf u}_h-\uhb\|^2+\|\widehat{\widetilde{\bf y}}_h-\yhb\|^2)+(\widehat{\widetilde{\pee}}_h-\bar{\bf p}_h,\widetilde{\bf u}_h-{\widehat{\bar{\bf u}}}_h).
\end{equation*}
Applying Young’s inequality with $a=C_{\rm dP}^2\|{\widehat{\bar{\bf u}}}_h-\widetilde{\bf u}_h\|$, $b=\normtwo$, and $\epsilon={3C_{\rm dP}^4}/{\alpha}$, with $a=C_{\rm dP}^2\|{\widehat{\bar{\bf u}}}_h-\widetilde{\bf u}_h\|$, $b=\q$, and $\epsilon={3C_{\rm dP}^4}/{\alpha}$, and with $a=\z$, $b=\|{\widehat{\bar{\bf u}}}_h-\widetilde{\bf u}_h\|$, and $\epsilon={\alpha}/{3}$ in the above equation, together with some routine calculations, yields
\begin{equation}
	\label{3.29a}
	\|{\widehat{\bar{\bf u}}}_h-\widetilde{\bf u}_h\|\leq C_{26}(\normtwo+\q+\z). 
\end{equation}
Here $C_{26}:=\max\{\sqrt{\frac{2}{\alpha}(\frac{3C_{\rm dP}^4}{2\alpha})+2},\frac{\sqrt{3}}{\alpha}\}$. 
Use of triangle inequality, \eqref{3.29a}, and Lemma \ref{discrete sobolev}{\it (ii)} yields
\begin{align}	
    \label{3.30a}
	&\1\leq C_{27}(\normtwo+\6+\7)\mbox{ with }C_{27}:=C_{\rm PI}(1+C_{26}).
\end{align}
{\bf Step-4 : (Complete upper and lower bounds)}
Use \eqref{3.26a}-\eqref{3.27a}, \eqref{3.30a}, \eqref{3.14}, and \eqref{s} to get
\begin{align*}
	\|E_{{\rm C},h}\|+\enorm{E_{{\rm S},h}}_{\rm pw}+\enorm{E_{{\rm A},h}}_{\rm pw}\leq C_{\rm d5}(\|\widetilde{E}{_{\rm C},h}\|+\enorm{\widetilde{E}_{{\rm S},h}}_{\rm pw}+\enorm{\widetilde{E}_{{\rm A},h}}_{\rm pw})
\end{align*}
with $C_{\rm d5}:=\max\{1, C_{27}(C_{11}+C_{12}+C_{22})\}$. Again using \eqref{3.26a}, \eqref{3.28a}, \eqref{3.14}-\eqref{3.15}, and \eqref{3.29a} results in 
\begin{align*}
	\|E_{{\rm C},h}\|+\|E_{{\rm S},h}\|+\|E_{{\rm A},h}\|\leq C_{\rm d6}(\|\widetilde{E}_{_{\rm C},h}\|+\|\widetilde{E}_{{\rm S},h}\|+\|\widetilde{E}_{{\rm A},h}\|).
\end{align*}
Here $C_{\rm d6}:=\max\{1,C_{11}+C_{13}+C_{19}+C_{23}\}$.
Following analogous calculation as we have done in step-4 of Section \ref{(2.2)(i)-(ii)a}, by introducing $(\widehat{\bar{{\bf y}}}_h,\widehat{\bar{r}}_h,\widehat{\bar{\bf p}}_h,\widehat{\bar{s}}_h,\widehat{\bar{\bf u}}_h)$ we get
\begin{align*}
	\|\widetilde{E}_{{\rm C},h}\|+\enorm{\widetilde{E}_{{\rm S},h}}_{\rm pw}+\enorm{\widetilde{E}_{{\rm A},h}}_{\rm pw}\leq C_{\rm d7}(\|E_{{\rm C},h}\|+\enorm{E_{{\rm S},h}}_{\rm pw}+\enorm{E_{{\rm A},h}}_{\rm pw}),
\end{align*} 
with $C_{\rm d7}:=\max\{1,\alpha^{-1}C_{\rm PI},C_{\rm dad}C_{\rm dst},C_{\rm dst}\}$. Analogous calculation leads to
\begin{align*}
	\|\widetilde{E}_{{\rm C},h}\|+\|\widetilde{E}_{{\rm S},h}\|+\|\widetilde{E}_{{\rm A},h}\|\leq C_{\rm d8}(\|E_{{\rm C},h}\|+\|E_{{\rm S},h}\|+\|E_{{\rm A},h}\|)\mbox{ with }
\end{align*}
 $C_{\rm d8}:=\max\{1,\alpha^{-1}C_{\rm PI}, C_{\rm dad}C_{\rm dst}C_{\rm PI}, C_{\rm PI}C_{\rm dst}\}$. This completes proof of Theorem \ref{d error equivalence}. 
 \end{proof}
 
\end{document}